\documentclass[reqno,11pt]{amsart}
\usepackage[left=3cm, right=3cm, top=3cm, bottom=3cm]{geometry}
\usepackage{amssymb}
\usepackage{amsmath}
\usepackage{color}

\usepackage[T2A]{fontenc}
\usepackage{inputenc}
\inputencoding{cp1251}

\theoremstyle{plain}
\newtheorem{theorem}{Theorem}[section]
\newtheorem*{theorem*}{Theorem}
\newtheorem{proposition}{Proposition}[section]
\newtheorem{lemma}{Lemma}[section]
\newtheorem{corollary}{Corollary}[section]

\theoremstyle{definition}
\newtheorem{definition}{Definition}[section]

\theoremstyle{remark}
\newtheorem{remark}{Remark}[section]
\newtheorem{example}{Example}[section]

\begin{document}

\title[Analogs of generalized resolvents and eigenfunction
expansions]{Analogs of generalized resolvents and eigenfunction
expansions of relations generated by pair of differential operator
expressions one of which depends on spectral parameter in
nonlinear manner}

\author{Volodymyr Khrabustovskyi}

\address{Ukrainian State Academy of Railway Transport,  Kharkiv, Ukraine}

\email{v{\_}khrabustovskyi@ukr.net}


\keywords{Relation generated by pair of differential expressions
one of which depends on spectral parameter in nonlinear manner,
non-injective resolvent, generalized resolvent, eigenfunction
expansion}

\begin{abstract}For the relations generated by pair of differential operator
expressions one of which depends on the spectral parameter in the
Nevanlinna manner we construct analogs of the generalized
resolvents which are integro-differential operators. The
expansions in eigenfunctions of these relations are obtained
\end{abstract}

\maketitle

\section*{Introduction}
We consider either on finite or infinite interval operator
differential equation of arbitrary order
\begin{gather}
\label{GrindEQ__1_} l_\lambda[y]=m[f],\ t\in\bar{\mathcal{I}},\
\mathcal{I}=(a,b)\subseteq\mathbb{R}^1
\end{gather}
in the space of vector-functions with values in the separable
Hilbert space $\mathcal{H}$, where
\begin{gather}
\label{GrindEQ__2_} l_\lambda[y]=l[y]-\lambda m[y]-n_\lambda[y],
\end{gather}
$l[y],m[y]$ are symmetric operator differential expression. The
order of $l_\lambda[y]$ is equal to $r>0$. For the expression
$m[y]$ the subintegral quadratic form $m\{y,y\}$ of the Dirichlet
integral $m[y,y]=\int_{\mathcal{I}}m\{y,y\}dt$ is nonnegative for
$t\in\bar{\mathcal{I}}$. The leading coefficient of the expression
$m[y]$ may lack the inverse from $B(\mathcal{H})$ for any
$t\in\bar{\mathcal{I}}$ and even it may vanish on some intervals.
For the operator differential expression $n_\lambda[y]$ the form
$n_\lambda\{y,y\}$ depends on $\lambda$ in the Nevanlinna manner
for $t\in\bar{\mathcal{I}}$. Therefore the order $s\geq 0$ of
$m[y]$ is even and $\leq r$.

In the Hilbert space $L^2_m(\mathcal{I})$ with metrics generated
by the form $m[y,y]$ for equation
(\ref{GrindEQ__1_})-(\ref{GrindEQ__2_}) we construct analogs
$R(\lambda)$ of the generalized resolvents which in general are
non-injective and which possess the following representation:
\begin{gather}\label{GrindEQ__3_}
R(\lambda)=\int_{\mathbb{R}^1}{dE_\mu\over \mu-\lambda}
\end{gather}
where $E_\mu$ is a generalized spectral family for which
$E_\infty$ is less or equal to the identity operator. (Abstract
operators which possess such representation were studied in
\cite{DSnoo1}.)

This construction is based on a special reduction of the equation
\begin{gather}\label{GrindEQ__4_}
l[y]=m[f]
\end{gather}
to the first order system with weight. Here $l$ and $m$ are
operator differential expressions which are not necessary
symmetric (in contrast to (\ref{GrindEQ__2_})). For construction
of $R(\lambda)$ we also introduce the characteristic operator of
the equation
\begin{gather}\label{GrindEQ__5_}
l_\lambda[y]=-{(\Im l_\lambda)[f]\over \Im\lambda},\
t\in\bar{\mathcal{I}}.
\end{gather}

In the case $r=1$, $n_\lambda[y]=H_\lambda(t)y$ (here the
mentioned reduction is not needed) the resolvents $R(\lambda)$ was
constructed in \cite{Khrab4}.

Further in the work we consider the boundary value problem
obtained by adding to equation
(\ref{GrindEQ__1_})-(\ref{GrindEQ__2_}) the dissipative boundary
conditions depending on a spectral parameter. We prove that for
some boundary conditions solutions of such problems are generated
by the operators $R(\lambda)$ if, in contrast to the case $s=0$,
$n_\lambda[y]=H_\lambda(t)y$, the boundary conditions contain the
derivatives of vector-function $f(t)$ that are taken on the ends
of the interval.

In the work we calculate $E_\Delta$ and derive an inequality of
Bessel type. In the case when the expression $n_\lambda[y]$
submits in a special way to the expression $m[y]$ we obtain the
transformation formulae and the Parseval equality. The general
results obtained in the work are illustrated on the example of
equation (\ref{GrindEQ__1_}) with coefficients which are periodic
on the semi-axes. We remark that in the case $n_\lambda[y]\equiv
0$ it follows from \cite{Khrab1,Khrab2,Khrab3,Khrab6} that if
${\mathcal{I}}=\mathbb{R}^1$, $r>s$ and $\mathrm{dim}
\mathcal{H}<\infty$ then $E_\mu$ for equation (\ref{GrindEQ__1_})
with periodic coefficients on the axis have no jumps. (For $r= s$
in the described case $E_\mu$ may have jump (see, e.g,
\cite{Khrab6})). We show that in contrast to the case
$n_\lambda[y]\equiv 0$ if $r=1$, $\mathrm{dim} \mathcal{H}=2$ then
$E_\mu$ for equation (\ref{GrindEQ__1_}) with periodic
coefficients on the axis may have jump.

In the case $n_\lambda[y]\equiv 0$ the results listed above are
known \cite{Khrab6}, and $R(\lambda)$ is the generalized resolvent
of the minimal relation generated by the pair of expressions
$l[y]$ and $m[y]$. For this case we show in the work that in the
regular case all generalized resolvents are exhausted by the
operators $R(\lambda)$, and thereby by virtue of \cite{Khrab5}
their full description with the help of boundary conditions is
given. A review of other results for the case $n_\lambda[y]\equiv
0$ is in the work \cite{Khrab+}.

In the works \cite{Bruk4}, \cite{Bruk5} the question of the
conditions for holomorphy and continuous reversibility of the
restrictions of maximal relations generated by $l_\lambda[y]$
(\ref{GrindEQ__2_}) with $m[y]\equiv 0$,
$n_\lambda[y]=H_\lambda(t)y$ in $L^2_{\Im H_{\lambda_0}(t)/\Im
\lambda_0}$ ($\Im \lambda_0\not= 0$) and also by the integral
equation with the Nevanlinna matrix measure was studied (using
some of the results from \cite{Khrab5}). We remark that the
relations inverse to those ones considered in \cite{Bruk4},
\cite{Bruk5} do not possess the representation
(\ref{GrindEQ__3_}). Also we note that the resolvent equation
(\ref{GrindEQ__1_})-(\ref{GrindEQ__2_}) does not reduced to the
equations considered in \cite{Bruk4}, \cite{Bruk5}.

Many question, that concern differential operators and relations
in the space of vector-functions, are considered in the monographs
\cite{Atkinson,Berez1,Berez2,GorGor,LyaSto,Marchenko,RBKholkin,Sakhno}
containing an extensive literature. The method of studying of
these operators and relations based on use of the abstract Weyl
function and its generalization (Weyl family) was proposed in
\cite{DerMalamud,DHMdS1,DHMdS2}.

We denote by $(\ .\ )$ and $\|\cdot\|$ the scalar product and the
norm in various spaces with special indexes if it is necessary.

Let an interval $\Delta \subseteq \mathbb{R}^1 ,\,
f\left(t\right)\, \left(t\in \Delta \right)$ be a function with
values in some Banach space $B$. The notation $f\left(t\right)\in
C^{k} \left(\Delta,B\right),\ k=0,\, 1,\, ...$ (we omit the index
$k$ if $k=0$) means, that in any point of $\Delta$
$f\left(t\right)$ has continuous in the norm  $\left\| \, \cdot \,
\right\| _{B}$ derivatives of order up to and including $l$ that
are taken in the norm $\left\| \, \cdot \,\right\| _{B} $; if
$\Delta$ is either semi-open or closed interval then on its ends
belonging to $\Delta$ the one-side continuous derivatives exist.
The notation $f\left(t\right)\in C_{0}^{k} \left(\Delta,B\right)$
means that $f\left(t\right)\in C^{k} \left(\Delta,B \right)$ and
$f\left(t\right)=0$ in the neighbourhoods of the ends of $\Delta$.

\section{The reduction of equation (\ref{GrindEQ__4_}) to the
first order system of canonical type with weight. The Green
formula} We consider in the separable Hilbert space $\mathcal{H}$
equation \eqref{GrindEQ__4_}, where $l\left[y\right]$ and
$m\left[f\right]$ are differential expressions (that are not
necessary symmetric) with sufficiently smooth coefficients from
$B\left(\mathcal{H}\right)$ and of orders $r>0$ and $s$
correspondingly. Here $r\ge s\ge 0$, $s$ is even and these
expressions are presented in the divergent form. Namely:
\begin{equation} \label{GrindEQ__51+_}
l\left[y\right]=\sum\limits _{k=0}^{r}i^{k} l_{k} \left[y\right] ,
\end{equation}
where $l_{2j} =D^{j} p_{j} \left(t\right)D^{j} $, $l_{2j-1}
=\frac{1}{2} D^{j-1} \left\{Dq_{j} \left(t\right)+s_{j}
\left(t\right)D\right\}D^{j-1} $, $p_{j} \left(t\right)$, $q_{j}
\left(t\right),\, s_{j} \left(t\right)\in C^{j}
\left(\bar{\mathcal{I}},B\left(\mathcal{H}\right)\right)$, $D={d
\mathord{\left/{\vphantom{d dt}}\right.\kern-\nulldelimiterspace}
dt} $; $m\left[f\right]$ is defined in a similar way with $s$
instead of $r$ and $\tilde{p}_{j} \left(t\right),\, \,
\tilde{q}_{j} \left(t\right),\, \, \tilde{s}_{j} \left(t\right)\in
B\left(\mathcal{H}\right)$ instead of $p_{j} \left(t\right),\, \,
q_{j} \left(t\right),\, \, s_{j} \left(t\right)$.

In the case of even $r=2n\ge s$, $p_{n}^{-1} \in B\left(\mathcal{
H}\right)$ we denote
\begin{gather}\label{GrindEQ__6_}
Q\left(t,l\right)=\left(\begin{array}{cc} {0} & {iI_{n} } \\
{-iI_{n} } & {0} \end{array}\right)=\frac{J}{i} ,\, \, \,
S\left(t,l\right)=Q\left(t,l\right), \\ \label{GrindEQ__7_}
H\left(t,\, l\right)=\left\| h_{\alpha \beta } \right\| _{\alpha
,\, \beta =1}^{2} ,\, \, h_{\alpha \beta } \in
B\left(\mathcal{H}^n \right),
\end{gather}
where $I_{n} $ is the identity operator in $B\left(\mathcal{H}^{n}
\right);\, \, h_{11} $ is a three diagonal operator matrix whose
elements under the main diagonal are equal to $\left(\frac{i}{2}
q_{1} ,\, \ldots ,\, \frac{i}{2} q_{n-1} \right)$, the elements
over the main diagonal are equal to $\left(-\frac{i}{2} s_{1} ,\,
\, \ldots ,\, \, -\frac{i}{2} s_{n-1} \right)$, the elements on
the main diagonal are equal to $\left(-p_{0} ,\, \, \ldots ,\, \,
-p_{n-2} ,\, \, \frac{1}{4} s_{n} p_{n}^{-1} q_{n} -p_{n-1}
\right)$; $h_{12} $ is an operator matrix with the identity
operators $I_{1} $ under the main diagonal, the elements on the
main diagonal are equal to $\left(0,\, \, \ldots ,\, \, 0,\, \,
-\frac{i}{2} s_{n} p_{n}^{-1} \right)$, the rest elements are
equal to zero; $h_{21} $ is an operator matrix with identity
operators $I_{1} $ over the main diagonal, the elements on the
main diagonal are equal to $\left(0,\, \, \ldots ,\, \, 0,\, \,
\frac{i}{2} p_{n}^{-1} q_{n} \right)$, the rest elements are equal
to zero; $h_{22} =\mathrm{diag}\left(0,\, \, \ldots ,\, \, 0,\, \,
p_{n}^{-1} \right)$.

Also in this case we denote
\footnote{\label{foot1}$W\left(t,l,m\right)$ is given for the case
$s=2n$ . If  $s<2n$ one have set the corresponding elements of
operator matrices $m_{\alpha \beta } $ be equal to zero. In
particular if  $s<2n$ then  $m_{12} =m_{21} =m_{22} =0$  and
therefore $W\left(t,l,m\right)=\mathrm{diag}\left(m_{11}
,0\right)$ in view of \eqref{GrindEQ__13_}.}
\begin{equation} \label{GrindEQ__8_}
W\left(t,\, l,\, m\right)=C^{*-1} \left(t,l\right)\left\{\left\|
m_{\alpha \beta } \right\| _{\alpha ,\, \beta =1}^{2}
\right\}C^{-1} \left(t,l\right), m_{\alpha \beta } \in
B\left(\mathcal{H}^{n} \right),
\end{equation}
where $m_{11} $ is a tree diagonal operator matrix whose elements
under the main diagonal are equal to $\left(-\frac{i}{2}
\tilde{q}_{1} ,\, \ldots ,\, -\frac{i}{2} \tilde{q}_{n-1}
\right)$, the elements over the main diagonal are equal to
$\left(\frac{i}{2} \tilde{s}_{1} ,\, \ldots ,\, \frac{i}{2}
\tilde{s}_{n-1} \right)$, the elements on the main diagonal are
equal to $\left(\tilde{p}_{0} ,\, \ldots ,\, \tilde{p}_{n-1}
\right)$; $m_{12} =\mathrm{diag}\left(0,\, \ldots ,\, 0,\,
\frac{i}{2} \tilde{s}_{n} \right)$, $m_{21}
=\mathrm{diag}\left(0,\, \ldots ,\, 0,\, -\frac{i}{2}
\tilde{q}_{n} \right)$, $m_{22} =\mathrm{diag}\left(0,\, \ldots
,\, 0,\, \tilde{p}_{n} \right)$.

Operator matrix $C\left(t,l\right)$ is defined by the condition
\begin{multline}
\label{GrindEQ__9_} C\left(t,l\right)col\left\{f\left(t\right),\,
f'\left(t\right),\, \ldots ,\, f^{\left(n-1\right)}
\left(t\right),\, f^{\left(2n-1\right)} \left(t\right),\, \ldots
,\, f^{\left(n\right)} \left(t\right)\right\}=\\=
 col\, \left\{f^{\left[0\right]}
\left(t|l\right),\, f^{\left[1\right]} \left(t|l\right),\, \ldots
,\, f^{\left[n-1\right]} \left(t|l\right),\, f^{\left[2n-1\right]}
\left(t\left|l\right. \right),\, \ldots ,\, f^{\left[n\right]}
\left(t\left|l\right. \right)\right\},
\end{multline}
where $f^{\left[k\right]} \left(t\left|L\right. \right)$ are
quasi-derivatives of vector-function $f\left(t\right)$ that
correspond to differential expression $L$.

The quasi-derivatives corresponding to $l$ are equal (cf.
\cite{RBUpsala}) to
\begin{gather} \label{GrindEQ__10_}
y^{\left[j\right]} \left(t\left|l\right. \right)=y^{\left({
j}\right)} \left(t\right),\, \, \, { j}=0,\, \, ...,\, \,
\left[\frac{r}{2} \right]-1, \\ \label{GrindEQ__11_}
y^{\left[n\right]} \left(t\left|l\right.
\right)=\left\{\begin{array}{l} {p_{n} y^{\left(n\right)}
-\frac{i}{2} q_{n} y^{\left(n-1\right)} ,\, \, r=2n} \\
{-\frac{i}{2} q_{n+1} y^{\left(n\right)} ,\, \, r=2n+1}
\end{array}\right.,
\\\label{GrindEQ__12_}
y^{\left[r-j\right]} \left(t\left|l\right.
\right)=-Dy^{\left[r-j-1\right]} \left(t\left|l\right.
\right)+p_{j} y^{\left(j\right)} +\frac{i}{2} \left[s_{j+1}
y^{\left(j+1\right)} -q_{j} y^{\left(j-1\right)} \right],\, j=0,\,
...,\, \left[\frac{r-1}{2} \right],\, q_{0} \equiv 0.
\end{gather}
At that $l\left[y\right]=y^{\left[r\right]} \left(t\left|l\right.
\right)$. The quasi-derivatires $y^{\left[k\right]}
\left(t\left|m\right. \right)$ corresponding to $m$ are defined in
the same way with even $s$ instead of $r$ and $\tilde{p}_{j},
\tilde{q}_{j} ,\tilde{s}_{j} $ instead of $p_{j} ,q_{j} ,s_{j} $.

It is easy to see that
\begin{equation} \label{GrindEQ__13_}
C\left(t,l\right)=\left(\begin{array}{cc} {I_{n} } & {0} \\
{C_{21} } & {C_{22} } \end{array}\right),\, \, \, C_{\alpha\beta }
\in B\left(\mathcal{H}^{n} \right),
\end{equation}
$C_{21} ,\, C_{22} $ are upper triangular operator matrices with
diagonal elements $\left(-\frac{i}{2} q_{1} ,\, \ldots ,\,
-\frac{i}{2} q_{n} \right)$ and $\left(\left(-1\right)^{n-1} p_{n}
,\, \left(-1\right)^{n-2} p_{n} ,\, \ldots ,\, p_{n} \right)$
correspondingly.

In the case of odd $r=2n+1>s$ we denote
\begin{gather} \label{GrindEQ__14_}
Q\left(t,l\right)=\begin{cases}J/\oplus q_{n+1}\\
q_{1}\end{cases},\quad S\left(t,l\right)=\begin{cases} J/i\oplus
s_{n+1},& n>0 \\ s_{1},& n=0\end{cases},
\\
\label{GrindEQ__15_} H\left(t,\, l\right)=\begin{cases}\left\|
h_{\alpha \,\beta } \right\|_{\alpha ,\, \beta =1}^{2},& n>0 \\
p_{0},& n=0
\end{cases},
\end{gather}
where $B\left(\mathcal{H}^{n} \right)\ni h_{11} $ is a
three-diagonal operator matrix whose elements under the main
diagonal are equal to $\left(\frac{i}{2} q_{1} ,\, \ldots ,\,
\frac{i}{2} q_{n-1} \right)$, the elements over the main diagonal
are equal to $\left(-\frac{i}{2} s_{1} ,\, \ldots ,\, -\frac{i}{2}
s_{n-1} \right)$, the elements on the main diagonal are equal to
$\left(-p_{0} ,\, \ldots ,\, -p_{n-1} \right)$, the rest elements
are equal to zero. $B\, \left(\mathcal{ H}^{n+1} ,\, \mathcal{
H}^{n} \right)\ni h_{12} $ is an operator matrix whose elements
with numbers $j,\, j-1$ are equal to $I_{1} ,\, j=2,\, \ldots ,\,
n$, the element with number $n,\, n+1$ is equal to $\frac{1}{2}
s_{n} $, the rest elements are equal to zero. $B\,
\left(\mathcal{H}^{n} ,\, \mathcal{H}^{n+1} \right)\ni h_{21} $ is
an operator matrix whose elements with numbers $j-1,\, j$ are
equal to $I_{1} ,\, j=2,\, \ldots ,\, n$, the element with number
$n+1,\, n$ is equal to $\frac{1}{2} q_{n} $, the rest elements are
equal to zero. $B\, \left(\mathcal{H}^{n+1} \right)\ni h_{22} $ is
an operator matrix whose last row is equal to $\left(0,\, \ldots
,\, 0,\, -iI_{1} ,\, -p_{n} \right)$, last column is equal to
$col\, \left(0,\, \ldots ,\, 0,\, iI_{1} ,\, -p_{n} \right)$, the
rest elements are equal to zero.

Also in this case we denote \footnote{See the previous footnote}
\begin{equation} \label{GrindEQ__16_}
W\left(t,\, l,\, m\right)=\left\| m_{\alpha \beta } \right\| _{\alpha ,\, \beta =1}^{2} ,
\end{equation}
where $m_{11} $ is defined in the same way as $m_{11} $
\eqref{GrindEQ__8_}. $B\left(\mathcal{H}^{n+1} ,\, \mathcal{H}^{n}
\right)\ni m_{12} $ is an operator matrix whose element with
number $n,\, n+1$ is equal to $-\frac{1}{2} \tilde{s}_{n} $, the
rest elements are equal to zero. $B\, \left(\mathcal{H}^{n} ,\,
\mathcal{H}^{n+1} \right)\ni m_{21} $ is an operator matrix whose
element with number $n+1,\, n$ is equal to $-\frac{1}{2}
\tilde{q}_{n} $, the rest elements are equal to zero. $B\,
\left(\mathcal{H}^{n+1} \right)\ni m_{22} =\mathrm{diag}\left(0,\,
\ldots ,\, 0,\, \tilde{p}_{n} \right)$.

Obviously for $H\left(t,l\right)$ \eqref{GrindEQ__7_},
\eqref{GrindEQ__15_} and $W\left(t,l,m\right)$
\eqref{GrindEQ__8_}, \eqref{GrindEQ__16_} one has
\begin{equation} \label{GrindEQ__17_}
H^{*} \left(t,l\right)=H\left(t,l^{*} \right), W^{*} \left(t,l,\, m\right)=W\left(t,l,\, m^{*} \right).
\end{equation}

\begin{lemma}\label{lm1}
Let the order of $\Im l$ is even. Then
\begin{equation} \label{GrindEQ__18_}
\Im H\left(t,l\right)=W\left(t,l,\, -\Im l\right)=W\left(t,l^{*}
,\, -\Im l\right).
\end{equation}
\end{lemma}

\begin{proof}
Let us prove the first equality in \eqref{GrindEQ__18_} for even
$r=2n$. Let us represent $H\left(t,l\right)$ \eqref{GrindEQ__7_}
in the form
\begin{equation} \label{GrindEQ__19_}
H\left(t,l\right)=A\left(t,l\right)+B\left(t,l\right),
\end{equation}
where
\begin{gather} \label{GrindEQ__20_}
B\left(t,l\right)=\left\| B_{jk} \right\| _{j,\, k=1}^{2} ,\, \,
\, B_{jk} \in B\left(\mathcal{H}^{n} \right), \\
\label{GrindEQ__21_} B_{11} =\mathrm{diag}\left(0,\, ...,\, 0,\,
s_{n} p_{n}^{-1} q_{n} /4\right),\, \, \, B_{12}
=\mathrm{diag}\left(0,\, ...,\, 0,\, -is_{n} p_{n}^{-1} /2\right),
\\
\label{GrindEQ__22_} B_{21} =\mathrm{diag}\left(0,\, ...,\, 0,\,
ip_{n}^{-1} q_{n} /2\right),\, \, \, B_{22}
=\mathrm{diag}\left(0,\, ...,\, 0,\, p_{n}^{-1} \right).
\end{gather}
In view of \eqref{GrindEQ__13_}, \eqref{GrindEQ__20_} - \eqref{GrindEQ__22_} one has
\begin{equation} \label{GrindEQ__2311_}
B\left(t,l\right)C\left(t,l\right)=\left\| u_{jk} \right\| _{j,\,
k=1}^{2n} ,\, \, \, u_{jk} \in B\left(\mathcal{H}\right),
\end{equation}
$u_{n2n} =-i{s_{n}  \mathord{\left/{\vphantom{s_{n}
2}}\right.\kern-\nulldelimiterspace} 2} ,\, \, \, u_{2n\, \, 2n}
=I_{1} $, rest $u_{jk} =0$.

Hence $$C^{*}
\left(t,l\right)B\left(t,l\right)C\left(t,l\right)=\left\|
{v}_{jk} \right\|_{j,k=1}^{2n} ,\, \, {v}_{jk} \in
B\left(\mathcal{H}\right),$$ ${ v}_{n\, \, 2n} =-\frac{i}{2}
\left(s_{n} -q_{n}^{*} \right)$, ${v}_{2n\, \, 2n} =p_{n}^{*} $,
rest $v_{jk} =0$.

Hence $C^{*} \left(t,l\right)\Im
H\left(t,l\right)C\left(t,l\right)=C^{*}
\left(t,l\right)W\left(t,l,-\Im l\right)C\left(t,l\right)$ in view
of \eqref{GrindEQ__7_}, \eqref{GrindEQ__8_}, \eqref{GrindEQ__9_},
\eqref{GrindEQ__19_}. The first equality in \eqref{GrindEQ__18_}
for even $r$ is proved. Its proof for odd $r$ follows from
\eqref{GrindEQ__15_}, \eqref{GrindEQ__16_}.

One can see from the proof that
\begin{equation} \label{GrindEQ__231_}
W\left(t,l,\Im l\right)=-\Im H\left(t,l\right).
\end{equation}

The second equality in \eqref{GrindEQ__18_} is a corollary of
\eqref{GrindEQ__231_} and \eqref{GrindEQ__17_}. Lemma \ref{lm1} is
proved
\end{proof}

For sufficiently smooth vector-function $f\left(t\right)$ by
corresponding capital letter we denote (if $f(t)$ has a subscript
then we add the same subscript to $F$)
\begin{multline} \label{GrindEQ__23_}
\mathcal{H}^{r} \ni F\left(t,\, l,m\right)=\\=\begin{cases}
\left(\sum\limits _{j=0}^{s/2}\oplus f^{\left(j\right)}
\left(t\right) \right)\oplus
0 \oplus ...\oplus 0 ,\quad r=2n,&r=2n+1>1,s<2n,\\
\left(\sum\limits _{j=0}^{n-1}\oplus f^{\left(j\right)}
\left(t\right) \right)\oplus 0 \oplus ...\oplus 0 \oplus \,
\left(-if^{\left(n\right)} \left(t\right)\right),&r=2n+1>1,
s=2n,\\ f\left(t\right),&r=1, \\ \left(\sum\limits
_{j=0}^{n-1}\oplus f^{\left(j\right)} \left(t\right) \right)\oplus
\left(\sum\limits _{j=1}^{n}\oplus f^{\left[r-j\right]}
\left(t\left|l\right. \right) \right),& r=s=2n\end{cases}
\end{multline}

From now on in equation \eqref{GrindEQ__4_}
\begin{equation*} \label{GrindEQ__241_}
p_{n}^{-1} \left(t\right)\in B\left(\mathcal{H}\right)\, \, \,
\left(r=2n\right);\, \left(q_{n+1} \left(t\right)+s_{n+1}
\left(t\right)\right)^{-1} \in B\left(\mathcal{H}\right)\, \, \,
\left(r=2n+1\right).
\end{equation*}

\begin{theorem}\label{th1}
Equation \eqref{GrindEQ__4_} is equivalent to the following first
order system
\begin{equation} \label{GrindEQ__24_}
\frac{i}{2} \left(\left(Q\left(t\right)\vec{y}\right)^{{'} }
+S\left(t\right)\vec{y}\hspace{2px}^\prime\right)+H\left(t\right)\vec{y}=W\left(t\right)F\left(t\right)
\end{equation}
with coefficients $Q\left(t\right)=Q\left(t,l\right)$,
$S(t)=S(t,l)$ \eqref{GrindEQ__6_}, \eqref{GrindEQ__14_},
$H(t)=H(t,l)$ \eqref{GrindEQ__7_}, \eqref{GrindEQ__15_}, weight
$W\left(t\right)=W\left(t,\, l^{*} ,\, m\right)$, and with
$F\left(t\right)=F\left(t,\, l^{*} ,m\right)$ that are obtained
from \eqref{GrindEQ__8_}, \eqref{GrindEQ__16_} and
\eqref{GrindEQ__23_} correspondingly with $l^{*} $ instead of $l$.
Namely if $y\left(t\right)$ a solution of equation
\eqref{GrindEQ__4_} then
\begin{multline} \label{GrindEQ__25_}
\vec{y}\left(t\right)=\vec{y}\left(t,\, l,\, m,\,
f\right)=\\=\begin{cases}\left(\sum\limits _{j=0}^{n-1}\oplus
y^{\left(j\right)} \left(t\right) \right)\oplus \left(\sum\limits
_{j=1}^{n}\oplus \left(y^{\left[r-j\right]} \left(t\left|l\right.
\right)-f^{\left[s-j\right]}
\left(t\left|m\right. \right)\right) \right),& r=2n\\
\left(\sum\limits _{j=0}^{n-1}\oplus y^{\left(j\right)}
\left(t\right) \right)\oplus \left(\sum\limits _{j=1}^{n}\oplus
\left(y^{\left[r-j\right]} \left(t\left|l\right.
\right)-f^{\left[s-j\right]} \left(t\left|m\right. \right)\right)
\right)\oplus \left(-iy^{\left(n\right)} \left(t\right)\right),&
r=2n+1>1 \\\quad \text{{\rm (here }}f^{\left[k\right]}
\left(t\left|m\right.
\right)\equiv 0\text{ \rm as }k< \frac{s}{2}\text{\rm)}\\
\\y\left(t\right),& r=1\end{cases}
\end{multline}
is a solution of \eqref{GrindEQ__24_} with the coefficients,
weight and $F\left(t\right)$ menfioned above. Any solution of
equation \eqref{GrindEQ__24_} with such coefficients, weight and
$F\left(t\right)$ is equal to \eqref{GrindEQ__25_}, where
$y\left(t\right)$ is some solution of equation
\eqref{GrindEQ__4_}.
\end{theorem}

Let us notice that different vector-functions $f(t)$ can generate
different RHSs of equation \eqref{GrindEQ__24_} but unique RHS of
equation \eqref{GrindEQ__4_}.

\begin{proof}
We need the following three lemmas.

\begin{lemma}\label{lm2}
Let $L_{\alpha } $ be a differential expression of $l$ type and of
order $\alpha$. Let us add to $L_{\alpha } $the expressions of
$i^{k} l_{k} $ type, where $k=\alpha +1,\, \ldots ,\, \beta $,
with coefficiens equal to zero. We obtain the expressions
$L_{\beta } $ which formally has the order $\beta $, but in fact
$L_{\beta } $ and $L_{\alpha } $ coincide. Then for sufficiently
smooth vector-function $f\left(t\right)$
$$f^{\left[\beta -j\right]} \left(t\left|L_{\beta } \right.
\right)=\begin{cases}f^{\left[\alpha -j\right]}
\left(t\left|L_{\alpha } \right. \right),& j=0,\, \ldots ,\,
\left[\frac{a+1}{2} \right], \\ 0,& j=\left[\frac{a+1}{2}
\right]+1,\, \ldots ,\, \left[\frac{\beta }{2} \right]
\end{cases}$$
(here $f^{[0]}(t|L_1)$ is defined by (\ref{GrindEQ__11_}) with
$r=1$).
\end{lemma}

\begin{proof}
Proof of Lemma \ref{lm2} follows from formulae
\eqref{GrindEQ__11_} -- \eqref{GrindEQ__12_} for
quasi-derivatives. \end{proof}

\begin{lemma}\label{lm3}
Let $f\left(t\right)\in C^{s} \left(\left[\alpha ,\beta
\right],\mathcal{H}\right),\, y\left(t\right)$ is a solution of
corresponding equation \eqref{GrindEQ__4_}. Then the sequence
$f_{k} \left(t\right)\in C^{\infty } \left(\left[\alpha ,\beta
\right],\, \mathcal{H}\right)$ and solutions $y_{k}
\left(t\right)$ of equation \eqref{GrindEQ__4_} with
$f\left(t\right)=f_{k} \left(t\right)$ exist such that
\[f_{k} \left(t\right)\stackrel{C^{s} \left(\left[\alpha ,\beta \right],
\mathcal{H}\right)}{\longrightarrow} f\left(t\right),\, \, \,
y_{k} \left(t\right)\stackrel{C^{r} \left(\left[\alpha ,\beta
\right],\mathcal{H}\right)}{\longrightarrow} y\left(t\right).\]
\end{lemma}

This is trivial consequence of Weierstrass theorem for
vector-functions \cite{Schwartz} and formula (1.21) from
\cite{DalKrein}.

\begin{lemma}\label{lm4}
Let vector-function $f\left(t\right)\in C^{s}
\left(\bar{\mathcal{I}},\, \mathcal{H}\right)$. Then
\begin{multline} \label{GrindEQ__251_}
W\left(t,\, l^{*} ,\, m\right)F\left(t,\, l^{*}
,m\right)=\\=\begin{cases} \left(\sum\limits _{j=0}^{s/2-1}\oplus
\left(f^{\left[s-j\right]} \left(t\left|m\right.
\right)+\left(f^{\left[s-j-1\right]} \left(t\left|m\right.
\right)\right)^{{'} } \right)\right)\oplus\\\qquad\oplus
f^{\left[s/2\right]} \left(t\left|m\right. \right)\oplus 0\oplus
\ldots \oplus 0,& r=2n+1, r=2n, 0<s<2n  \\
\left(\sum\limits _{j=0}^{s/2-1}\oplus \left(f^{\left[s-j\right]}
\left(t\left|m\right. \right)+\left(f^{\left[s-j-1\right]}
\left(t\left|m\right. \right)\right)^{{'} } \right)
\right)\oplus\\\qquad\oplus 0\oplus \ldots +\oplus 0\oplus
\left(-if^{\left[n\right]}
\left(t\left|m\right. \right)\right),& r=2n+1, s=2n>0 \\
\tilde{p}_{0} \left(t\right)f\left(t\right)\oplus 0\oplus \ldots
\oplus 0,& s=0 \\\bigg[ \left(\sum\limits _{j=0}^{s/2-1}\oplus
\left(f^{\left[s-j\right]} \left(t\left|m\right.
\right)+\left(f^{\left[s-j-1\right]} \left(t\left|m\right.
\right)\right)^{{'} } \right) \right)\oplus\\\qquad\oplus 0\oplus
\ldots \oplus 0\bigg]+H\left(t,l\right)\left(0\oplus \ldots \oplus
0\oplus f^{\left[n\right]} \left(t\left|m\right. \right)\right),&
r=s=2n
\end{cases}
\end{multline}
\end{lemma}

Let us notice that $W(t,l^*,m)F(t,l^*,m)$ does not change if the
null-components in $F(t,l^*,m)$ we change by any
$\mathcal{H}$-valued vector-functions.

\begin{proof}
Let us prove Lemma \ref{lm4} for $r=s=2n$. It is sufficient to
verify that
\begin{multline}\label{GrindEQ__26_}
\left(\left\| m_{\alpha \beta } \left(t\right)\right\| _{\alpha
,\beta =1}^{2} \right)col\left\{f\left(t\right),\,
f'\left(t\right),\, ...,\, f^{\left(n-1\right)} \left(t\right),\,
f^{\left(2n-1\right)} \left(t\right),\, ...\, f^{\left(n\right)}
\left(t\right)\right\}= \\=C^{*} \left(t,l^{*}
\right)\left\{\left[\left(\sum\limits
_{j=0}^{n-1}\left(f^{\left[r-j\right]} \left(t\left|m\right.
\right)+\left(f^{[r-j-1]} \left(t\left|m\right.
\right)\right)^{{'} } \right)\right)\oplus 0\oplus ...\oplus 0
\right]+\right.\\\left. +H\left(t\left|l\right.
\right)\left(0\oplus 0\oplus ...\oplus 0\oplus f^{\left[n\right]}
\left(t\left|m\right. \right)\right)\right\}.
\end{multline}

But in view of \eqref{GrindEQ__8_}, \eqref{GrindEQ__11_},
\eqref{GrindEQ__12_} the left side of equality
\eqref{GrindEQ__26_} is equal to
\[\left(\sum\limits _{j=0}^{n-1}\oplus \left(f^{\left[r-j\right]} \left(t\left|m\right. \right)\right)+\left(f^{\left[r-j-1\right]} \left(t\left|m\right. \right)\right)^{{'} }  \right)\oplus {\rm O} \oplus ...\oplus {\rm O} \oplus f^{\left[n\right]} \left(t\left|m\right. \right).\]
And hence equality \eqref{GrindEQ__26_} is true since
$C\left(t,l^{*} \right)\left[\ldots \right]=\left[\ldots \right]$
and the last column of $C^{*} \left(t,l^{*}
\right)H\left(t,l\right)$ is equal to $col\left(0,...,\, 0,\,
I_{1} \right)$ in view of \eqref{GrindEQ__7_},
\eqref{GrindEQ__13_}.

The proof for $r=2n+1,\, s=2n$ is carried out via direct
calculation using \eqref{GrindEQ__16_}, \eqref{GrindEQ__11_},
\eqref{GrindEQ__12_}.

The proof for $s<2n$ follows from the case $s=2n$ consicered
above, Lemmas \ref{lm2}, \ref{lm3} and fact that elements $u_{jk}
\in B\left(\mathcal{H}\right)$ of matrix $W\left(t,l^{*}
,m\right)$ are equal to zero if $s<2n$ and $i>{s
\mathord{\left/{\vphantom{s 2}}\right.\kern-\nulldelimiterspace}
2} $ or $j>{s \mathord{\left/{\vphantom{s
2}}\right.\kern-\nulldelimiterspace} 2} $. Lemma \ref{lm4} is
proved.
\end{proof}

Let us return to the proof of Theorem \ref{th1}. Let
$y\left(t\right)$ is a solution of equation \eqref{GrindEQ__4_}.
Then
\begin{multline} \label{GrindEQ__261_}
\frac{i}{2}
\left\{\left(Q\left(t,l\right)\vec{y}\left(t,l,m,0\right)\right)^{{'}
}
+S\left(t,l\right)\vec{y}\hspace{2px}'\left(t,l,m,0\right)\right\}-
H\left(t,l\right)\vec{y}\left(t,l,m,0\right)=\\=\mathrm{diag}\left(y^{\left[r\right]}
\left(t\left|l\right. \right),0,\ldots ,0\right)
\end{multline}
in view of formulae that are analogues to formulae (4.10), (4.11),
(4.24), (4.25) from \cite{KoganRB}. Using \eqref{GrindEQ__261_}
and Lemma \ref{lm4} we show via direct calculations that
$\vec{y}\left(t,l,m,f\right)$ \eqref{GrindEQ__25_} is a solution
of \eqref{GrindEQ__24_} for $r=s=2n,\, \, r=2n+1,\, \, s=2n$.
Therefore in view of Lemmas \ref{lm2}, \ref{lm3}\
$\vec{y}\left(t,l,m,f\right)$ is a solution of
\eqref{GrindEQ__24_} for $s<2n$.

Conversely let $\vec{\tilde{y}}\left(t\right)=col\left(y_{1}
,\ldots ,y_{r} \right)$ is a solution of \eqref{GrindEQ__24_}. Let
$y\left(t\right)$ is a solution of Cauchy problem that is obtained
by adding the initial condition
$\vec{y}\left(0,l,m,f\right)=\vec{\tilde{y}}\left(0\right)$ to
equation \eqref{GrindEQ__4_}. Then
$\vec{\tilde{y}}\left(t\right)=\vec{y}\left(t,l,m,f\right)$ in
view of existence and uniqueness theorem. Theorem \ref{th1} is
proved
\end{proof}

Let us notice that Theorem \ref{th1} remains valid if
null-components of $F(t,l^*,m)$ we change by any
$\mathcal{H}$-valued vector-functions.

For differential expression $L=\sum\limits _{k=0}^{R}i^{k} L_{k}
$, where $L_{2j} =D^{j} P_{j} \left(t\right)D^{j}$,\\ $L_{2j-1}
=\frac{1}{2} D^{j-1} \left\{DQ_{j} \left(t\right)+S_{j}
\left(t\right)\, D\right\}D^{j-1} $, we denote by
\begin{equation} \label{GrindEQ__27_}
L\left[f,\, g\right]=\int_{\mathcal{I}}L\left\{f,g\right\}dt ,
\end{equation}
the bilinear form which corresponds to Dirichlet integral for this
expression. Here
\begin{multline} \label{GrindEQ__28_}
L\left\{f,g\right\}=\sum\limits _{j=0}^{\left[{R
\mathord{\left/{\vphantom{R 2}}\right.\kern-\nulldelimiterspace}
2} \right]}\left(P_{j} \left(t\right)f^{\left(j\right)}
\left(t\right),\, g^{\left(j\right)} \left(t\right)\right)+\\+
\frac{i}{2} \sum\limits _{j=1}^{\left[\frac{R+1}{2}
\right]}\left(S_{j} \left(t\right)f^{\left(j\right)}
\left(t\right),\, g^{\left(j-1\right)}
\left(t\right)\right)-\left(Q_{j}
\left(t\right)f^{\left(j-1\right)}
\left(t\right),g^{\left(j\right)} \left(t\right)\right)
\end{multline}

\begin{theorem}[On the relationships between bilinear
forms]\label{th2} Let $f\left(t\right),\, y\left(t\right),\, f_{k}
\left(t\right),\, y_{k} \left(t\right)\, \left(k=1,2\right)$ be
sufficiently smooth vector-function. Starting from these functions
by the formulae (\ref{GrindEQ__23_}), (\ref{GrindEQ__25_}) we
construct $F(t,l,m)$, $F_k(t,l,m)$, $\vec{y}(t,l,m,f)$,
$\vec{y}_k(t,l,m,f_k)$. Then:
\\
1.
\begin{equation} \label{GrindEQ__29_}
\left(W\left(t,\, l,\, m\right)F_{1} \left(t,l,m\right),\, F_{2}
\left(t,l,m\right)\right)=m\left\{f_{1} ,f_{2} \right\}.
\end{equation}
2. a) If the order of $\Im l$ is even, then
\begin{multline} \label{GrindEQ__30_}
\left(W\left(t,l,-\Im l\right)\vec{y}\left(t,l,m,f\right),\,
\vec{y}\left(t,\, l,\, m,\, f\right)\right)-\Im \left(W\left(t,\,
l^{*} ,m^{*} \right)\vec{y}\left(t,\, l,m,f\right)\right),\,
F\left(t,l^{*} ,m\right)=\\=-\left(\Im
l\right)\left\{y,y\right\}-\Im\left( m^{*}
\left\{y,f\right\}\right).
\end{multline}
\quad\ b)
\begin{multline}\label{GrindEQ__31_}
m\left\{y_{1} ,f_{2} \right\}-m\left\{f_{1} ,y_{2} \right\}=\\=
\left(W\left(t,l,m\right)\vec{y}_{1} \left(t,l,m,f_{1}
\right),F_{2} \left(t,l,m\right)\right)-\left(W\left(t,l^{*}
,m\right)F_{1} \left(t,l^{*} ,m\right),\vec{y}_2\left(t,l^{*}
,m^{*} ,f_{2} \right)\right),
\end{multline}
although for $r=s$ the corresponding terms in the right-and
left-hand side of \eqref{GrindEQ__30_} and \eqref{GrindEQ__31_} do
not coincide.
\end{theorem}

\begin{proof}
1. follows from \eqref{GrindEQ__8_}, \eqref{GrindEQ__16_},
\eqref{GrindEQ__23_}, \eqref{GrindEQ__28_}.

2. Let $r=s=2n$. For convenience when using notations of
\eqref{GrindEQ__23_} type we omit the argument $m$. For example by
$F\left(t,l^*\right)$ we denote $F\left(t,l^*,m\right)$.

a) We denote
\begin{equation} \label{GrindEQ__32_}
\mathcal{F}\left(t,m\right)=col\left\{0,...,0,f^{\left[2n-1\right]}
\left(t\left|m\right. \right),...,f^{\left[n\right]}
\left(t\left|m\right. \right)\right\}\in \mathcal{H}^{r}.
\end{equation}
One has
\begin{multline}\label{GrindEQ__33_}
\left(W\left(t,l,-\Im
l\right)\vec{y}\left(t,l,m,f\right),\vec{y}\left(t,l,m,f\right)\right)=\left(W\left(t,l,-\Im
l\right)Y\left(t,l\right),Y\left(t,l\right)\right)-\\
-\left(W\left(t,l,-\Im
l\right)Y\left(t,l\right),\mathcal{F}\left(t,m\right)\right)-\left(W\left(t,l,-\Im
l\right)\mathcal{F}\left(t,m\right),Y\left(t,l\right)\right)+\\
+\left(W\left(t,l,-\Im l\right)\mathcal{F}\left(t,m\right),\mathcal{F}\left(t,m\right)\right)=-\left(\Im l\right)\left[y,y\right]+\\
+2\Re \left(p_{r}^{*-1} y^{\left[n\right]} \left(t|\Im
l\right),f^{\left[n\right]} \left(t|m\right)\right)+\Im
\left(p_{r}^{-1} f^{\left[n\right]}
\left(t|m\right),f^{\left[n\right]} \left(t|m\right)\right).
\end{multline}
Here the last equality follows from \eqref{GrindEQ__17_},
\eqref{GrindEQ__29_}, \eqref{GrindEQ__251_}, \eqref{GrindEQ__7_}.
On the other hand we have
\begin{multline} \label{GrindEQ__331_}
\Im \left(W\left(t,l^{*}
,m^*\right)\vec{y}\left(t,l,m,f\right),F\left(t,l^{*}
\right)\right)=\Im \left(W\left(t,l^{*} ,m^*\right)Y\left(t,l^{*}
\right),F\left(t,l^{*} \right)\right)+\\
+\Im \left(W\left(t,l^{*}
,m^*\right)\left(\left(Y\left(t,l\right)-Y\left(t,l^{*}\right)\right)-
\mathcal{F}\left(t,m\right)\right),F\left(t,l^{*}
\right)\right)=\Im \left(m^*\left\{y,f\right\}\right)+\\
+2\Re \left(p_{r}^{*-1} y^{\left[n\right]} \left(t|\Im
l\right),f^{\left[n\right]} \left(t|m\right)\right)+\Im
\left(p_{n}^{-1} f^{\left[n\right]}
\left(t|m\right),f^{\left[n\right]} \left(t|m\right)\right).
\end{multline}
Here the last equality is proved similarly to \eqref{GrindEQ__33_}
taking into account that $y^{\left[n\right]}
\left(t|l\right)-y^{\left[n\right]} \left(t|l^{*}
\right)=2iy^{\left[n\right]} \left(t|\Im l\right)$. Comparing
\eqref{GrindEQ__33_}, \eqref{GrindEQ__331_} we obtain
\eqref{GrindEQ__30_}.

b) In view of \eqref{GrindEQ__25_}, \eqref{GrindEQ__29_},
\eqref{GrindEQ__17_} and Lemma \ref{lm4} we have
\begin{multline}\label{GrindEQ__34_}
\left(W\left(t,l,m\right)\vec{y}_1\left(t,l,m,f_{1} \right),F_{2}
\left(t,l\right)\right)=m\left\{y_{1} \left(t,l,m,f_{1}
\right),f_{2} \right\}-\\- \left(\mathcal{F}_{1}
\left(t,m\right),H\left(t,l^{*} \right)col\left\{0,\, ...,0,f_{2}
^{\left[n\right]} \left(t\left|m^{*} \right.
\right)\right\}\right)=\\
 =m\left\{y_{1} ,f_{2}
\right\}-\left(p_{n}^{-1} f_{1}^{\left[n\right]}
\left(t\left|m\right. \right),f_{2}^{\left[n\right]}
\left(t\left|m^{*} \right. \right)\right).
\end{multline}
Similarly
\begin{multline} \label{GrindEQ__35_}
\left(W\left(t,l^{*} ,m\right)F_1\left(t,l^{*} \right),\vec{y}_{2}
\left(t,l^{*} ,m^{*} ,f_{2} \right)\right) =m\left\{f_{1} ,y_{2}
\right\}-\left(p_{n}^{-1} f_{1}^{\left[n\right]}
\left(t\left|m\right. \right),f_{2}^{\left[n\right]}
\left(t\left|m^{*} \right. \right)\right)
\end{multline}

Comparing \eqref{GrindEQ__34_}, \eqref{GrindEQ__35_} we obtain
\eqref{GrindEQ__31_}.

For $r=2n+1,\, s=2n$ or $r=2n+1\vee 2n,\, s<2n$, the corresponding
terms in \eqref{GrindEQ__30_}, \eqref{GrindEQ__31_} coincide in
view of \eqref{GrindEQ__8_}, \eqref{GrindEQ__16_},
\eqref{GrindEQ__23_}, \eqref{GrindEQ__25_}, \eqref{GrindEQ__29_}.
For example in these cases
\begin{gather*}
\left(W\left(t,l,-\Im
l\right)\vec{y}\left(t,l,m,f\right),\vec{y}\left(t,l,m,f\right)\right)=\left(\left(W\left(t,l,-\Im
l\right)\right)Y\left(t,l\right),Y\left(t,l\right)\right)=-\left(\Im
l\right)\left\{y,y\right\} \end{gather*}

Theorem \ref{th2} is proved.\end{proof}

Let us notice that Theorem \ref{th2} remains valid if
null-components in $F_k(t,l,m)$, $F(t,l^*,m)$, $F_1(t,l^*,m)$ we
change by any $\mathcal{H}$-valued vector-functions.

\begin{theorem}[The Green formula]\label{th3} Let $\mathrm{l}_{k} ,\, \mathrm{m}_{k} \,
\left(k=1,2\right)$ are differential expressions of $l$
\eqref{GrindEQ__51+_}, $m$ type correspondingly. The orders of
$\mathrm{l}_{k} $ are equal to $r$, the orders $\mathrm{m}_{k} $
are different in general, even and are equal to $s_{k}\leq r$. Let
$y_{k} \left(t\right)\in C^{r} \left(\left[\alpha ,\, \beta
\right],\, \mathcal{H}\right)$, $f_{k} \left(t\right)\in C^{s_{k}
} \left(\left[\alpha ,\, \beta \right],\, \mathcal{H}\right)$, and
$\mathrm{l}_{k} \left[y_{k} \right]=\mathrm{m}_{k} \left[f_{k}
\right],\, \, k=1,\, 2$. Then
\begin{multline} \label{GrindEQ__36_}
\int _{\alpha }^{\beta }\mathrm{m}_{1} \left\{f_{1} ,y_{2}
\right\} dt-\int _{\alpha }^{\beta }\mathrm{m}_{2}^{*}
\left\{y_{1} ,f_{2} \right\}dt -\int _{\alpha }^{\beta
}\left(\mathrm{l}_{1} -\mathrm{l}_{2}^{*} \right)\left\{y_{1}
,y_{2} \right\}dt =
\\\left. =\left(\frac{i}{2} \left(Q\left(t,\mathrm{l}_{1}
\right)+Q_{}^{*} \left(t,\mathrm{l}_{2} \right)\right)\vec{y}_{1}
\left(t,\, \mathrm{l}_{1} ,\, \mathrm{m}_{1} ,\, f_{1} \right),\,
\vec{y}_{2} \left(t,\, \mathrm{l}_{2} ,\, \mathrm{m}_{2} ,f_{2}
\right)\right)\right|_\alpha^\beta,
\end{multline}
where $Q(t,\mathrm{l}_k)$,
$\vec{y}_k(t,\mathrm{l}_k,\mathrm{m}_k,f_k)$ correspond to
equations $\mathrm{l}_{k} \left[y\right]=\mathrm{m}_{k}
\left[f\right]$ by formulae \eqref{GrindEQ__6_},
\eqref{GrindEQ__14_}, \eqref{GrindEQ__25_} with
$\mathrm{l}_{k},\mathrm{m}_{k},y_{k},f_{k} $ instead of $l, m, y,
f$ correspondingly.
\end{theorem}

\begin{proof}
We need the following
\begin{lemma}\label{lm5}
For sufficiently smooth vector-function $g_{1} \left(t\right),\,
g_{2} \left(t\right)$ one has
\begin{multline} \label{GrindEQ__37_}
\left(\left(H\left(t,\mathrm{l}_{1}
\right)-H\left(t,\mathrm{l}_{2}^{*} \right)\right)\vec{g}_{1}
\left(t,\mathrm{l}_{1},\mathrm{m}_1,0\right),\vec{g}_{2}
\left(t,\mathrm{l}_{2},\mathrm{m}_2
,0\right)\right)=\\=\begin{cases} -\left(\mathrm{l}_{1}
-\mathrm{l}_{2}^{*} \right)\left\{g_{1} ,g_{2} \right\},& r=2n \\
-\left(\mathrm{l}_{1} -\mathrm{l}_{2}^{*} \right)\left\{g_{1}
,g_{2} \right\}+\left(\mathrm{l}_{2n+1}^{1}
-\mathrm{l}_{2n+1}^{2^{*} } \right)\left\{g_{1} ,g_{2} \right\},&
r=2n+1,
\end{cases}
\end{multline}
where $\mathrm{l}_{2n+1}^{k} $ are the analogs of $l_{2n+1}$.
\end{lemma}

\begin{proof} Let $r=2n$. Then in view of
\eqref{GrindEQ__19_}-\eqref{GrindEQ__2311_}, \eqref{GrindEQ__25_},
\eqref{GrindEQ__9_}, \eqref{GrindEQ__17_} we have
\begin{multline*}
\left(\left(H\left(t,\mathrm{l}_{1}
\right)-H\left(t,\mathrm{l}_{2}^{*} \right)\right)\vec{g}_{1}
\left(t,\mathrm{l}_{1},\mathrm{m}_1 ,0\right),\vec{g}_{2}
\left(t,\mathrm{l}_{2},\mathrm{m}_2 ,0\right)\right)=\\
=\left(\left(A\left(t,\mathrm{l}_{1}
\right)-A\left(t,\mathrm{l}_{2}^* \right)\right)\vec{g}_{1}
\left(t,\mathrm{l}_{1},\mathrm{m}_1 ,0\right),\vec{g}_{2}
\left(t,\mathrm{l}_{2},\mathrm{m}_2 ,0\right)\right)+\\
+\left(C^{*} \left(t,\mathrm{l}_{2} \right)B\left(t,\mathrm{l}_{1}
\right)C\left(t,\mathrm{l}_{1} \right)col\left\{g_{1} ,g'_{1} ,\,
...,\, g_{1}^{\left(n-1\right)} ,g_{1}^{\left(2n-1\right)} ,\,
...,\,
g_{1}^{\left(n\right)} \right\}\right.,\\
\left. col\left\{g_{2} ,g'_{2} ,\, ...,\, g_{2}^{\left(n-1\right)}
,g_{2}^{\left(2n-1\right)} ,\, ...,\, g_{2}^{\left(n\right)}
\right\}\right)-\\
-\left(col\left\{g_{1} ,g'_{1} ,\ldots ,g_{1}^{\left(2n-1\right)}
,\ldots ,g_{1}^{\left(n\right)} \right\},C^{*}
\left(t,\mathrm{l}_{1} \right)B\left(t,\mathrm{l}_{2}
\right)C\left(t,\mathrm{l}_{2} \right)col\left\{g_{2} ,g'_{2}
,\ldots ,g_{2}^{\left(n-1\right)} ,g_{2}^{\left(2n-1\right)}
,\ldots ,g_{2}^{\left(n\right)} \right\}\right)=
\\=-\left(\mathrm{l}_{1} -\mathrm{l}_{2}^{*}
\right)\left\{f,g\right\}.
\end{multline*}

The proof of \eqref{GrindEQ__37_} for $r=2n+1$ follows directly
from \eqref{GrindEQ__15_}, \eqref{GrindEQ__25_}. Lemma \ref{lm5}
is proved.
\end{proof}

Now Green formula \eqref{GrindEQ__36_} we obtain from the
following Green formula for the equation \eqref{GrindEQ__24_} that
correspond to equations
$\mathrm{l}_{k}\left[y\right]=\mathrm{m}_{k} \left[f\right]$
\begin{multline}\label{GrindEQ__38_}
\int _{\alpha }^{\beta }\left(W\left(t,\mathrm{l}_{1}^{*}
,\mathrm{m}_{1} \right)F_{1} \left(t,\mathrm{l}_{1}^{*}
,\mathrm{m}_{1} \right),\vec{y}_{2}
\left(t,\mathrm{l}_{2} ,\mathrm{m}_{2} ,f_{2} \right)\right)dt -\\
-\int _{\alpha }^{\beta }\left(W\left(t,\mathrm{l}_{2}^{*}
,\mathrm{m}_{2}^* \right)\vec{y}_{1} \left(t,\mathrm{l}_{1}
,\mathrm{m}_{1} ,f_{1} \right),F_{2} \left(t,\mathrm{l}_{2}^{*}
,\mathrm{m}_{2} \right)\right)dt+
\\
+\int _{\alpha }^{\beta }\left(\left(H\left(t,\mathrm{l}_{1}
\right)-H\left(t,\mathrm{l}_{2}^{*} \right)\right)\vec{y}_{1}
\left(t,\mathrm{l}_{1} ,\mathrm{m}_{1} ,f_{1} \right),\vec{y}_{2}
\left(t,\mathrm{l}_{2} ,\mathrm{m}_{2} f_{2}
\right)\right)dt -\\
-\int _{\alpha }^{\beta }\frac{i}{2}
\left\{\left((S\left(t,\mathrm{l}_{1} \right)-Q^{*}
\left(t,\mathrm{l}_{2} \right))\vec{y}\hspace{2px}'_{1}
\left(t,\mathrm{l}_{1} ,\mathrm{m}_{1} ,f_{1} \right),\vec{y}_{2}
\left(t,\mathrm{l}_{2} ,\mathrm{m}_{2} ,f_{2}
\right)\right)\right. - \\\left.
-\left(\left(Q\left(t,\mathrm{l}_{1} \right)-S^{*}
\left(t,\mathrm{l}_{2} \right)\right)\vec{y}_{1}
\left(t,\mathrm{l}_{1} ,\mathrm{m}_1,f_{1}
\right),\vec{y}\hspace{2px}'_{2} \left(t,\mathrm{l}_{2}
,\mathrm{m}_{2} ,f_{2} \right)\right)\right\}dt=\\\left.
=\left(\frac{i}{2} \left(Q\left(t,\mathrm{l}_{1} \right)+Q^{*}
\left(t,\mathrm{l}_{2} \right)\right)\vec{y}_{1}
\left(t,\mathrm{l}_{1} ,\mathrm{m}_1,f_{1} \right),\vec{y}_{2}
\left(t,\mathrm{l}_{2} ,\mathrm{m}_{2}, f_{2}
\right)\right)\right|_\alpha^\beta.
\end{multline}

Let $r=s_k=2n$. For convenience by
$F_k\left(t,\mathrm{l}_k^*\right),Y_k(t,\mathrm{l}_k)$ we denote
$F_k\left(t,\mathrm{l}_k^*,\mathrm{m}_k\right),Y_k(t,\mathrm{l}_k,\mathrm{m}_k)$
correspondingly. Then in view of \eqref{GrindEQ__7_},
\eqref{GrindEQ__25_}, \eqref{GrindEQ__251_}, \eqref{GrindEQ__29_},
\eqref{GrindEQ__37_} one has:
\begin{multline}
\label{GrindEQ__39_} \left(W\left(t,\mathrm{l}_{1}^{*}
,\mathrm{m}_{1} \right)F_{1} \left(t,\mathrm{l}_{1}^{*}
\right),\vec{y}_{2} \left(t,\mathrm{l}_{2} ,\mathrm{m}_{2} ,f_{2}
\right)\right)=\mathrm{m}_{1} \left\{f_{1} ,y_{2} \right\}+
\\
+\left(H\left(t,\mathrm{l}_{1} \right)col\left\{0,\,
...,0,f_{1}^{\left[n\right]} \left(t\left|\mathrm{m}_{1} \right.
\right)\right\},col\left\{0,\, ...,0,y_{2}^{\left[n\right]}
\left(t\left|\mathrm{l}_{2} \right. \right)-y_{2}^{\left[n\right]}
\left(t\left|\mathrm{l}_{1}^{*} \right.
\right)-f_{2}^{\left[n\right]} \left(t\left|\mathrm{m}_{2} \right.
\right)\right\}\right);
\end{multline}
\begin{multline}\label{GrindEQ__40_}
\left(W\left(t,\mathrm{l}_{2}^{*} ,\mathrm{m}_{2}^{*}
\right)\vec{y}_{1} \left(t,\mathrm{l}_{1} ,\mathrm{m}_{1} ,f_{1}
\right),F_{2} \left(t,\mathrm{l}_{2}^{*}
\right)\right)=\mathrm{m}_{2}^{*} \left\{y_{1} ,f_{2} \right\}+\\
+\left(H\left(t,\mathrm{l}_{2}^{*}
\right)col\left\{0,...,0,y_{1}^{\left[n\right]}
\left(t\left|\mathrm{l}_{1} \right. \right)-y_{1}^{\left[n\right]}
\left(t\left|\mathrm{l}_{2}^{*} \right.
\right)-f_{1}^{\left[n\right]} \left(t\left|\mathrm{m}_{1} \right.
\right)\right\},col\left\{0,...,0,f_{2}^{\left[n\right]}
\left(t\left|\mathrm{m}_{2} \right. \right)\right\}\right);
\end{multline}
\begin{multline}
\label{GrindEQ__41_} \left(\left(H\left(t,\mathrm{l}_{1}
\right)-H\left(t,\mathrm{l}_{2}^{*} \right)\right)\vec{y}_{1}
\left(t,\mathrm{l}_{1} ,\mathrm{m}_{1} ,f_{1} \right),\vec{y}_{2}
\left(t,\mathrm{l}_{2} ,\mathrm{m}_{2} ,f_{2}
\right)\right)=-\left(\mathrm{l}_{1} -\mathrm{l}_{2}^{*}
\right)\left\{y_{1} ,y_{2}
\right\}-\\
-\left(\left(H\left(t,\mathrm{l}_{1}\right)-H\left(t,\mathrm{l}_{2}^{*}
\right)\right)Y_{1} \left(t,\mathrm{l}_{1} \right),\mathcal{F}_{2}
\left(t,\mathrm{m}_{2}
\right)\right)-\left(\left(H\left(t,\mathrm{l}_{1}
\right)-H\left(t,\mathrm{l}_{2}^{*}
\right)\right)\mathcal{F}_1\left(t,\mathrm{m}_{1}
\right),Y_{2} \left(t,\mathrm{l}_{2} \right)\right)+\\
+\left(\left(H\left(t,\mathrm{l}_{1}
\right)-H\left(t,\mathrm{l}_{2}^{*}
\right)\right)col\left\{0,\ldots ,0,f_{1}^{\left[n\right]}
\left(t\left|\mathrm{m}_{1} \right.
\right)\right\},col\left\{0,\ldots ,0,f_{2}^{\left[n\right]}
\left(t\left|\mathrm{m}_{2} \right. \right)\right\}\right).
\end{multline}
where $\mathcal{F}_k(t,\mathrm{m}_k)$ are the analogs of
\eqref{GrindEQ__32_}.

Let us denote by $p_{j}^{k} ,q_{j}^{k} ,s_{j}^{k}$ the
coefficients of $\mathrm{l}_{j} $. Then in view of
\eqref{GrindEQ__7_}
\begin{multline}\label{GrindEQ__42_}
\left(H\left(t,\mathrm{l}_{1} \right)col\left\{0,\ldots
,0,f_{1}^{\left[n\right]} \left(t\left|\mathrm{m}_{1} \right.
\right)\right\},col\left\{0,\ldots ,0,y_{2}^{\left[n\right]}
\left(t\left|\mathrm{l}_{2} \right. \right)-y_{2}^{\left[n\right]}
\left(t\left|\mathrm{l}_{1}^{*} \right. \right)\right\}\right)=\\
=\left((p_n^1)^{-1} f_{1}^{\left[n\right]}
\left(t\left|\mathrm{m}_{1} \right. \right),y_{2}^{\left[n\right]}
\left(t\left|\mathrm{l}_{2} \right. \right)-y_{2}^{\left[n\right]}
\left(t\left|\mathrm{l}_{1}^{*} \right. \right)\right),
\end{multline}
and
\begin{multline}\label{GrindEQ__43_}
\left(col\left\{0,\ldots ,0,y_{1}^{\left[n\right]}
\left(t\left|\mathrm{l}_{1} \right. \right)-y_{1}^{\left[n\right]}
\left(t\left|\mathrm{l}_{2}^{*} \right.
\right)\right\},H\left(t,\mathrm{l}_{2} \right)col\left\{0,\ldots
,0,f_{2}^{\left[n\right]} \left(t\left|\mathrm{m}_{2} \right.
\right)\right\}\right)=\\
=\left(y_{1}^{\left[n\right]} \left(t\left|\mathrm{l}_{1} \right.
\right)-y_{1}^{\left[n\right]} \left(t\left|\mathrm{l}_{2}^{*}
\right. \right),(p_n^2)^{-1} f_{2}^{\left[n\right]}
\left(t\left|\mathrm{m}_{2} \right. \right)\right).
\end{multline}

On the another hand in view of \eqref{GrindEQ__7_},
\eqref{GrindEQ__11_} we have
\begin{multline} \label{GrindEQ__44_}
-\left(\left(H\left(t,\mathrm{l}_{1}
\right)-H\left(t,\mathrm{l}_{2}^{*} \right)\right)Y_{1}
\left(t,\mathrm{l}_{1} \right),\mathcal{F}_{2}
\left(t,\mathrm{m}_{2} \right)\right)=\\
=-\left(\left({i(p_{n}^1)^{-1} q_{n}^{1}
\mathord{\left/{\vphantom{ip_{n}^{1-1} q_{n}^{1}
2}}\right.\kern-\nulldelimiterspace} 2} -{i(p_n^{2*})^{-1}
s_{n}^{2*} \mathord{\left/{\vphantom{i(p_{n}^{2*})^{-1} s_{n}^{2*}
2}}\right.\kern-\nulldelimiterspace} 2}
\right)y_{1}^{\left(n-1\right)} +\left((p_{n}^1)^{-1}
-(p_{n}^{2*})^{-1} \right)y_{1}^{\left[n\right]}
\left(t\left|\mathrm{l}_{1} \right. \right),f_{2}^{\left[n\right]}
\left(t\left|\mathrm{m}_{2} \right.
\right)\right)=\\
=\left(\left(p_{n}^{2*}\right)^{-1} \left(y_{1}^{\left[n\right]}
\left(t\left|\mathrm{l}_{1} \right. \right)-y_{1}^{\left[n\right]}
\left(t\left|\mathrm{l}_{2}^{*} \right.
\right)\right),f_{2}^{\left[n\right]} \left(t\left|\mathrm{m}_{2}
\right. \right)\right),
\end{multline}
where the last equality is a corollary of \eqref{GrindEQ__11_} and
its following modification:
$$
(p_{n}^{1})^{-1} y_{1}^{\left[n\right]}
\left(t\left|\mathrm{l}_{1} \right. \right)=y_{1}^{\left(n\right)}
-\frac{i}{2} (p_{n}^{1})^{-1} q^1_{n} y_{1}^{\left(n-1\right)}$$

Analogously it can be proved that
\begin{multline} \label{GrindEQ__45_}
\left(\left(H\left(t,\mathrm{l}_{1}
\right)-H\left(t,\mathrm{l}_{2}^{*} \right)\right)\mathcal{F}_{1}
\left(t,\mathrm{m}_{1} \right),Y_{2} \left(t,\mathrm{l}_{2}
\right)\right)=\\
=\left(f_{1}^{\left[n\right]} \left(t\left|\mathrm{m}_{1} \right.
\right),(p_{n}^{1*})^{-1} \left(y_{2}^{\left[n\right]}
\left(t\left|\mathrm{l}_{2} \right. \right)-y_{2}^{\left[n\right]}
\left(t\left|\mathrm{l}_{1}^{*} \right. \right)\right)\right).
\end{multline}

Comparing \eqref{GrindEQ__38_}--\eqref{GrindEQ__45_} we get
\eqref{GrindEQ__36_} since the last $\int_\alpha^\beta$  in the
left-hand-side of \eqref{GrindEQ__38_} is equal to zero if $r=2n$
in view of \eqref{GrindEQ__6_}.

For $s<r=2n$ the proof of \eqref{GrindEQ__36_} easy follows from
\eqref{GrindEQ__23_}, \eqref{GrindEQ__25_}, \eqref{GrindEQ__29_},
\eqref{GrindEQ__37_}, \eqref{GrindEQ__38_} in view of footnote
\ref{foot1}.

Now let $r=2n+1$. Then the last $\int_\alpha^\beta$ in the
left-hand-side of \eqref{GrindEQ__38_} is equal to $\int _{\alpha
}^{\beta }\left(\mathrm{l}_{2n+1}^{1} -\mathrm{l}_{2n+1}^{2*}
\right)\left\{y_{1},y_2\right\} dt$. Hence the proof of
\eqref{GrindEQ__36_} for $s\le 2n<r=2n+1$ follows from
\eqref{GrindEQ__16_}, \eqref{GrindEQ__23_}, \eqref{GrindEQ__25_},
\eqref{GrindEQ__29_}, \eqref{GrindEQ__37_}, \eqref{GrindEQ__38_}.
Theorem \ref{th3} is proved.

\end{proof}

\begin{remark}\label{rem1} In view of Lemmas \ref{lm2}, \ref{lm3} all results of this item are valid if
the condition of evenness of $s$ is changed by the condition $s\le
2\left[\frac{r}{2} \right]$.
\end{remark}

\section{Characteristic operator}

We consider an operator differential equation in separable Hilbert
space $\mathcal{H}_{1} $:
\begin{equation} \label{GrindEQ__46_}
\frac{i}{2} \left(\left(Q\left(t\right)x\left(t\right)\right)^{{'}
} +Q^{*} \left(t\right)x'\left(t\right)\right)-H_{\lambda }
\left(t\right)x\left(t\right)=W_{\lambda }
\left(t\right)F\left(t\right),\quad t\in \bar{\mathcal{I}},
\end{equation}
where $Q\left(t\right),\, \left[\Re \, Q\left(t\right)\right]^{-1}
,\, H_{\lambda } \left(t\right)\in B\left(\mathcal{H}_{1}
\right),\, Q\left(t\right)\in C^{1}
\left(\bar{\mathcal{I}},B\left(\mathcal{H}_{1} \right)\right)$;
the operator function $H_{\lambda } \left(t\right)$ is continuous
in $t$ and is Nevanlinna's in $\lambda $. Namely the following
condition holds:

(\textbf{A}) The set $\mathcal{A}\supseteq \mathbb{C}\setminus
\mathbb{R}^1$ exists, any its point have a neighbourhood
independent of $t\in \bar{\mathcal{I}}$, in this neighbourhood
$H_{\lambda } \left(t\right)$ is analytic $\forall t\in
\bar{\mathcal{I}};\, \forall \lambda \in \mathcal{A}\, H_{\lambda
} \left(t\right)=H^*_{\bar \lambda}(t)\in
C\left(\bar{\mathcal{I}},B\left(\mathcal{H}_1\right)\right)$; the
weight $W_{\lambda } \left(t\right)=\Im H_{\lambda }
\left(t\right)/\Im \lambda \ge 0\left(\Im  \lambda \ne 0\right)$.

In view of \cite{Khrab5} $\forall \mu \in \mathcal{A}\bigcap
\mathbb{R}^1:\, W_{\mu } \left(t\right)=\left.\partial H_{\lambda}
\left(t\right)/\partial \lambda\right|_{\lambda=\mu} $ is Bochner
locally integrable in the uniform operator topology.

For convenience we suppose that $0\in \bar{\mathcal{I}}$ and we
denote $\Re \, Q\left(0\right)=G$.

Let $X_{\lambda } \left(t\right)$ be the operator solution of
homogeneous equation \eqref{GrindEQ__46_} satisfying the initial
condition $X_{\lambda } \left(0\right)=I$, where $I$ is an
identity operator in $\mathcal{H}_{1} $. Since
$H_\lambda(t)=H^*_{\bar\lambda}(t)$ then
\begin{gather}
\label{GrindEQ__47+_} X^*_{\bar\lambda}(t)[\Re
Q(t)]X_\lambda(t)=G,\ \lambda\in \mathcal{A}.
\end{gather}

For any $\alpha ,\, \beta \in \bar{\mathcal{I}},\, \alpha \le
\beta $ we denote $\Delta _{\lambda } \left(\alpha ,\beta
\right)=\int _{\alpha }^{\beta }X_{\lambda }^{*}
\left(t\right)W_{\lambda } \left(t\right)X_{\lambda }
\left(t\right) dt$,\\ $N=\left\{h\in \mathcal{H}_{1} \left|h\in
Ker\Delta _{\lambda } \left(\alpha ,\beta \right)\right. \forall
\alpha ,\beta \right\},P$ is the ortho-projection onto $N^{\bot }
$. $N$ is independent of $\lambda \in \mathcal{A}$ \cite{Khrab5}.

For $x\left(t\right)\in \mathcal{H}_{1} $ or $x\left(t\right)\in
B\left(\mathcal{H}_{1} \right)$ we denote
$U\left[x\left(t\right)\right]=\left(\left[\Re \,
Q\left(t\right)\right]x\left(t\right),x\left(t\right)\right)$ or
$U\left[x\left(t\right)\right]=x^{*} \left(t\right)\left[\Re \,
Q\left(t\right)\right]x\left(t\right)$ respectively.

As in \cite{Khrab4,Khrab5} we introduce the following
\begin{definition}
An analytic operator-function $M\left(\lambda \right)=M^{*}
\left(\bar{\lambda }\right)\in B\left(\mathcal{H}_{1} \right)$ of
non-real $\lambda $ is called a characteristic operator of
equation \eqref{GrindEQ__46_} on $\mathcal{I}$ (or, simply, c.o.),
if for $\Im  \lambda \ne 0$ and for any $\mathcal{H}_{1} $ -
valued vector-function $F\left(t\right)\in L_{W_{\lambda } }^{2}
\left(\mathcal{I}\right)$ with compact support the corresponding
solution $x_{\lambda } \left(t\right)$ of equation
\eqref{GrindEQ__46_} of the form
\begin{equation} \label{GrindEQ__47_}
x_{\lambda } \left(t,F\right)=\mathcal{R}_\lambda F=\int
_{\mathcal{I}}X_{\lambda } \left(t\right) \left\{M\left(\lambda
\right)-\frac{1}{2} sgn\left(s-t\right)\left(iG\right)^{-1}
\right\}X_{\bar{\lambda }}^{*} \left(s\right)W_{\lambda }
\left(s\right)F\left(s\right)ds
\end{equation}
satisfies the condition
\begin{gather}\label{GrindEQ__47++_}
\left(\Im  \lambda \right)\mathop{\lim }\limits_{\left(\alpha
,\beta \right)\uparrow \mathcal{I}} \left(U\left[x_{\lambda }
\left(\beta ,F\right)\right]-U\left[x_{\lambda } \left(\alpha
,F\right)\right]\right)\le 0\, \, \, \left(\Im  \lambda \ne
0\right).
\end{gather}
\end{definition}

Let us note that in \cite{Khrab5} c.o. was defined if
$Q(t)=Q^*(t)$. Our case is equivalent to this one since equation
\eqref{GrindEQ__46_} coincides with equation of
\eqref{GrindEQ__46_} type with $\Re Q(t)$ instead of $Q(t)$ and
with $H_\lambda(t)-{1\over 2}\Im Q'(t)$ instead of $H_\lambda(t)$.

The properties of c.o. and sufficient conditions of the c.o.'s
existence are obtained in \cite{Khrab4,Khrab5}.

In the case $\mathrm{dim}\mathcal{H}_1<\infty$,
$Q(t)=\mathcal{J}=\mathcal{J}^*=\mathcal{J}^{-1}$, $-\infty<a=c$
the description of c.o.'s was obtained in \cite{Orlov} (the
results of \cite{Orlov} were specified and supplemented in
\cite{Khrab+}). In the case $\mathrm{dim} \mathcal{H}_1=\infty$
and $\mathcal{I}$ is finite the description of c.o.'s was obtained
in \cite{Khrab5}. These descriptions are obtained under the
condition that
\begin{gather}\label{star8}
\exists\lambda_0\in \mathcal{A},\ [\alpha,\beta]\subseteq
\overline{\mathcal{I}}:\ \Delta_{\lambda_0}(\alpha,\beta)\gg 0.
\end{gather}

\begin{definition}\label{def+}
\cite{Khrab4,Khrab5} Let  $M\left(\lambda \right)$ be the c.o. of
equation (\ref{GrindEQ__46_}) on  $\mathcal{I} $. We say that the
corresponding condition (\ref{GrindEQ__47++_}) is separated for
nonreal $\lambda =\mu _{0} $ if for any $\mathcal{H}_1$-valued
vector function $f\left(t\right)\in L_{W_{\mu_{0} }(t) }^{2}
\left(\mathcal{I} \right)$ with compact support the following
inequalities holds simultaneously for the solution $x_{\mu _{0} }
\left(t\right)$ (\ref{GrindEQ__47_}) of equation
(\ref{GrindEQ__46_}):
\begin{gather}
\label{12} \displaystyle \lim\limits_{\alpha\downarrow a}\Im\mu
_{0} U\left[x_{\mu _{0} } \left(\alpha\right)\right]\ge 0,\quad
\mathop{\lim }\limits_{\beta \uparrow b} \Im\mu _{0} U\left[x_{\mu
_{0} } \left(\beta \right)\right]\le 0.
\end{gather}
\end{definition}

\begin{theorem}\cite{Khrab4,Khrab5} (see also \cite{RBKholkin})\label{th++}
Let $M\left(\lambda \right)$ be the c.m. of equation
(\ref{GrindEQ__46_}). We represent  $M\left(\lambda \right)$ in
the form
\begin{gather}
\label{13} \displaystyle M\left(\lambda \right)=\left(\mathcal{P}
\left(\lambda \right)-\frac{1}{2} I\right)\left(iG\right)^{-1}.
\end{gather}

Then the condition (\ref{GrindEQ__47++_}) corresponding to
$M\left(\lambda \right)$ is separated for $\lambda =\mu _{0} $ if
and only if the operator $\mathcal{P} \left(\mu_{0} \right)$ is
the projection, i.e.
\begin{gather}
\label{14} \displaystyle \mathcal{P} \left(\mu _{0}
\right)=\mathcal{P} ^{2} \left(\mu _{0} \right).
\end{gather}
\end{theorem}

\begin{definition}\cite{Khrab4,Khrab5}\label{def++}
If the operator-function $M\left(\lambda \right)$ of the form
(\ref{13}) is the c.o. of equation (\ref{GrindEQ__46_}) on
$\mathcal{I} $ and, moreover, $\mathcal{P} \left(\lambda
\right)=\mathcal{P}^{2} \left(\lambda \right)$, then
$\mathcal{P}\left(\lambda \right)$ is called a characteristic
projection (c.p) of equation (\ref{GrindEQ__46_}) on $\mathcal{I}$
(or, simply, c.p).
\end{definition}

The properties of c.p.'s and sufficient conditions for their
existence are obtained in \cite{Khrab5}. Also \cite{Khrab5}
contains the description of c.p.'s and abstract an analogue of
Theorem \ref{th++}.

The following statement gives necessary and sufficient conditions
for existence of c.o., which corresponds to such separated
boundary conditions that corresponding boundary condition in
regular point is self-adjoint. This statement follows from Theorem
\ref{th++}.

Let us denote $\mathcal{H}_+$ ($\mathcal{H}_-$) the invariant
subspace of operator $G$, which corresponds to positive (negative)
part of $\sigma(G)$.

\begin{theorem}\label{th+}
Let $-\infty <a$. If $P=I$ then for existence of c.o. $M(\lambda)$
of equation \eqref{GrindEQ__46_} on $(a,b)$ such that
\begin{gather}\label{star5}
\exists\mu_0 \in \mathbb{C}\setminus \mathbb{R}^1:\
U[x_{\mu_0}(a,F)]=U[x_{\bar{\mu_0}}(a,F)]=0
\end{gather}
(and therefore condition \eqref{GrindEQ__47++_} is separated on
$\lambda=\mu_0$, $\lambda=\bar{\mu}_0$) it is necessary and
sufficient that
\begin{gather}\label{star6}
\mathrm{dim} \mathcal{H}_+=\mathrm{dim} \mathcal{H}_-
\end{gather}
(in \eqref{star5} $x_\lambda(t,F)$ is a solution
\eqref{GrindEQ__47_} of \eqref{GrindEQ__46_} which corresponds to
c.o. $M(\lambda)$, $L^2_{w_{\mu_0}(t)}(a,b)\ni F=F(t)$ is any
$\mathcal{H}_1$-valued vector-function with compact support). If
condition \eqref{star8} holds  then condition \eqref{star5} is
also sufficient for the existence of such c.o.
\end{theorem}

\begin{proof}
Necessity. Since $P=I$ we obtain
\begin{gather}\label{plus1}
U[X_{\mu_0}(a)(I-\mathcal{P}(\mu_0))]=
U[X_{\bar{\mu}_0}(a)(I-\mathcal{P}(\bar{\mu}_0))]=0\end{gather} in
view of the proof of n$^\circ 2^\circ$ of Theorem 1.1 from
\cite{Khrab5}.

Let for definiteness $\Im\mu_0>0$. Then in view of Theorem 2.4 and
formula (1.69) from \cite{Khrab5}, \eqref{14}, \eqref{plus1} and
the fact that
\begin{gather}\label{plus2}
\Im\lambda(X^*_\lambda(a)\Re Q(a)X_\lambda(a)-G)\leq 0,\
\lambda\in\mathcal{A}
\end{gather}
we conclude that $X_{\mu_0}(a)(I-\mathcal{P}(\mu_0))\mathcal{H}_1$
and $X_{\bar{\mu}_0}(a)(I-\mathcal{P}(\bar{\mu}_0))\mathcal{H}_1$
are correspondingly maximal $\Re Q(a)$-nonnegative and maximal
$\Re Q(a)$-nonpositive subspaces which are $\Re Q(a)$-neutral and
which are $\Re Q(a)$-orthogonal in view of Remark 3.2 from
\cite{Khrab5}, Theorem \ref{th++} and \eqref{GrindEQ__47+_}. Hence
$$
\left(X_{\mu_0}(a)(I-\mathcal{P}(\mu_0))\mathcal{H}_1\right)^{[\perp]}=
X_{\bar{\mu}_0}(a)(I-\mathcal{P}(\bar{\mu}_0))\mathcal{H}_1
$$
in view of \cite[p.73]{Azizov} (here by $[\perp]$ we denote $\Re
Q(a)$-orthogonal complement). Therefore
$X_{\mu_0}(a)(I-\mathcal{P}(\mu_0))\mathcal{H}_1$ is hypermaximal
$\Re Q(a)$-neutral subspace in view of \cite[p.43]{Azizov}. Thus
we obtain that in view of \cite[p.42]{Azizov} that $\mathrm{dim }
\mathcal{H}_+(a)=\mathrm{dim} \mathcal{H}_-(a)$, where
$\mathcal{H}_\pm(a)$ are analogs of $\mathcal{H}_\pm$ for $\Re
Q(a)$. In view of \eqref{plus2}
$X_{\mu_0}^{-1}(a)\mathcal{H}_+(a)$ and
$X_{\bar\mu_0}^{-1}(a)\mathcal{H}_-(a)$ are correspondingly
maximal uniformly $G$-positive and maximal uniformly $G$-negative
subspaces. Therefore $\mathcal{H}_1$ is equal to the direct and
$G$-orthogonal sum of these subspaces in view of
\eqref{GrindEQ__47+_} and \cite[p.75]{Azizov}. Hence we obtain
\eqref{star6} in view of the law of inertia \cite[p.54]{Azizov}.

Sufficiency follows from Theorem 4.4. from \cite{Khrab5}. Theorem
is proved.
\end{proof}

It is obvious that in Theorem \ref{th+} the point $a$ can be
replaced by the point $b$ if $b<\infty$, but cannot be replaced by
the point $b$ if $b=\infty$ as the example of operator $id/dt$ on
the semi-axis shows. Also this example shows that condition
\eqref{star5} is not necessary for the fulfilment of the condition
$U[x_{\mu_0}(a,F)]=0$ only.

In the case of self-adjoint boundary conditions the analogue of
Theorem \ref{th+} for regular differential operators in space of
vector-functions was proved in \cite{RB} (see also
\cite{RBKholkin}). For finite canonical systems depending on
spectral parameter in a linear manner such analogue was proved in
\cite{Mogil}. These analogs were obtained in a different way
comparing with Theorem \ref{th+}.

Let us consider operator differential expression $l_{\lambda } $
of \eqref{GrindEQ__51+_} type with coefficients $p_{j} =p_{j}
\left(t,\lambda \right)$, $ q_{j} =q_{j} \left(t,\lambda
\right),\, \, s_{j} =s_{j} \left(t,\lambda \right)$ and of order
$r$. Let $-l_{\lambda } $ depends on $\lambda $ in Nevanlinna
manner. Namely, from now on the following condition holds:

(\textbf{B}) The set $\mathcal{B}\supseteq {\mathbb{C}\setminus
\mathbb{R}}^1 $ exists, any its points have a neighbourhood
independent on $t\in \bar{\mathcal{I}}$, in this neighbourhood
coefficients $p_{j} =p_{j} \left(t,\lambda \right),\, \, q_{j}
=q_{j} \left(t,\lambda \right),\, \, s=s_{j} \left(t,\lambda
\right)$ of the expression $l_{\lambda } $ are analytic $\forall
t\in \bar{\mathcal{I}}$; $\forall \lambda \in \mathcal{B}{\rm ,}\,
\, p_{j} \left(t,\lambda \right)$, $q_{j} \left(t,\lambda
\right)$, $s_{j} \left(t,\lambda \right)\in C^{j}
\left(\bar{\mathcal{I}},B\left(\mathcal{H}\right)\right)$ and
\begin{equation} \label{GrindEQ__48_}
p_{n}^{-1} \left(t,\lambda \right)\in
B\left(\mathcal{H}\right)\left(r=2n\right),\, \left(q_{n+1}
\left(t,\lambda \right)+s_{n+1} \left(t,\lambda
\right)\right)^{-1} \in B\left(\mathcal{H}\right)\,
\left(r=2n+1\right),\ t\in\bar{\mathcal{I}};
\end{equation}
these coefficients satisfy the following conditions
\begin{gather} \label{GrindEQ__49_}
p_{j} \left(t,\lambda \right)=p_{j}^{*} \left(t,\bar{\lambda
}\right),\, q_{j} \left(t,\lambda \right)=s_{j}^{*}
\left(t,\bar{\lambda }\right),\ \lambda \in\mathcal{B}
\end{gather}
$\big(\eqref{GrindEQ__49_}\ \Longleftrightarrow\ l_{\lambda }
=l_{\bar{\lambda }}^{*}\ \underset{\text{in view of
}\eqref{GrindEQ__17_}}\Longleftrightarrow\
H(t,l_\lambda)=H(t,l^*_{\bar\lambda}),\lambda \in
\mathcal{B}\big)$;
\begin{multline} \label{GrindEQ__50_}
\forall h_{0} ,\ldots ,h_{\left[\frac{r+1}{2} \right]} \in
\mathcal{H}:\\ \frac{\Im \left(\sum\limits
_{j=0}^{\left[r/2\right]}\left(p_{j} \left(t,\lambda \right)h_{j}
,h_{j} \right) +\frac{i}{2} \sum\limits
_{j=1}^{\left[\frac{r+1}{2} \right]}\left\{\left(s_{j}
\left(t,\lambda \right)h_{j} ,h_{j-1} \right)-\left(q_{j}
\left(t,\lambda \right)h_{j-1} ,h_{j} \right)\right\} \right)}{\Im
\lambda } \le 0,\\ t\in \bar{\mathcal{I}},\, \, \, \Im \lambda \ne
0.
\end{multline}

Therefore the order of expression $\Im l_{\lambda } $ is even and
therefore if $r=2n+1$ is odd, then $q_{m+1} ,\, s_{m+1} $ are
independent on $\lambda $ and $s_{n+1} =q_{n+1}^{*} $.

Condition \eqref{GrindEQ__50_} is equivalent to the condition:
${\left(\Im l_{\lambda } \right)\left\{f,f\right\}
\mathord{\left/{\vphantom{\left(\Im l_{\lambda }
\right)\left\{f,f\right\} \Im \lambda
}}\right.\kern-\nulldelimiterspace} \Im \lambda } \le 0,\, \, t\in
\bar{\mathcal{I}},\, \, \Im \lambda \ne 0$.

Hence $W\left(t,l_{\lambda } ,-\frac{\Im l_{\lambda } }{\Im
\lambda } \right)=\frac{\Im H\left(t,l_{\lambda } \right)}{\Im
\lambda } \ge 0,\, \, t\in \bar{\mathcal{I}},\, \, \, \Im \lambda
\ne 0$ due to Lemma \ref{lm1} and Theorem \ref{th2} and therefore
$H\left(t,l_{\lambda } \right)$ satisfy condition \textbf{(A)}
with $\mathcal{A}=\mathcal{B}$. Therefore $\forall\mu\in
\mathcal{B}\cap \mathbb{R}^1$ $W(t,l_\mu,-{\Im l_\mu\over\Im\mu
})=\left.{\partial
H(t,l_\lambda)\over\partial\lambda}\right|_{\lambda=\mu}$ is
Bochner locally integrable in uniform operator topology. Here in
view of \eqref{GrindEQ__7_}, \eqref{GrindEQ__15_} $\forall \mu \in
\mathcal{B}\bigcap \mathbb{R}^{1}\ \exists \frac{\Im l_{\mu }
}{\Im \mu } \mathop{=}\limits^{def} \frac{\Im l_{\mu +i0} }{\Im
\left(\mu +i0\right)} =\left.\frac{\partial l_{\lambda }
}{\partial \lambda }\right|_{\lambda=\mu} $, where the
coefficients ${\partial p_j(t,\mu)\over\partial\lambda}$,
${\partial q_j(t,\mu)\over\partial\lambda}$, ${\partial
s_j(t,\mu)\over\partial\lambda}$ of expression ${\partial l_{\mu }
\mathord{\left/{\vphantom{\partial l_{\mu }
\partial \mu }}\right.\kern-\nulldelimiterspace} \partial \mu } $
are Bochner locally integrable in the uniform operator topology.

Let us consider in $\mathcal{H}_1=\mathcal{H}^r$ the equation
\begin{equation} \label{GrindEQ__51_}
\frac{i}{2} \left(\left(Q\left(t,l_{\lambda }
\right)\vec{y}\left(t\right)\right)^{{'} } +Q^{*}
\left(t,l_{\lambda }
\right)\vec{y}\hspace{2px}'\left(t\right)\right)-H\left(t,l_{\lambda
} \right)\vec{y}\left(t\right)=W\left(t,l_{\lambda } -\frac{\Im
l_{\lambda } }{\Im \lambda } \right)F\left(t\right).
\end{equation}
This equation is an equation of \eqref{GrindEQ__46_} type due to
\eqref{GrindEQ__17_} and Lemma \ref{lm1}. Equation
\eqref{GrindEQ__5_} is equivalent to equation \eqref{GrindEQ__51_}
with $F\left(t\right)=F\left(t,l_{\bar{\lambda }} ,-\frac{\Im
l_{\lambda } }{\Im \lambda } \right)$ due to Theorem \ref{th1}.

\begin{definition} Every characteristic operator of equation
\eqref{GrindEQ__51_} corresponding to the equation
\eqref{GrindEQ__5_} is said to be a characteristic operator of
equation \eqref{GrindEQ__5_} on $\mathcal{I}$ (or simply c.o.).
\end{definition}

Let $m$ be the same as in $n^{\circ } 1$ differential expression
of even order $s\le r$ with operator coefficients $\tilde{p}_{j}
\left(t\right)=\tilde{p}_{j}^{*} \left(t\right),\, \,
\tilde{q}_{j} \left(t\right),\, \tilde{s}_{j}
\left(t\right)=\tilde{q}_{j}^{*} \left(t\right)$ that are
independent on $\lambda $. Let
\begin{multline}
\label{GrindEQ__52_} \forall h_{0} ,\, \ldots ,\,
h_{\left[\frac{r+1}{2} \right]} \in \mathcal{H}:\, \, 0\le
\sum\limits _{j=0}^{{s \mathord{\left/{\vphantom{s
2}}\right.\kern-\nulldelimiterspace} 2} }\left(\tilde{p}_{j}
\left(t\right)h_{j} ,h_{j} \right) +{\Im}\sum\limits _{j=1}^{{s
\mathord{\left/{\vphantom{s 2}}\right.\kern-\nulldelimiterspace}
2} }\left(\tilde{q}_{j}
\left(t\right)h_{j-1} ,\, h_{j} \right) \le\\
 \le -\frac{\Im \left(\sum\limits
_{j=0}^{\left[{r \mathord{\left/{\vphantom{r
2}}\right.\kern-\nulldelimiterspace} 2} \right]}\left(p_{j}
\left(t,\lambda \right)h_{j} ,h_{j} \right) +\frac{i}{2}
\sum\limits _{j=1}^{\left[\frac{r+1}{2} \right]}\left(\left(s_{j}
\left(t,\lambda \right)h_{j} ,h_{j-1} \right)-\left(q_{j}
\left(t,\lambda \right)h_{j-1} ,h_{j} \right)\right) \right)}{\Im
\lambda },\\ t\in \bar{\mathcal{I}},\, \, \, \Im \lambda \ne 0.
\end{multline}
Condition \eqref{GrindEQ__52_} is equivalent to the condition:
$0\le m\left\{f,f\right\}\le {\Im l_{\lambda } \left\{f,f\right\}
\mathord{\left/{\vphantom{\Im l_{\lambda } \left\{f,f\right\} \Im
\lambda }}\right.\kern-\nulldelimiterspace} \Im \lambda } $, $t\in
\bar{\mathcal{I}}$, $\Im \lambda \ne 0$. Hence
\begin{equation} \label{GrindEQ__53_}
0\le W\left(t,l_{\lambda } ,m\right)\le W\left(t,l_{\lambda }
,-\frac{\Im l_{\lambda } }{\Im \lambda } \right)=\frac{\Im
H\left(t,l_{\lambda } \right)}{\Im \lambda } \quad t\in
\bar{\mathcal{I}},\, \, \, \Im \lambda \ne 0
\end{equation}
due to Theorem \ref{th2} and Lemma \ref{lm1}.

In view of Theorem \ref{th1} equation \eqref{GrindEQ__1_} is
equivalent to the equation
\begin{equation} \label{GrindEQ__54_}
\frac{i}{2} \left(\left(Q\left(t,l_{\lambda }
\right)\vec{y}\left(t\right)\right)^{{'} } +Q^{*}
\left(t,l_{\lambda }
\right)\vec{y}\hspace{2px}'\left(t\right)\right)-H\left(t,l_{\lambda
} \right)\vec{y}\left(t\right)=W\left(t,l_{\bar{\lambda }}
,m\right)F\left(t,l_{\bar{\lambda }} ,m\right),
\end{equation}
where $Q\left(t,l_{\lambda } \right),\, H\left(t,l_{\lambda }
\right)$ are defined by \eqref{GrindEQ__6_}, \eqref{GrindEQ__7_},
\eqref{GrindEQ__14_}, \eqref{GrindEQ__15_} with $l_{\lambda } $
instead of $l$ and $W\left(t,l_{\bar{\lambda }} ,m\right)$
$F\left(t,l_{\bar{\lambda }} ,m\right)$ are defined by
\eqref{GrindEQ__8_}, \eqref{GrindEQ__16_} \eqref{GrindEQ__23_}
with $l_{\bar{\lambda }} $ instead of $l$ and
$\vec{y}\left(t\right)=\vec{y}\left(t,l_{\lambda } ,m,f\right)$ is
defined by \eqref{GrindEQ__25_} with $l_{\lambda } $ instead of
$l$.

In some cases we will suppose additionally that

\noindent $\exists \lambda _{0} \in \mathcal{B};\, \, \alpha
,\beta \in \bar{\mathcal{I}},\, \, 0\in \left[\alpha ,\beta
\right]$, the number $\delta >0$:
\begin{equation} \label{GrindEQ__55_}
-\int _{\alpha }^{\beta }\left(\frac{\Im l_{\lambda _{0} } }{\Im
\lambda _{0} } \right)\left\{y\left(t,\lambda _{0}
\right),y\left(t,\lambda _{0} \right)\right\} \, dt\ge \delta
\left\| P\vec{y}\left(0,l_{\lambda _{0} } ,m,0\right)\right\| ^{2}
\end{equation}
for any solution $y\left(t,\lambda _{0} \right)$ of
\eqref{GrindEQ__5_} as $\lambda =\lambda _{0} ,\, \, f=0$. In view
of Theorem \ref{th2} this condition is equivalent to the fact that
for the equation \eqref{GrindEQ__51_} with $F\left(t\right)=0$

$\exists \lambda _{0} \in \mathcal{A}=\mathcal{B};\, \, \alpha
,\beta \in \bar{\mathcal{I}},\, \, 0\in \left[\alpha ,\beta
\right]$, the number $\delta >0$:
\begin{equation} \label{GrindEQ__56_}
\left(\Delta _{\lambda _{0} } \left(\alpha ,\beta
\right)g,g\right)\ge \delta \left\| Pg\right\| ^{2} ,\quad g\in
\mathcal{H}^{r} .
\end{equation}

Therefore in view of \cite{Khrab6} the fulfillment of
\eqref{GrindEQ__55_} imply its fulfillment with $\delta
\left(\lambda \right)>0$ instead of $\delta $ for all $\lambda \in
\mathcal{B}$.

\begin{lemma}\label{lm6}
Let $M\left(\lambda \right)$ be a c.o. of equation
\eqref{GrindEQ__5_}, for which condition \eqref{GrindEQ__55_}
holds with $P=I_{r} $, if $\mathcal{I}$ is infinite. Let
$\Im\lambda\not= 0$, $\mathcal{H}^{r} $-valued $F\left(t\right)\in
L_{W\left(t,l_{\bar{\lambda }} ,m\right)}^{2}
\left(\mathcal{I}\right)$ (in particular one can set
$F\left(t\right)=F\left(t,l_{\bar{\lambda }} ,m\right)$, where
$f\left(t\right)\in C^{s}
\left(\mathcal{I},\mathcal{H}\right),m\left[f,f\right]<\infty $).
Then the solution
\begin{equation} \label{GrindEQ__57_}
x_{\lambda } \left(t,F\right)=\mathcal{R}_{\lambda } F=\int
_{\mathcal{I}}X_{\lambda } \left(t\right)\left\{M\left(\lambda
\right)-\frac{1}{2} sgn\left(s-t\right)\left(iG\right)^{-1}
\right\}X_{\bar{\lambda }}^{*}
\left(s\right)W\left(s,l_{\bar{\lambda }} ,m\right)F\left(s\right)
ds
\end{equation}
of equation \eqref{GrindEQ__54_} with $F\left(t\right)$ instead
$F\left(t,l_{\bar{\lambda }} ,m\right)$, satisfies the following
inequality
\begin{equation} \label{GrindEQ__58_}
\left\| \mathcal{R}_{\lambda } F\right\|_{L_{W\left(t,l_{\lambda }
,-\frac{\Im l_{\lambda } }{\Im \lambda } \right)}^{2}
\left(\mathcal{I}\right)}^{2} \le \Im \left(\mathcal{R}_{\lambda }
F,F\right)_{L_{W\left(t,l_{\bar{\lambda }},m \right)}^{2}
\left(\mathcal{I}\right)}/\Im\lambda,\ \Im\lambda\not= 0,
\end{equation}
where $X_{\lambda } \left(t\right)$ is the operator solution of
homogeneous equation \eqref{GrindEQ__54_} such that $X_{\lambda }
\left(0\right)=I_{r} $, $G=\mathcal{R}Q\left(0,l_{\lambda }
\right)$; integral \eqref{GrindEQ__57_} converges strongly if
$\mathcal{I}$ is infinite.
\end{lemma}

\begin{proof} Let us denote
$$
K\left(t,s,\lambda \right)=X_{\lambda }
\left(t\right)\left\{M\left(\lambda \right)-\frac{1}{2}
sgn\left(s-t\right)\left(iG\right)^{-1} \right\}X_{\bar{\lambda
}}^{*} \left(s\right)$$

If \eqref{GrindEQ__55_} holds with $P=I_{r}$ if $\mathcal{I}$ is
infinite, then in view of \eqref{GrindEQ__53_} and
\cite[p.166]{Khrab5} there exists a locally bounded on $s$ and on
$\lambda $ constant $k\left(s,\lambda \right)$ such that
\begin{equation} \label{GrindEQ__59_}
\forall h\in \mathcal{H}^{r} :\quad \left\| K\left(t,s,\lambda
\right)h\right\| _{L_{W\left(t,{l}_{\bar\lambda } ,m\right)}^{2}
\left(\mathcal{I}\right)} \le k\left(s,\lambda \right)\left\|
h\right\|
\end{equation}

Hence integral \eqref{GrindEQ__57_} converges strongly if
$\mathcal{I}$ is in finite.

Let $F\left(t\right)$ have compact support and
$supp F\left(t\right)\subseteq \left[\alpha ,\beta \right]$.

Then in view of \eqref{GrindEQ__36_}
\begin{multline} \label{GrindEQ__60_}
\int _{\alpha }^{\beta }\left(W\left(t,l_{\lambda } ,-\frac{\Im
l_{\lambda } }{\Im \lambda } \right)\mathcal{R}_{\lambda }
F,\mathcal{R} _{\lambda } F\right) dt-\frac{\Im \int _{\alpha
}^{\beta }\left(W\left(t,l_{\bar{\lambda }}
,m\right)\mathcal{R}_{\lambda } F, F\right) dt}{\Im \lambda } =
\\
=\left.\frac{1}{2} \, \frac{\left(\Re Q\left(t,l_{\lambda }
\right)\mathcal{R} _{\lambda } F,\mathcal{R}_\lambda F\right)}{\Im
\lambda } \right|_\alpha^\beta \le 0
\end{multline}
where the last inequality is a corollary of $n^{\circ } 2$.
Theorem 1.1 from \cite[p.162]{Khrab5} and the following

\begin{lemma}\label{lm7}
Let $\mathcal{F}_{\lambda } $ is the set of $\mathcal{H}^{r}
$-valued function from $L_{W\left(t,l_{\bar{\lambda }}
,m\right)}^{2} \left(\alpha ,\beta \right)$,
\begin{gather}\label{GrindEQ__60+_} I_{\lambda }
\left(\alpha ,\beta \right)F=\int _{\alpha }^{\beta
}X_{\bar{\lambda }}^{*} \left(t\right)W\left(t,l_{\bar{\lambda }}
,m\right)F\left(t\right)dt ,\quad F\left(t\right)\in
\mathcal{F}_{\lambda }
\end{gather}
Then
\begin{equation} \label{GrindEQ__61-_}
I_{\lambda } \left(\alpha ,\beta \right)F\in \left\{Ker\int
_{\alpha }^{\beta }X_{\bar{\lambda }}^{*}
\left(t\right)W\left(t,l_{\bar{\lambda }} ,m\right)X_{\bar{\lambda
}} \left(t\right)dt \right\}^{\bot } \subseteq N^{\bot } .
\end{equation}
\end{lemma}

\begin{proof} Let $h\in Ker\int _{\alpha }^{\beta }X_{\bar{\lambda
}}^{*} \left(t\right)W\left(t,l_{\bar{\lambda }}
,m\right)X_{\bar{\lambda }} \left(t\right)dt $ $\Rightarrow
W\left(t,l_{\bar{\lambda }} ,m\right)X_{\bar{\lambda }}
\left(t\right)h=0$ $\Rightarrow I_{\lambda } \left(\alpha ,\beta
\right)F \bot h$. The second enclosure in \eqref{GrindEQ__61-_} is
a corollary of condition \eqref{GrindEQ__53_}. Lemma \ref{lm7} and
inequality \eqref{GrindEQ__60_} are proved.
\end{proof}

Thus Lemma \ref{lm6} is proved if $\mathcal{I}$ is finite. Let us
prove it for infinite $\mathcal{I}$. Let finite intervals
$\left(\alpha _{n} ,\beta _{n} \right)\uparrow \mathcal{I},\quad
F_{n} =\chi _{n} F$, where $\chi _{n} $ - is a characteristic
function of $\left(\alpha _{n} ,\beta _{n} \right)$. If
$\left(\alpha ,\beta \right)\subseteq \left(\alpha _{n} ,\beta
_{n} \right)$ then
$$\left\| \mathcal{R}_{\lambda } F_{n} \right\|
_{L_{W\left(t,l_{\lambda } ,-\frac{\Im l_{\lambda } }{\Im \lambda
} \right)}^{2} } \left(\alpha ,\beta \right)\le \frac{\left\|
F\right\| _{L_{W\left(t,l_{\bar{\lambda }} ,m\right)}^{2}
\left(\mathcal{I}\right)} }{\left|\Im \lambda \right|}$$
in view
of \eqref{GrindEQ__60_}, \eqref{GrindEQ__53_}. But local uniformly
on $t$: $\left(\mathcal{R}_{\lambda } F_{n}
\right)\left(t\right)\to \left(\mathcal{R}_{\lambda }
F\right)\left(t\right)$, in view of \eqref{GrindEQ__59_}. Hence
\begin{equation} \label{GrindEQ__61_}
\left\| \mathcal{R}_{\lambda } F\right\| _{L_{W\left(t,l_{\lambda
} ,-\frac{\Im l_{\lambda } }{\Im \lambda } \right)}^{2}
\left(\alpha ,\beta \right)} \le \frac{\left\| F\right\|
_{L_{W\left(t,l_{\bar{\lambda }} ,m\right)}^{2}
\left(\mathcal{I}\right)} }{\left|\Im \lambda \right|} .
\end{equation}
for any finite $\left(\alpha ,\beta \right)$. Hence
\eqref{GrindEQ__61_} holds with $\mathcal{I}$ instead of
$\left(\alpha ,\beta \right)$. In view of last fact
$\mathcal{R}_{\lambda } F_{n} \to \mathcal{R}_{\lambda } F$ in
$L_{W\left(t,l_{\lambda } ,-\frac{\Im l_{\lambda } }{\Im \lambda }
\right)}^{2} \left(\mathcal{I}\right)$. Hence \eqref{GrindEQ__58_}
is proved since it is proved for $F_{n} $. Lemma \ref{lm6} is
proved.
\end{proof}

Let us notice that in view of \cite{Khrab5} $PM\left(\lambda
\right)P$ is a c.o. of equation \eqref{GrindEQ__5_}, if
$M\left(\lambda \right)$ is its c.o. Obviously the closures of
operators $\mathcal{R}_{\lambda } $ corresponding to c.o.s
$M\left(\lambda \right)$ and $P M\left(\lambda \right) P$ are
equal in $B\left(L_{W\left(t,l_{\bar{\lambda }} ,m\right)}^{2}
\left(\mathcal{I}\right),\, L_{W\left(t,l_{\lambda },-\frac{\Im
l_{\lambda } }{\Im \lambda }\right)}
\left(\mathcal{I}\right)\right)$.

Let us notice what in view of \eqref{GrindEQ__52_} $l_{\lambda } $
can be a represented in form \eqref{GrindEQ__2_} where
\begin{equation} \label{GrindEQ__561_}
l=\Re l_{i} , n_{\lambda } =l_{\lambda } -l-\lambda m; {\Im
n_{\lambda } \left\{f,f\right\}
\mathord{\left/{\vphantom{n_{\lambda } \left\{f,f\right\}
\mathcal{I}\lambda \ge 0}}\right.\kern-\nulldelimiterspace}
\Im\lambda \ge 0} ,t\in \bar{\mathcal{I}},\, {\Im}\lambda \ne 0.
\end{equation}

From now on we suppose that $l_{\lambda } $ has a representation
in \eqref{GrindEQ__2_}, \eqref{GrindEQ__561_} and therefore the
order of $n_\lambda$ is even.

\section{Main results}

We consider pre-Hilbert spaces $\mathop{H}\limits^{\circ } $ and
$H$ of vector-functions $y\left(t\right)\in C_{0}^{s}
\left(\bar{\mathcal{I}},\mathcal{H}\right)$ and
$y\left(t\right)\in C^{s}
\left(\bar{\mathcal{I}},\mathcal{H}\right),\,
m\left[y\left(t\right),\, y\left(t\right)\right]<\infty $
correspondingly with a scalar product

\[\left(f\left(t\right),\, g\left(t\right)\right)_{m} =m\left[f\left(t\right),\, g\left(t\right)\right],\]
where $m\left[f,\, g\right]$ is defined by \eqref{GrindEQ__27_}
with expression $m$ from condition \eqref{GrindEQ__52_} instead of
$L$.

The null-elements of $H$ is given by

\begin{proposition}\label{prop1}
Let $f\left(t\right)\in H$. Then
\[m\left[f,f\right]=0\Leftrightarrow m\left[f\right]=f^{\left[s\right]} \left(t\right)=\, ...\, =f^{\left[{s \mathord{\left/{\vphantom{s 2}}\right.\kern-\nulldelimiterspace} 2} \right]} \left(t\right)=0,\, \, \, t\in \bar{\mathcal{I}}.\]
\end{proposition}
\begin{proof}
Let us denote by $m\left(t\right)\in B\left(\mathcal{H}^{n+1}
\right)$ the operator matrix corresponding to the quadratic form
in left side of \eqref{GrindEQ__52_}. Since $m\left(t\right)\ge 0$
one has
$$
m\left[f,f\right]=0\Leftrightarrow
m\left(t\right)col\left\{f\left(t\right),\ldots
,f^{\left(s/2\right)} ,0,\ldots ,0\right\}=0\Leftrightarrow
f^{\left[s\right]} \left(t\right)=\ldots =f^{\left[s/2\right]} =0
$$
\end{proof}

\begin{example}\label{ex1} Let $\dim \mathcal{H}=1,\, \, s=2,\, \tilde{p}_{1}
\left(t\right)>0,\, \left|\tilde{q}_{1}
\left(t\right)\right|^2=4{\tilde{p}_{1}
\left(t\right)\tilde{p}_{0} \left(t\right)} $. Then for expression
$m$ the first inequality \eqref{GrindEQ__52_} holds and
$m\left\{f_{0} ,f_{0} \right\}\equiv 0$ for $f_{0}
\left(t\right)=\exp \left({i\over 2}\int_0^t \tilde q_1/\tilde p_1
dt\right)\ne 0$ in view of Proposition \ref{prop1}.
\end{example}

By $\mathop{L_{m}^{2} }\limits^{\circ } \left(\mathcal{I}\right)$
and $L_{m}^{2} \left(\mathcal{I}\right)$ we denote the completions
of spaces $\mathop{H}\limits^{\circ } $ and $H$ in the norm
$\left\| \, \bullet \, \right\| _{m} =\sqrt{\left(\, \bullet ,\,
\bullet \right)_{m} } $ correspondingly. By
$\mathop{P}\limits^{\circ } $ we denote the orthoprojection in
$\mathop{L_{m}^{2} }\limits^{} \left(\mathcal{I}\right)$ onto
$\mathop{L_{m}^{2} }\limits^{\circ } \left(\mathcal{I}\right)$.

\begin{theorem}\label{th4}
Let $M\left(\lambda \right)$ is a c.o. of equation
\eqref{GrindEQ__5_}, for which the condition \eqref{GrindEQ__55_}
with $P=I_{r} $ holds if $\mathcal{I}$ is infinite. Let
$\Im\lambda\not= 0$, $f\left(t\right)\in H$ and
\begin{multline} \label{GrindEQ__62_}
col\left\{y_{j} \left(t,\lambda ,f\right)\right\}=\\=\int
_{\mathcal{I}}X_{\lambda } \left(t\right) \left\{M\left(\lambda
\right)-\frac{1}{2} sgn\left(s-t\right)\left(iG\right)^{-1}
\right\}X_{\bar{\lambda }}^{*}
\left(s\right)W\left(s,l_{\bar{\lambda }}
,m\right)F\left(s,l_{\bar{\lambda }} ,m\right)\, ds,\, y_{j} \in
\mathcal{H}
\end{multline}
be a solution of equation \eqref{GrindEQ__54_}, that corresponds
to equation \eqref{GrindEQ__1_}, where $X_{\lambda }
\left(t\right)$ is the operator solution of homogeneons equation
\eqref{GrindEQ__54_} such that $X_{\lambda } \left(0\right)=I_{r}
;\, \, G=\Re Q\left(0,l_{\lambda } \right)$ (if $\mathcal{I}$ is
infinite integral \eqref{GrindEQ__62_} converges strongly). Then
the first component of vector function \eqref{GrindEQ__62_} is a
solution of equation \eqref{GrindEQ__1_}. It defines densely
defined in $L_{m}^{2} \left(\mathcal{I}\right)$
integro-differential operator
\begin{equation} \label{GrindEQ__63_}
R\left(\lambda \right)f=y_{1} \left(t,\lambda ,f\right),\quad f\in H
\end{equation}
which has the following properties after closing

\noindent 1${}^\circ$
\begin{gather}\label{GrindEQ__64_}
R^{*} \left(\lambda \right)=R\left(\bar{\lambda }\right),\, \, \,
{\Im}\lambda \ne 0
\end{gather}

\noindent 2${}^\circ$
\begin{gather}\label{GrindEQ__65_}
R\left(\lambda \right)\text{ is holomorphic on
}\mathbb{C}\setminus \mathbb{R}^1
\end{gather}

\noindent 3${}^\circ$
\begin{gather}\label{GrindEQ__66_}
\left\| R\left(\lambda \right)f\right\| _{L_{m}^{2}
\left(\mathcal{I}\right)}^{2} \le \frac{\Im \left(R\left(\lambda
\right)f,f\right)_{L_{m}^{2} \left(\mathcal{I}\right)} }{\Im
\lambda } ,\, \, \, {\Im}\lambda \ne 0,\, \, \, f\in L_{m}^{2}
\left(\mathcal{I}\right)
\end{gather}
\end{theorem}

Let us notice that the definition of the operator $R\left(\lambda
\right)$ is correct. Indeed if $f\left(t\right)\in H$,
$m\left[f,f\right]=0$, then $R\left(\lambda \right)f\equiv 0$
since $W\left(t,l_{\bar{\lambda }}
,m\right)F\left(t,l_{\bar{\lambda }},m \right)\equiv 0$ due to
\eqref{GrindEQ__29_}, \eqref{GrindEQ__53_}.

\begin{proof}
In view of Lemma \ref{lm6} integral \eqref{GrindEQ__62_} converges
strongly if $\mathcal{I}$ is infinite. In view of Theorem
\ref{th1} $y_{1} \left(t,\lambda ,f\right)$ \eqref{GrindEQ__63_}
is a solution of equation \eqref{GrindEQ__1_}.

In view of \eqref{GrindEQ__52_}, \eqref{GrindEQ__30_}
\begin{multline} \label{GrindEQ__67_}
\left\| R\left(\lambda \right)f\right\| _{L_{m}^{2} \left(\alpha
,\beta \right)} -\frac{\Im \left(R\left(\lambda
\right)f,f\right)_{L_{m}^{2} \left(\alpha ,\beta \right)} }{\Im
\lambda } \le \left\| R\left(\lambda \right)f\right\|
_{L_{-\frac{\Im l_{\lambda } }{\Im \lambda }}^{2}\left(\alpha
,\beta \right) } -\frac{\Im \left(R\left(\lambda
\right)f,f\right)_{L_{m}^{2} \left(\alpha ,\beta \right)} }{\Im
\lambda } =\\
=\left\| \mathcal{R}_{\lambda } F\left(t,l_{\bar{\lambda }}
,m\right)\right\| _{L_{W\left(t,l_{\lambda } ,-\frac{\Im
l_{\lambda } }{\Im \lambda } \right)}^{2} \left(\alpha ,\beta
\right)} -\frac{\Im \left(\mathcal{R}_{\lambda }
F\left(t,l_{\bar{\lambda }} ,m\right),F(t,,l_{\bar{\lambda }}
,m\right)_{L_{W\left(t,l_{\bar{\lambda }} ,m\right)}^{2}
\left(\alpha ,\beta \right)} }{\Im \lambda }.
\end{multline}
In view of Lemma \ref{lm6} a nonnegative limit of the
right-hand-side of \eqref{GrindEQ__67_} exists, when $\left(\alpha
,\beta \right)\uparrow \mathcal{I}$. Hence \eqref{GrindEQ__66_} is
proved.

Let $\mathcal{H}^{r}$-valued $F(t)\in
L^2_{W(t,l_{\bar\lambda},m)}(\mathcal{I})$. Then in view of
\eqref{GrindEQ__53_}, Lemma \ref{lm6}, \eqref{GrindEQ__18_} one
has
\begin{gather} \label{GrindEQ__68_}
\left\| \mathcal{R}_{\lambda } F\right\| _{L_{W\left(t,l_{\lambda
} ,m\right)}^{2} \left(\mathcal{I}\right)}^{2} \le \left\|
\mathcal{R}_{\lambda } F\right\| _{_{L_{W\left(t,l_{\lambda }
,-\frac{\Im l}{\Im \lambda } \right)}^{2}
\left(\mathcal{I}\right)} }^{2} \le \frac{\Im
\left(\mathcal{R}_{\lambda }
F,F\right)_{L_{W\left(t,l_{\bar{\lambda }} ,m\right)}^{2}
\left(\mathcal{I}\right)} }{\Im \lambda } , \\
\label{GrindEQ__69_} \left\| \mathcal{R}_{\lambda}F \right\|
_{L_{W\left(t,l_{\bar{\lambda }} ,m\right)}^{2}
\left(\mathcal{I}\right)}^{2} \le \left\| \mathcal{R}_{\lambda}F
\right\| _{L_{W\left(t,l_{\bar{\lambda }} ,-\frac{\Im l_{\lambda }
}{\Im \lambda } \right)}^{2} \left(\mathcal{I}\right)} =\left\|
\mathcal{R}_{\lambda}F \right\| _{L_{W\left(t,l_{\lambda }
,-\frac{\Im l_{\lambda } }{\Im \lambda } \right)}^{2}
\left(\mathcal{I}\right)} .
\end{gather}

In view of \eqref{GrindEQ__68_}, \eqref{GrindEQ__69_} we have
\begin{gather} \label{GrindEQ__70_}
\left\| \mathcal{R}_{\lambda } F\right\| _{L_{W\left(t,l_{\lambda
} ,m\right)}^{2} \left(\mathcal{I}\right)} \le {\left\| F\right\|
_{L_{W\left(t,l_{\bar{\lambda }} ,m\right)}^{2}
\left(\mathcal{I}\right)}  \mathord{\left/{\vphantom{\left\|
F\right\| _{L_{W\left(t,l_{\bar{\lambda }} ,m\right)}^{2}
\left(\mathcal{I}\right)}  \left|\Im \lambda
\right|}}\right.\kern-\nulldelimiterspace} \left|\Im \lambda
\right|},\\ \label{GrindEQ__71_} \left\| \mathcal{R}_{\lambda }
F\right\| _{L_{W\left(t,l_{\bar{\lambda }} ,m\right)}^{2}
\left(\mathcal{I}\right)} \le {\left\| F\right\|
_{L_{W\left(t,l_{\bar{\lambda }} ,m\right)}^{2}
\left(\mathcal{I}\right)}  \mathord{\left/{\vphantom{\left\|
F\right\| _{L_{W\left(t,l_{\bar{\lambda }} ,m\right)}^{2}
\left(\mathcal{I}\right)}  \left|\Im \lambda
\right|}}\right.\kern-\nulldelimiterspace} \left|\Im \lambda
\right|} .
\end{gather}

Let $F\left(t\right)\in L_{W\left(t,l_{\bar{\lambda }}
,m\right)}^{2} \left(\mathcal{I}\right)$, $G\left(t\right)\in
L_{W\left(t,l_{\lambda } ,m\right)}^{2} \left(\mathcal{I}\right)$
are $\mathcal{H}^{r} $-valued functions with compact support. We
have
\begin{equation} \label{GrindEQ__72_}
\left(\mathcal{R}_{\lambda } F,G\right)_{L_{W\left(t,l_{\lambda }
,m\right)}^{2} \left(\mathcal{I}\right)}
=\left(F,\mathcal{R}_{\bar{\lambda }}
,G\right)_{L_{W\left(t,l_{\bar{\lambda }} ,m\right)}^{2}
\left(\mathcal{I}\right)} .
\end{equation}
since $M\left(\lambda \right)=M^{*} \left(\bar{\lambda }\right)$.
Due to inequalities \eqref{GrindEQ__70_}, \eqref{GrindEQ__71_}
equality \eqref{GrindEQ__69_} is valid for $F\left(t\right),\,
G\left(t\right)$ with non-compact support.

Now it follows from, \eqref{GrindEQ__31_}, \eqref{GrindEQ__72_}
that $\forall f\left(t\right),g\left(t\right)\in H$
\begin{multline*}
m\left[R\left(\lambda
\right)f,g\right]-m\left[f,R\left(\bar{\lambda
}\right)g\right]=\left(\mathcal{R}_{\lambda }
F\left(t,l_{\bar{\lambda }} ,m\right),G\left(t,l_{\lambda }
,m\right)\right)_{L_{W\left(t,l_{\lambda } ,m\right)}^{2}
\left(\mathcal{I}\right)} -\\
-\left(F\left(t,l_{\bar{\lambda }}
,m\right),\mathcal{R}_{\bar{\lambda }} G\left(t,l_{\lambda }
,m\right)\right)_{L_{W\left(t,l_{\bar{\lambda }} ,m\right)}^{2}
\left(\mathcal{I}\right)} =0
\end{multline*}
Thus the closure of
the operator $R\left(\lambda \right)f$ in $L_{m}^{2}
\left(\mathcal{I}\right)$ possesses property \eqref{GrindEQ__64_}.

Since in view of \eqref{GrindEQ__66_} for any
$f\left(t\right),g\left(t\right)\in H$
$$\left(R\left(\lambda \right)f,g\right)_{L_{m}^{2} \left(\alpha
,\beta \right)} \to \left(R\left(\lambda
\right)f,g\right)_{L_{m}^{2} \left(\mathcal{I}\right)}\text{ as
}\left(\alpha ,\beta \right)\uparrow \mathcal{I}$$ uniformly in
$\lambda $ from any compact set from $\mathbb{C}\setminus
\mathbb{R}^1$, we see that, in view of the analyticity of the
operator function $M\left(\lambda \right)$ and vector-function
$W\left(t,l_{\bar{\lambda }} ,m\right)\, F\left(t,l_{\bar{\lambda
}} \right)$ (see {\eqref{GrindEQ__251_}} with $l=l_\lambda)$ the
operator $R\left(\lambda \right)$ depends analytically on the
non-real $\lambda $ in view of \cite[p. 195]{Kato}. Theorem
\ref{th4} is proved.
\end{proof}

For $r=1$, $n_\lambda[y]=H_\lambda(t)y$ Theorem \ref{th4} is known
\cite{Khrab4}.

Let us notice that if
$L_m^2(\mathcal{I})=\mathop{L_m^2}\limits^{\circ\ }(\mathcal{I})$
then Theorem \ref{th4} is valid with $f(t)\in \overset{\circ}{H}$
instead of $f(t)\in H$ and without condition \eqref{GrindEQ__55_}
with $P=I_r$ for infinite $\mathcal{I}$.

The following theorem establishes a relationship between the
resolvents $R(\lambda)$ that are given by Theorem \ref{th4} and
the boundary value problems for equation \eqref{GrindEQ__1_},
\eqref{GrindEQ__2_} with boundary conditions depending on the
spectral parameter. Similarly to the case $n_{\lambda }
\left[y\right]\equiv 0$ \cite{Khrab6} we see that the pair
$\left\{y,f\right\}$ satisfies the boundary conditions that
contain both $y$ derivatives and $f$ derivatives of corresponding
orders at the ends of the interval.

\begin{theorem}\label{th5}
Let the interval $\mathcal{I}=\left(a,b\right)$ be finite and
condition \eqref{GrindEQ__55_} with $P=I_{r} $ holds.

Let the operator-functions $\mathcal{M}_{\lambda },
\mathcal{N}_{\lambda} \in B\left(\mathcal{H}^{r} \right)$ depend
analytically on the non-real $\lambda $,
\begin{equation} \label{GrindEQ__70_}
\mathcal{M}_{\bar{\lambda }}^{*} \left[\Re Q\left(a,l_{\lambda }
\right)\right]\mathcal{M}_{\lambda } =\mathcal{N}_{\bar{\lambda
}}^{*} \left[\Re Q\left(b,l_{\lambda }
\right)\right]\mathcal{N}_{\lambda } \quad \left(\Im  \lambda \ne
0\right),
\end{equation}
where $Q\left(t,l_{\lambda } \right)$ is the coefficient of
equation \eqref{GrindEQ__54_} corresponding by Theorem \ref{th1}
to equation \eqref{GrindEQ__1_},
\begin{equation} \label{GrindEQ__71_}
\left\| \mathcal{M}_{\lambda } h\right\| +\left\|
\mathcal{N}_{\lambda } h\right\| >0\quad \left(0\ne h\in
\mathcal{H}^{r} ,\, \Im \lambda \ne 0\right),
\end{equation}
the lineal $\left\{\mathcal{M}_{\lambda } h\oplus
\mathcal{N}_{\lambda } h\left|h\in \mathcal{H}^{r} \right.
\right\}\subset \mathcal{H}^{2r} $ is a maximal
$\mathcal{Q}$-nonnegative subspace if $\Im  \lambda \ne 0$, where
$\mathcal{Q}=\left(\Im \lambda \right)\mathrm{diag} \left(\Re
Q\left(a,l_{\lambda } \right),\, -\Re Q\left(b,l_{\lambda }
\right)\right)$ (and therefore
\begin{equation} \label{GrindEQ__72_}
\Im  \lambda \left(\mathcal{N}_{\lambda }^{*} \left[\Re
Q\left(b,l_{\lambda } \right)\right]\mathcal{N}_{\lambda }
-\mathcal{M}_{\lambda }^{*} \left[\Re Q\left(a,l_{\lambda }
\right)\right]\mathcal{M}_{\lambda } \right)\le 0\quad \left.
\left(\Im  \lambda \ne 0\right)\right).
\end{equation}

Then

1$^\circ$. For any $f\left(t\right)\in H$ the boundary problem
that is obtained by adding the boundary conditions

\begin{equation} \label{GrindEQ__73_}
\exists h=h\left(\lambda ,f\right)\in \mathcal{H}^{r} :\, \,
\vec{y}\left(a,\lambda ,m,f\right)=\mathcal{M}_{\lambda } h,\, \,
\, \vec{y}\left(b,\lambda ,m,f\right)=\mathcal{N}_{\lambda }
h\end{equation} to the equation \eqref{GrindEQ__1_}, where
$\vec{y}\left(t,\lambda ,f\right)$ is defined by
\eqref{GrindEQ__25_}, has the unique solution $R\left(\lambda
\right)f$ in $C^r(\bar{\mathcal{I}},\mathcal{H})$ as $\Im  \lambda
\ne 0$. It is generated by the resolvent $R\left(\lambda \right)$
that is constructed, as in Theorem \ref{th4}, using the c.o.
\begin{equation} \label{GrindEQ__74_}
M\left(\lambda \right)=-\frac{1}{2} \left(X_{\lambda }^{-1}
\left(a\right)\mathcal{M}_{\lambda } +X_{\lambda }^{-1}
\left(b\right)\mathcal{N}_{\lambda } \right)\left(X_{\lambda
}^{-1} \left(a\right)\mathcal{M}_{\lambda } -X_{\lambda }^{-1}
\left(b\right)\mathcal{N}_{\lambda } \right)^{-1} \,
\left(iG\right)^{-1} ,
\end{equation}
where
$$\left(X_{\lambda }^{-1} \left(a\right)\mathcal{M}_{\lambda
} -X_{\lambda }^{-1} \left(b\right)\mathcal{N}_{\lambda }
\right)^{-1} \in B\left(\mathcal{H}^{r} \right)\quad \left(\Im
\lambda \ne 0\right),$$
$X_\lambda(t)$ is an operator solution of
the homogeneous equation \eqref{GrindEQ__54_} such that
$X_{\lambda } \left(0\right)=I_{r} $.

2$^\circ$. For any operator $R(\lambda)$ from Theorem \ref{th4}
vector-function $R(\lambda)f$ $(f\in H)$ is a solution of some
boundary problem as in 1$^\circ$.
\end{theorem}

Let us notice that if $f\left(t\right)\mathop{=}\limits^{H}
g\left(t\right)$ then in boundary conditions \eqref{GrindEQ__73_}:
$\vec{y}\left(t,l,m,f\right)=\vec{y}\left(t,l,m,g\right)$ in view
of \eqref{GrindEQ__25_} and Proposition \ref{prop1}.

\begin{proof}
The proof of Theorem \ref{th5} follows from Theorems \ref{th1},
\ref{th4} and from \cite[Remark 1.1]{Khrab5}.
\end{proof}

For the case $n_{\lambda } \left[y\right]\equiv 0$, Theorem
\ref{th5} is known \cite{Khrab5}.

The example below show that the following is possible: for some
resolvent $R\left(\lambda \right)$ from Theorem \ref{th4} $\exists
f_{0} \left(t\right)\mathop{\ne }\limits^{H} 0$ such that
$m\left[f_{0} \right]=0$ and therefore the "resolvent" equation
\eqref{GrindEQ__1_} for $R\left(\lambda \right)f_{0} $
 is homogeneous but $R\left(\lambda
\right)f_0\mathop{\ne }\limits^{H} 0,\, \, \, \Im \lambda \ne 0$.

\begin{example}\label{ex2}
Let $m$ in \eqref{GrindEQ__1_} be such expression that equation
$m\left[f\right]=0$ has a solution $f_{0}
\left(t\right)\mathop{\ne }\limits^{H} 0$. Let in Theorem
\ref{th5}: $\mathcal{M}_{\lambda } =\left(\begin{array}{cc} {I_{n}
} & {0} \\ {0} & {0}
\end{array}\right),\, \, \mathcal{N}_{\lambda }
=\left(\begin{array}{cc} {0} & {I}_n \\ {0} & {0}
\end{array}\right)$, $R\left(\lambda \right)$ is the corresponding
resolvent. Then $R(\lambda)f_0\not= 0,\Im\lambda\not= 0$, while if
$\mathcal{M}_{\lambda } =\left(\begin{array}{cc} {0} & {0} \\
{I_{n} } & {0}
\end{array}\right),\, \, \mathcal{N}_{\lambda }
=\left(\begin{array}{cc} {0} & {0} \\ {0} & {I_{n} }
\end{array}\right)$ then for the corresponding resolvent
$R\left(\lambda \right)f_{0} \mathop{=}\limits^{H} 0,\, \, \Im
\lambda \ne 0$ (and therefore in view of \cite[p. 87]{DSnoo1}
$E_\infty f_0=0$ for generalized spectral family $E_\mu$, which
corresponds to $R(\lambda)$ by \eqref{GrindEQ__3_}).

It is known \cite[p.86]{DSnoo1} that the operator-function
$R\left(\lambda \right)$ \eqref{GrindEQ__64_}-\eqref{GrindEQ__66_}
can be represented in the form
\begin{equation} \label{GrindEQ__78_}
R\left(\lambda \right)=\left(T\left(\lambda \right)-\lambda \right)^{-1} ,
\end{equation}
where $T\left(\lambda \right)$ is such linear relation that
$$\Im
T\left(\lambda \right)\le 0\, \, \left(\max \right),\, \,
T\left(\bar{\lambda }\right)=T^{*} \left(\lambda \right),\, \,
\lambda \in \mathbb{C}^{+} ,$$ the Cayley transform $C_{\mu }
\left(T\left(\lambda \right)\right)$ defines a holomorphic
function in $\lambda \in \mathbb{C}_{+} $ for some (and hence for
all) $\mu \in \mathbb{C}_{+} $. Applications of abstract relations
of $T(\lambda)$ type (Nevanlinna families) to the theories of
boundary relations and of generalized resolvents are proposed in
\cite{DHMdS1,DHMdS2}
\end{example}

The description of $T\left(\lambda \right)$ corresponding to
$R\left(\lambda \right)$ from Theorem \ref{th4} in regular case
gives

\begin{corollary} Let $\mathcal{I}$ is finite and condition
\eqref{GrindEQ__55_} with $P=I_{r} $ holds. Let us consider the
relation $T\left(\lambda \right)=\overline{T'\left(\lambda
\right)}$ as $\Im \lambda \ne 0$, where
\begin{multline}\label{GrindEQ__782+_}
T'\left(\lambda
\right)=\bigg\{\left.\left\{\tilde{y}\left(t\right),
\tilde{f}\left(t\right)\right\}\right|\tilde{y}\left(t\right)\mathop{=}\limits^{L_{m}^{2}
\left(\mathcal{I}\right)} y\left(t\right)\in C^{r}
\left(\bar{\mathcal{I}}\right),
\tilde{f}\left(t\right)\mathop{=}\limits^{L_{m}^{2}
\left(\mathcal{I}\right)} f\left(t\right)\in H,\left(l-n_{\lambda
}
\right)\left[y\right]=m\left[f\right],\\\vec{y}\left(t,l-n_{\lambda
} ,m,f\right)\text{satisfy boundary condition}
\\\exists h=h\left(\lambda ,f\right)\in
\mathcal{H}^{r} :\vec{y}\left(a,l-n_{\lambda }
,m,f\right)=\mathcal{M}_{\lambda } h,\, \,
\vec{y}\left(b,l-n_{\lambda } ,m,f\right)=\mathcal{N}_{\lambda }
h,\\\text{ where operators
}\mathcal{M}_\lambda,\mathcal{N}_\lambda\text{ satisfy the
conditions of Theorem \ref{th5},}\\
\vec{y}\left(t,l-n_{\lambda } ,m,f\right)\mathop{=}\limits^{def}
\vec{y}\left(t,l_{\lambda } ,m,f\right)\left|_{m=0\, in\,
l_{\lambda } } \right. =\\
=\begin{cases} \left(\sum\limits _{j=0}^{n-1}\oplus
y^{\left(j\right)} \left(t\right) \right)\oplus \sum\limits
_{j=1}^{n}\oplus  \left(y^{\left[s'-j\right]}
\left(t\left|l-n_{\lambda } \right. \right)-f^{\left[s-j\right]}
\left(t\left|m\right. \right)\right),& r=2n \\
\left(\sum\limits _{j=0}^{n-1}\oplus y^{\left(j\right)}
\left(t\right) \right)\oplus \left(\sum\limits _{j=1}^{n}\oplus
\left(y^{\left[r-j\right]} \left(t\left|l-n_{\lambda } \right.
\right)-f^{\left[s-j\right]} \left(t\left|m\right.
\right)\right)\right)\oplus \left(-iy^{\left(n\right)}
\left(t\right)\right),& r=2n+1>1
\\\qquad \big(\text{here }s'=\text{order of expression }
l-n_{\lambda },\\\qquad\ y^{[0]}(t|l-n_\lambda)=-{i\over 2}
\left(q_1(t,\lambda)-\tilde q_1(t,\lambda)\right)y\text{ as }s'=1
\\\qquad\ y^{\left[k'\right]} \left(t\left|l-n_{\lambda }
\right. \right)\equiv 0\text{ as }k'<\left[\frac{s'}{2} \right],\
f^{\left[k\right]} \left(t\left|m\right.
\right)\equiv 0\text{ as }k<\frac{s}{2} \big) \\
y\left(t\right),& r=1
\end{cases}\bigg\}.
\end{multline}

Then

1${}^\circ$. $\left(T\left(\lambda \right)-\lambda \right)^{-1} $
is equal to resolvent $R(\lambda)$ \eqref{GrindEQ__62_},
\eqref{GrindEQ__63_} from Theorem \ref{th4} corresponding to c.o.
$M\left(\lambda \right)$ \eqref{GrindEQ__74_}.

2${}^\circ$. Let $R\left(\lambda \right)$ is resolvent
\eqref{GrindEQ__62_}, \eqref{GrindEQ__63_} from Theorem \ref{th4}.
Then $R\left(\lambda \right)=\left(T\left(\lambda \right)-\lambda
\right)^{-1} $, where $T\left(\lambda \right)$ is some relation as
in item 1$^\circ$.

\end{corollary}

\begin{proof}
The proof follows from \eqref{GrindEQ__25_}, Lemma \ref{lm2},
Theorem \ref{th5} and Remark 1.1 from \cite{Khrab5}.
\end{proof}

Let in \eqref{GrindEQ__1_}, \eqref{GrindEQ__2_} $n_{\lambda }
\left[y\right]\equiv 0$ l.e. $l_{\lambda } =l-\lambda m$, where
$l=l^{*} ,\, m=m^{*} $ and coefficients of expressions $m$ satisfy
condition \eqref{GrindEQ__52_}.

We consider in $L_{m}^{2} \left(\mathcal{I}\right)$ the linear
relation
\begin{multline} \label{GrindEQ__781_}
\mathcal{L}'_{0} =\bigg\{
\left\{\tilde{y}\left(t\right),\tilde{g}\left(t\right)\right\}
|\tilde{y}\left(t\right)\mathop{=}\limits^{L_{m}^{2}
\left(\mathcal{I}\right)}
y\left(t\right),\tilde{g}\left(t\right)\mathop{=}\limits^{L_{m}^{2}
\left(\mathcal{I}\right)} g\left(t\right),y\left(t\right)\in C^{r}
\left(\bar{\mathcal{I}},\mathcal{H}\right),g\left(t\right)\in
H,l\left[y\right]=m\left[g\right],\\
\vec{y}\left(t,l,m,g\right)\text{ is equal to zero in the edge of
}\mathcal{I}\text{ if this edge is finite and }
\vec{y}\left(t,l,m,g\right)\\\text{ is equal to zero in the some
neighbourhood of the edge of }\mathcal{I}\text{ if this edge is
infinite }\bigg\}
\end{multline}\footnote{Let us notice that
vector-function $g\left(t\right)$ in \eqref{GrindEQ__781_} may be
non-equal to zero in the finite edge or in the some neighbourhood
of infinite edge of $\mathcal{I}$.} where
$\vec{y}\left(t,l,m,g\right)$ is defined by
{\eqref{GrindEQ__782+_}} with $l_{\lambda}=l-\lambda m$, $f=g$.

Below we assume that relation $\mathcal{L}_0^\prime$ consists of
the pairs of $\{y,g\}$ type.

The relation $\mathcal{L}'_{0} $ is symmetric due to the following
Green formula with $\lambda_k=0$:

Let $y_{k} \left(t\right)\in C^{r} \left(\left[\alpha ,\beta
\right],\mathcal{H}\right)$, $f_{k} \left(t\right)\in C^{s}
\left(\left[\alpha ,\beta \right],\mathcal{H}\right)$, $\lambda
_{k} \in \mathbb{C}$, $l\left[y_{k} \right]-\lambda _{k}
m\left[y_{k} \right]=m\left[f_{k} \right],\, \, k=1,2$. Then
\begin{multline} \label{GrindEQ__782_}
\int _{\alpha }^{\beta }m\left\{f_{1} ,y_{2} \right\}dt -\int
_{\alpha }^{\beta }m\left\{y_{1} ,f_{2} \right\}dt +\left(\lambda
_{1} -\bar{\lambda }_{2} \right)\int _{\alpha }^{\beta
}m\left\{y_{1} ,y_{2} \right\}dt =
\\
\left.=i\left(\Re Q\left(t,l_{\lambda } \right)\vec{y}_{1}
\left(t,l_{\lambda _{1} } ,m,f_{1} \right),\vec{y}_{2}
\left(t,l_{\lambda _{2} } ,m,f_{2}
\right)\right)\right|_\alpha^\beta,
\end{multline}
where $\vec{y}_{k} \left(t,l_{\lambda _{k} } ,m,f_{k} \right)$ for
$\lambda _{k} \in \mathbb{R}^{1} $ is defined by
{\eqref{GrindEQ__782+_}} with $l_\lambda=l-\lambda m$.

This formula is a corollary of Theorem \ref{th3} if $\Im \lambda
_{k} \ne 0$. For its proof for example in the case $\lambda _{1}
\in \mathbb{R}^{1} $ we need to modify \eqref{GrindEQ__36_} for
equation $l\left[y_{1} \right]-\left(\lambda _{1} +i\varepsilon
\right)m\left[y\right]=m\left[f_{1} -i\varepsilon y_{1} \right]$
and then to pass to limit in \eqref{GrindEQ__36_} as $\varepsilon
\to +0$.

In general the relation $\mathcal{L}'_{0} $ is not closed. We
denote $\mathcal{L}_{0} =\bar{\mathcal{L}}_{0}^\prime $.

\begin{theorem}\label{th6}
Let $l_{\lambda } =l-\lambda m$ and the conditions of Theorem
\ref{th4} hold. Then the operator $R\left(\lambda \right)$ from
Theorem \ref{th4} is the generalized resolvent of the relation
$\mathcal{L}_{0} $. Let $\mathcal{I}$ be finite and additionally
the condition \eqref{GrindEQ__55_} hold. Then every such
generalized resolvent can be constructed as the operator
$R\left(\lambda \right)$.
\end{theorem}

\begin{proof}
In view of \cite{DSnoo1} and taking into account properties
\eqref{GrindEQ__64_}-\eqref{GrindEQ__66_} of the operator
$R\left(\lambda \right)$ it is sufficiently to prove that
$R\left(\lambda \right)\left(\mathcal{L}_{0} -\lambda
\right)\subseteq \mathbf{I}$, where $\mathbf{I}$ is a graph of the
identical operator in $L_{m}^{2} \left(\mathcal{I}\right)$. But
this proposition is proved similarly to \cite[p. 453]{Khrab5}
taking into account \eqref{GrindEQ__782_} and the fact that in
view of (\ref{GrindEQ__78_}) $\overrightarrow{(\tilde y-
y)}(t,l-\lambda m,m,0)=\overrightarrow{\tilde y}(t,l-\lambda
m,m,g-\lambda y)-\vec y(t,l,m,g)$ if $\{y,g\}\in \mathcal{L}'_0$,
$\tilde y=R(\lambda)(g-\lambda y)$.

Conversely let $\mathcal{I}$ be finite, $R_{\lambda } $ a
generalized resolvent of relation $\mathcal{L}_{0} $. We denote
${N}_{\lambda } =\left\{y\left(t\right)\in C^{r}
\left(\mathcal{I},\mathcal{H}\right),\lambda \in \mathcal{B},\,
l\left[y\right]-\lambda m\left[y\right]=0\right\}$. We need the
following

\begin{lemma}\label{lm8}
Let condition \eqref{GrindEQ__55_} hold. Then the lineal
${N}_{\lambda } $ is closed in $L_{m}^{2}
\left(\mathcal{I}\right)$.
\end{lemma}

\begin{proof}The proof of Lemma \ref{lm8} follows from
\eqref{GrindEQ__29_}.\end{proof}

\begin{lemma}\label{lm9}Let $\lambda \in \mathcal{B}$. Then
$\overline{R\left(\mathcal{L}'_{0} -\bar{\lambda
}\right)}={N}_{\lambda }^{\bot } $.
\end{lemma}

\begin{proof} Let $x\left(t\right)\in N_{\lambda },
f\left(t\right)\in H,\, y\left(t\right)$ is a solution of the
following Cauchy problem:
\begin{equation} \label{GrindEQ__78_}
l\left[y\right]-\bar{\lambda }m\left[y\right]=m\left[f\right],\,
\, \, \vec{y}\left(a,l_{\bar{\lambda }} ,m,f\right)=0.
\end{equation}
Then
\begin{equation} \label{GrindEQ__79_}
m\left[f,x\right]=i\left(\Re Q\left(b,l_{\lambda }
\right)\vec{y}\left(b,l_{\bar{\lambda }}
,m,f\right),\vec{x}\left(b,l_{\lambda } ,m,0\right)\right)
\end{equation}
in view of Green formula \eqref{GrindEQ__782_}. Therefore
$\overline{R\left(\mathcal{L}'_{0} -\bar{\lambda
}\right)}\subseteq {N}_{\lambda }^{\bot } $.

Let $g\left(t\right)\in {N}_{\lambda }^{\bot } $. Then $\exists \,
\, H\ni g_{n} \mathop{\to }\limits^{L_{m}^{2}
\left(\mathcal{I}\right)} g$, $g_{n} =x_{n} \oplus f_{n} $, $x_{n}
\in {N}_{\lambda } $, $f_{n} \in {N}_{\lambda }^{\bot }
\Rightarrow f_{n} \in H$. Let $y_n$ be a solution of problem
\eqref{GrindEQ__78_} with $f_{n} $ instead of $f$. In view of
\eqref{GrindEQ__79_} with $f=f_{n} $, one has: $\vec{y}_{n}
\left(b,l_{\bar{\lambda }} ,m,f_{n} \right)=0\Rightarrow f_{n} \in
R\left(\mathcal{L}'_{0} -\bar{\lambda }\right)$. But $f_{n}
\mathop{\to }\limits^{L_{m}^{2} \left(\mathcal{I}\right)} g$.
Therefore $\overline{R\left(\mathcal{L}'_{0} -\bar{\lambda
}\right)}\supseteq {N}_{\lambda }^{\bot } $. Lemma \ref{lm9} is
proved.
\end{proof}

\begin{lemma}\label{lm10}
Let the condition \eqref{GrindEQ__55_} hold, $\lambda \in
\mathcal{B}$. Let $\left\{\tilde{y},\tilde{f}\right\}\in
\mathcal{L}_{0}^{*} -\lambda $,
$\tilde{f}\mathop{=}\limits^{L_{m}^{2} \left(\mathcal{I}\right)}
f\in H$. Then $\tilde{y}\mathop{=}\limits^{L_{m}^{2}
\left(\mathcal{I}\right)} y\in C^{r}
\left(\bar{\mathcal{I}},\mathcal{H}\right)$ and $y\left(t\right)$
satisfies the equation \eqref{GrindEQ__1_}.
\end{lemma}

\begin{proof}
Let $C^{r} \left(\bar{\mathcal{I}},\mathcal{H}\right)\ni y_{0} $
be a solution of \eqref{GrindEQ__1_}. Let $\left\{\varphi ,\psi
\right\}\in \mathcal{L}'_{0} -\bar{\lambda }$. Then $\vec{\varphi
}\left(a,l_{\bar{\lambda }} ,m,\psi \right)=\vec{\varphi
}\left(b,l_{\bar{\lambda }} ,m,\psi \right)=0$ in view of
\eqref{GrindEQ__782+_}, \eqref{GrindEQ__781_}. Hence
$m\left[\varphi ,f\right]=m\left[\psi ,y_{0} \right]$ due to Green
formula \eqref{GrindEQ__782_}. But $m\left[\varphi
,f\right]=\left(\psi ,\tilde{y}\right)_{L_{m}^{2}
\left(\mathcal{I}\right)} $ in view of the definition of the
adjoint relation. Hence $\left(\psi ,\tilde{y}-y_{0}
\right)\mathop{=}\limits_{L_{m}^{2} \left(\mathcal{I}\right)}0$.
Therefore $\tilde y-y_{0} \mathop{=}\limits^{L_{m}^{2}
\left(\mathcal{I}\right)} y-y_{0} \in {N}_{\lambda } $ in view of
Lemmas \ref{lm8}, \ref{lm9}. Hence
$\tilde{y}\mathop{=}\limits^{L_{m}^{2} \left(\mathcal{I}\right)}
y\in C^{r} \left(\bar{\mathcal{I}},\mathcal{H}\right)$ and $y$ is
a solution of \eqref{GrindEQ__1_}. Lemma \ref{lm10} is proved.
\end{proof}

We return to the proof of Theorem \ref{th6}.

Let $f\in H$. Then in view of Lemma \ref{lm10} $R_{\lambda }
f\mathop{=}\limits^{L_{m}^{2} \left(\bar{\mathcal{I}}\right)} y\in
C^{r} \left(\bar{\mathcal{I}},\mathcal{H}\right)$ and $y$
satisfies equation \eqref{GrindEQ__1_}. Therefore taking into
account Theorem \ref{th1}, \cite[p.148]{DalKrein} and
\eqref{GrindEQ__47+_} we have
\begin{equation} \label{GrindEQ__791_}
y\left(t\right)= \left[X_{\lambda }
\left(t\right)\right]_{1}\left\{ h-{1\over 2}\left(iG\right)^{-1}
\left(\int_{a}^b sqn(s-t)X_{\bar{\lambda }}^{*}
\left(s\right)W\left(s,l_{\bar{\lambda }}
,m\right)F\left(s,l_{\bar{\lambda }} ,m\right)ds\right)\right\},
\end{equation}
where $\left[X_{\lambda } \left(t\right)\right]_{1} \in
B\left(\mathcal{H}^{r} ,\mathcal{H}\right)$ is the first row of
operator solution $X_{\lambda } \left(t\right)$ from Theorem
\ref{th4}, that is written in the matrix form, $h=h_{\lambda }
\left(f\right)\in N^{\bot } $ is defined in the unique way in view
of \eqref{GrindEQ__29_} and condition \eqref{GrindEQ__55_}.

Let us prove that $h$ depends on $I_{\lambda }
f\mathop{=}\limits^{def} \int _{a}^{b}X_{\bar{\lambda }}^{*}
\left(s\right)W\left(s,l_{\bar{\lambda }}
,m\right)F\left(s,l_{\bar{\lambda }} ,m\right)ds $ in unique way.
Operator $\left(I_{\lambda } :H\to N^{\bot } \right)$ in view of
Lemma \ref{lm7}. Moreover $I_{\lambda } N^{\bot } =N^{\bot }$ i.e.
$\forall h_{0} \in N^{\bot } \exists f_{0} \in H:\, h_{0}
=I_{\lambda } f_{0} $. For example we can set
\begin{gather}\label{star}
f_0=f_0(t,\lambda)=\left[X_{\bar{\lambda }}
\left(t\right)\right]_{1} \left\{\Delta _{\bar{\lambda }}
\left(\mathcal{I}\right)\left|_{N^{\bot } } \right. \right\}^{-1}
h_{0}
\end{gather}
and to utilize the equality.
\begin{gather*}
W\left(s,l_{\bar{\lambda }} ,m\right)F_{0} \left(s,l_{\bar{\lambda
}} ,m\right)=W\left(s,l_{\bar{\lambda }} ,m\right)X_{\bar{\lambda
}} \left(s\right)\left\{\ldots \right\}^{-1} h_{0}
\end{gather*}

If $f\left(t\right),g\left(t\right)\in H$ are such functions that
$I_{\lambda } f=I_{\lambda } g$, then in view of
\eqref{GrindEQ__791_}
\begin{multline} \label{GrindEQ__80_}
\left.\Im \lambda \left(\left(\Re Q\left(t,l_{\lambda }
\right)\right)\overrightarrow{\Delta y} \left(t,l_{\lambda }
,m,f-g\right),\overrightarrow{\Delta y}\left(t,l_{\lambda }
,m,f-g\right)\right)\right|_\alpha^\beta =
\\
=\left.\Im \lambda \left(\left(\Re Q\left(t,l_{\lambda }
\right)\right)X_{\lambda } \left(t\right)\left(h_{\lambda }
\left(f\right)-h_{\lambda } \left(g\right)\right),X_{\lambda }
\left(t\right)\left(h_{\lambda } \left(f\right)-h_{\lambda }
\left(g\right)\right)\right)\right|_\alpha^\beta,
\end{multline}
where $\Delta y=R_{\lambda } \left[f-g\right]$. But in view of
\eqref{GrindEQ__782_} the left hand side of \eqref{GrindEQ__80_}
is nonpositive since $R_{\lambda } $ has property of
\eqref{GrindEQ__66_} type. The right hand of \eqref{GrindEQ__80_}
is nonnegative in view of \eqref{GrindEQ__36_}. Hence $h_{\lambda
} \left(f\right)=  h_{\lambda } \left(g\right)$ in view of
\eqref{GrindEQ__36_}, \eqref{GrindEQ__55_}. Thus $h$ depends on
$I_{\lambda } f$ in unique way and obviously in the linear way.
Therefore

\begin{equation} \label{GrindEQ__81_}
h=M\left(\lambda \right)I_{\lambda } f,
\end{equation}
where $M\left(\lambda \right):N^{\bot } \to N^{\bot } $ is a
linear operator and so $R_\lambda f$ ($f\in H$) can be represented
in the form (\ref{GrindEQ__63_}).

Further, for definiteness, we will consider the most complicated
case $r=s=2n$.

Let us prove that $M\left(\lambda \right)\in B\left(N^{\bot }
\right),\, \, \Im \lambda \ne 0$. Let $h_{0} \in N^{\bot } ,\, \,
y=R_{\lambda } f_{0} $, where $f_{0} =f_{0} \left(t,\lambda
\right)$ see (\ref{star}). Then in view of (\ref{GrindEQ__791_})
and Theorem \ref{th1} we have
\begin{equation} \label{GrindEQ__811_}
X_{\lambda } \left(t\right)M\left(\lambda \right)h_{0}
=Y\left(t,l_{\lambda } ,m\right)-\mathcal{F}_{0}
\left(t,m\right)-{1\over 2}X_{\lambda }
\left(t\right)\left(iG\right)^{-1}\left( I_{\lambda }
\left(a,t\right)- I_{\lambda } \left(t,b\right)\right)F_{0} ,
\end{equation}
where $Y\left(t,l_{\lambda } ,m\right),\, \, F_{0} =F_{0}
\left(t,l_{\bar{\lambda }} ,m\right),\ \mathcal{F}_{0}
\left(t,m\right)$ are defined by \eqref{GrindEQ__23_},
\eqref{GrindEQ__32_} correspondingly with $y$ and $f_{0} $
correspondingly instead of $f,\ I_{\lambda } \left(0,t\right)F_{0}
$ is defined by \eqref{GrindEQ__60+_}. Therefore
\begin{equation} \label{GrindEQ__812_}
\Delta _{\lambda } \left(a,b\right)M(\lambda) h_{0}
=I_{\bar{\lambda }} y-I_{\bar{\lambda }}
\left(a,b\right)\left(\mathcal{F}_{0} \left(t,m\right)+{1\over
2}X_{\lambda } \left(t\right)\left(iG\right)^{-1}\left( I_{\lambda
} \left(a,t\right)- I_{\lambda } \left(t,b\right)\right)F_{0}
\right),
\end{equation}
where $I_{\bar{\lambda }} y,\, \, I_{\bar{\lambda }}
\left(a,b\right)\left(\ldots \right)\in N^{\bot } $ in view of
\eqref{GrindEQ__61-_}. But
$$
\forall g\in \mathcal{H}^{r} :\, \, \left|\left(I_{\bar{\lambda }}
y,g\right)\right|\le \mathop{\max }\limits_{t\in
\bar{\mathcal{I}}} \left\| X_{\lambda } \left(t\right)\right\|
\left\{\int _{\mathcal{I}}\left\| W\left(t,l_{\lambda }
,m\right)\right\| dt \right\}^{1/2} \left\| R_{\lambda } f_{0}
\right\| _{L_{m}^{2} \left(\mathcal{I}\right)} \left\| g\right\|$$
in view of Cauchy inequality and \eqref{GrindEQ__29_}. Therefore
\begin{gather}\label{star2}
\exists\text{ constant }c\left(\lambda \right):\, \, \,
\left|\left(I_{\bar{\lambda }} y,g\right)\right|\le c\left(\lambda
\right)\left\| y\right\| \, \left\| g\right\|
\end{gather}
since
$$\left\| R_{\lambda } f_{0} \right\| _{L_{m}^{2}
\left(\mathcal{I}\right)} \le \left\| \Delta _{\bar{\lambda }}
\left(a,b\right)\right\| ^{1/2} \left\| \left(\Delta
_{\bar{\lambda }} \left(a,b\right)\left|_{N^{\bot } } \right.
\right)^{-1} \right\| {\left\| h_{0} \right\|
\mathord{\left/{\vphantom{\left\| h_{0} \right\|  \left|\Im
\lambda \right|}}\right.\kern-\nulldelimiterspace} \left|\Im
\lambda \right|}$$ in view of \eqref{GrindEQ__29_}, \eqref{star2}
and inequality: $\left\| R_{\lambda } f_{0} \right\| _{L_{m}^{2}
\left(\mathcal{I}\right)} \le {\left\| f_{0} \right\| _{L_{m}^{2}
\left(\mathcal{I}\right)} \mathord{\left/{\vphantom{\left\| f_{0}
\right\| _{L_{m}^{2} \left(\mathcal{I}\right)}  \left|\Im \lambda
\right|}}\right.\kern-\nulldelimiterspace} \left|\Im \lambda
\right|} $.

Obviously $\left|\left(I_{\bar{\lambda} }
\left(a,b\right)\left(\ldots \right),g\right)\right|$ satisfies
the estimate of type \eqref{star2}. Therefore $M\left(\lambda
\right)\in B(N^{\bot })$.

Now we have to prove that $M(\lambda)$ is a c.o. o equation
\eqref{GrindEQ__51_}.

Let us prove that $M\left(\lambda \right)$ is strogly continuous
for nonreal $\lambda $. To prove this fact it is enough to verify
it for $\Delta _{\lambda } \left(a,b\right)M\left(\lambda
\right)$; while the last one obviously follows from strongly
continuity of vector-function $I_{\bar{\lambda }}R_\lambda
f_0(t,\lambda) $.

In view of \eqref{GrindEQ__29_} we have $\forall
g\in\mathcal{H}^r$
$$
\left(I_{\bar{\lambda }} R_{\lambda } f_{0} \left(t,\lambda
\right)-I_{{\mu }} R_{\bar{\mu }} f_{0}
\left(t,\mu\right),g\right)=m\left[R_{\lambda } f_{0}
\left(t,\lambda \right),\left[X_{\lambda }
\left(t\right)\right]_{1} g\right]-m\left[R_{\mu } f_{0}
\left(t,\mu \right),\left[X_{\mu } \left(t\right)\right]_{1}
g\right].$$

Then the required statement can be derived from the equality
\begin{multline*}
m\left\{\left[X_{\lambda } \left(t\right)-X_{\mu }
\left(t\right)\right]_{1} g,\, \left[X_{\lambda }
\left(t\right)-X_{\mu } \left(t\right)\right]_{1} g\right\}=\\
=\left(W\left(t,l_{\lambda } ,m\right)\left(\left(X_{\lambda }
\left(t\right)-X_{\mu } \left(t\right)\right)g+\left(\lambda-\mu
\right)\mathcal{F}(t,m)\right),\left(X_{\lambda }
\left(t\right)-X_{\mu } \left(t\right)\right)g+\left(\lambda-\mu
\right)\mathcal{F}(t,m)\right),
\end{multline*}
where $\mathcal{F}(t,m)$ is defined by \eqref{GrindEQ__32_} with
$f\left(t\right)=\left[X_{\mu } \left(t\right)\right]_{\, 1} g$,
$\left\| X_{\lambda } \left(t\right)-X_{\mu }
\left(t\right)\right\| \underset{\mu \to \lambda}\to 0$ uniformly
in $t\in \left[a,b\right]$. and from the analogous equality for
$m\{f_0(t,\lambda)-f_0(t,\mu),f_0(t,\lambda)-f_0(t,\mu)\}$.

Let us prove that $M\left(\lambda \right)$ is analytic for nonreal
$\lambda $. To prove this fact it is enough in view of strongly
continuity of $M(\lambda)$ to prove the analyticity in $\lambda$
of $\left(I_{\lambda \mu } M\left(\lambda \right)I_{\lambda }
f,g\right)$, where $f\left(t\right)\in C^{r}
\left(\bar{\mathcal{I}},\mathcal{H}\right),\, g\in \mathcal{H}^{r}
$, $(\Im\lambda)(\Im\mu)>0$,
$$I_{\lambda \mu } =\int _{a}^{b}X_{\mu }^{*}
\left(t\right)W\left(t,l_{\mu } ,m\right)X_{\lambda }
\left(t\right) dt\, \in \, B\left(N^{\bot } \right),$$
$I_{\lambda\mu }^{-1} \in B\left(N^{\bot } \right)$ if
$\left|\lambda -\mu \right|$ is sufficiently small. In view of
\eqref{GrindEQ__812_}, \eqref{GrindEQ__63_}, Theorem \ref{th1},
\eqref{GrindEQ__29_}, \eqref{GrindEQ__251_}, \eqref{GrindEQ__7_}
we have
\begin{multline}\label{star4}
\left(I_{\lambda \mu } M\left(\lambda \right)I_{\lambda }
f,g\right)=m\left[R_{\lambda } f,\left[X_{\mu }
\left(t\right)\right]_{\, 1} g\right]+\left(\lambda -\mu
\right)\int _{a}^{b}\left(\left(R_{\lambda }
f\right)^{\left[n\right]} \left(t\left|m\right. \right),\,
g^{\left(n\right)} \left(t\right)\right) dt+\\+ \text{terms
independent on }R_\lambda f\text{ and analytic in }\lambda,
\end{multline}
where $g^{\left(n\right)} \left(t\right)\mathop{=}\limits^{def}
\left(p_{n} -\bar{\mu }\tilde{p}_{n} \right)^{-1}
\left(\left[X_{\mu } \right]_{\, 1} g\right)^{\left[n\right]}
\left(t\left|m\right. \right)$.

For scalar or vector-function $F\left(\lambda \right)$ let us
denote $$\Delta _{km} F\left(\lambda \right)=\frac{F\left(\lambda
+\Delta _{k} \lambda \right)-F\left(\lambda \right)}{\Delta _{k}
\lambda } -\frac{F\left(\lambda +\Delta _{m} \lambda
\right)-F\left(\lambda \right)}{\Delta _{m} \lambda }.$$ Let us
denote
$$\mathrm{R}_{n} \left(\lambda \right)=\int
_{a}^{b}\left(\tilde{p}_{n} \left({R}_{\lambda }
f\right)^{\left[n\right]}(t|m) ,g^{\left(n\right)} \right) dt.$$

In view of \eqref{GrindEQ__11_}, \eqref{GrindEQ__52_} we have
\begin{gather}
\left|\Delta _{km} \mathrm{R}_{n} \left(\lambda \right)\right|\le
\left(m\left[\Delta _{km} R_{\lambda } f,\, \Delta _{km}
R_{\lambda } f\right]\right)^{1\over 2}\left(\int_a^b(\tilde p_n
g^{(n)},g^{(n)})dt \right)^{1/2}
\end{gather}

Therefore $\mathrm{R}_{n} \left(\lambda \right)$ depends
analytically on nonreal $\lambda $ in view of analyticity of
$R_{\lambda } $ and so analyticity of $M\left(\lambda \right)$ is
proved in view of \eqref{star4}.

Let us consider the solution $x_{\lambda }
\left(t,F\right)=\mathcal{R}_{\lambda } F$ \eqref{GrindEQ__57_} of
equation \eqref{GrindEQ__51_}. Let us prove that $x_{\lambda }
\left(t,F\right)$ satisfies the condition \eqref{GrindEQ__47++_}.
Let us denote $y\left(t,\lambda ,f\right)=R_{\lambda } f$. Then in
view of Green formula \eqref{GrindEQ__36_}

\begin{equation} \label{GrindEQ__814_}
\left. m\left[y,y\right]-\frac{\Im m\left[y,f\right]}{\Im \lambda
} =\frac{1}{2} \left(\Re Q\left(t,l_{\lambda }
\right)\vec{y}\left(t,l_{\lambda }
,m,f\right),\vec{y}\left(t,l_{\lambda }
,m,f\right)\right)\right|_a^b/\Im\lambda
\end{equation}

But the left hand side of \eqref{GrindEQ__814_} is $\le 0$ since
$R_{\lambda } f$ is generalized resolvent. So
\begin{equation} \label{GrindEQ__815_}
\forall f\in H:\, \, \left.\Re \left(Q\left(t,l_{\lambda }
\right)\vec{y}\left(t,l_{\lambda }
,m,f\right),\vec{y}\left(t,l_{\lambda }
,m,f\right)\right)\right|_a^b /\Im \lambda \le 0.
\end{equation}
But for every $\mathcal{H}^{r} $-valued $F\left(t\right)\in
L_{W\left(t,l_{\bar{\lambda }} ,m\right)}^{2}
\left(\bar{\mathcal{I}}\right)$ there exists such vector-function
$f\left(t\right)\in H$ that $x_{\lambda }
\left(a,F\right)=\vec{y}\left(a,l_{\lambda } ,m,f\right)$,
$x_{\lambda } \left(b,F\right)=\vec{y}\left(b,l_{\lambda }
m,f\right)$. So \eqref{GrindEQ__47++_} is proved in view of
\eqref{GrindEQ__815_}.

To prove that $M(\lambda)$ is  a c.o. of equation
\eqref{GrindEQ__51_} it remains to show that
$M(\bar\lambda)=M^*(\lambda)$.

Let us consider the following operator $\tilde{M}\left(\lambda
\right)\in B\left(N^{\bot } \right)$: $$\tilde{M}\left(\lambda
\right)=M\left(\lambda \right),\, \, \tilde{M}\left(\bar{\lambda
}\right)=M^{*} \left(\lambda \right),\, \, \Im \lambda >0$$

This operator is a c.o. of equation \eqref{GrindEQ__51_} in view
of \cite{Khrab5}. This c.o. generate by Theorem \ref{th4} the
operator $R\left(\lambda \right)$ \eqref{GrindEQ__63_}.

But $R\left(\lambda \right)=R_{\lambda } ,\, \Im \lambda
>0\Rightarrow R\left(\bar{\lambda }\right)=R^{*} \left(\lambda
\right)=R_{\lambda }^{*} =R_{\bar{\lambda }} ,\, \Im \lambda
>0\Rightarrow$ $\Rightarrow\, \forall f\in H$:
\begin{multline*}
\left\| \left[X_{\bar{\lambda }} \left(t\right)\right]_{\, 1}
\left(M^{*} \left(\lambda \right)-M\left(\bar{\lambda
}\right)\right)\int _{a}^{b}X_{\lambda }^{*}
\left(s\right)W\left(s,l_{\lambda } ,m\right)F\left(s,l_{\lambda }
,m\right)ds \right\|_m =0\\
\Rightarrow \forall h\in N^{\bot } :\, \, \Delta _{\bar{\lambda }}
\left(a,b\right)\left(M\left(\bar{\lambda }\right)-M^{*}
\left(\lambda \right)\right)h=0\Rightarrow M\left(\bar{\lambda
}\right)=M^{*} \left(\lambda \right).
\end{multline*}

Theorem \ref{th6} is proved.
\end{proof}

For generalized resolvents of differential operators a
representation of \eqref{GrindEQ__63_} type was obtained in
\cite{Shtraus1} for the scalar case and in \cite{Bruk1} for the
case of operator coefficients. For generalized resolvents for
\eqref{GrindEQ__1_}, \eqref{GrindEQ__2_} with $s=0$,
$n_\lambda[y]\equiv 0$ the representation of such a type was
obtained in \cite{Bruk2,Bruk3,Khrab3}.

Let $\mathcal{I}_{k} ,k=1,2$ be finite intervals, $\mathcal{I}_{1}
\subset \mathcal{I}_{2} $. Then, in spite of the fact that
$f\left(t\right)\in C^{s} \left(\bar{\mathcal{I}}_{2}
,\mathcal{H}\right)$ but $\chi _{\mathcal{I}_{1} }
f\left(t\right)\notin C^{s} \left(\bar{\mathcal{I}}_{2}
,\mathcal{H}\right)$, where $\chi _{\mathcal{I}_{1} } $ is the
characteristic function of $\mathcal{I}_{1} $, one has.

\begin{corollary}\label{cor2}
Let $0\in \mathcal{I}_{1} $ and the condition \eqref{GrindEQ__55_}
with $\mathcal{I}=\mathcal{I}_{2} $ holds. Let $R_{\lambda } $ be
the generalized resolvent of the relation $\mathcal{L}_{0} $ in
$L_{m}^{2} \left(\mathcal{I}\right)$ with
$\mathcal{I}=\mathcal{I}_{2} $. Then by Theorems \ref{th4},
\ref{th6} there exists c.o. $M\left(\lambda \right)$ of equation
\eqref{GrindEQ__5_} such that $R_{\lambda } f=y_{1}
\left(t,\lambda ,f\right)$ \eqref{GrindEQ__62_}, $t\in
\mathcal{I}=\mathcal{I}_{2} $, $f\in
H\left(=H\left(\mathcal{I}_{2} \right)\right)$. Let us define the
operator $y_{1}^{1} \left(t,\lambda ,f\right)=R_{\lambda }^{1}
f,t\in \mathcal{I}=\mathcal{I}_{1} $, $f\in
H\left(=H\left(\mathcal{I}_{1} \right)\right)$ by the same formula
\eqref{GrindEQ__62_} as operator $R_{\lambda } f$, but with
$\mathcal{I}=\mathcal{I}_{1} $ instead of
$\mathcal{I}=\mathcal{I}_{2} $. Then this operator is (after
closing) the generalized resolvent of the relation
$\mathcal{L}_{0} $ in $L_{m}^{2} \left(\mathcal{I}\right)$ with
$\mathcal{I}=\mathcal{I}_{1} $.
\end{corollary}

It is known \cite{DSnoo1} that \eqref{GrindEQ__64_} -
\eqref{GrindEQ__66_} implies \eqref{GrindEQ__3_}, where
$E_{\mu}\in B\left(L_{m}^{2} \left(\mathcal{I}\right)\right)$,
$E_{\mu }=E_{\mu -0}$,
\begin{equation} \label{GrindEQ__83_}
0\le E_{\mu _{1} } \le E_{\mu _{2} } \le \mathbf{I} ,\, \, \, \mu
_{1} <\mu _{2} ;\, \, \, E_{-\infty } =0.
\end{equation}
Here $\mathbf{I}$ is the identity operator in $L_{m}^{2}
\left(\mathcal{I}\right)$. We denote $E_{\alpha \beta }
=\frac{1}{2} \left[E_{\beta +0} +E_{\beta } -E_{\alpha +0}
-E_{\alpha } \right]$.

\begin{theorem}\label{th7}
Let $M\left(\lambda \right)$ be the characteristic operator of
equation \eqref{GrindEQ__5_} (and therefore by
\cite[p.162]{Khrab5} $\Im {P} M\left(\lambda \right){P} /\Im
\lambda \ge 0$ as $\Im \lambda \ne 0$) and $\sigma \left(\mu
\right)=w-\mathop{\lim }\limits_{\varepsilon \downarrow 0}
\frac{1}{\pi } \int _{0}^{\mu }\Im  { P} M\left(\mu +i\varepsilon
\right) {P} d\mu $ be the spectral operator-function that
corresponds to ${P} M\left(\lambda \right){P} $.

Let for equation \eqref{GrindEQ__5_} the condition
\eqref{GrindEQ__55_} with $P=I_r$ hold if $\mathcal{I}$ is
infinite. Let $E_{\mu } $ be generalized spectral family
\eqref{GrindEQ__83_} corresponding by \eqref{GrindEQ__3_} to the
resolvent $R\left(\lambda \right)$ from Theorem \ref{th4} which is
constructed with the help of c.o. $M\left(\lambda \right)$. Then
for any $\left[\alpha ,\beta \right]\subset \mathcal{B}$ the
equalities
\begin{equation} \label{GrindEQ__84_}
\begin{matrix}
{\mathop{P}\limits^{\circ } E_{\alpha ,\beta } f\left(t\right)=
\mathop{P}\limits^{\circ } \int _{\alpha }^{\beta }\left[X_{\mu }
\left(t\right)\right]_{1}  d\sigma \left(\mu \right)\varphi
\left(\mu ,\, f\right),\text{ if }f\left(t\right)\in
\mathop{H}\limits^{\circ },\ \mathcal{I}\text{ is infinite},}
\\ {E_{\alpha ,\beta } f\left(t\right)=\int _{\alpha }^{\beta }\left[X_{\mu } \left(t\right)\right]_{1} d\sigma \left(\mu \right)\varphi \left(\mu ,f\right) ,\, \, if\, \, f\left(t\right)\in H,\, \mathcal{I}\text{ is finite}} \end{matrix}
\end{equation}
are valid in $L_{m}^{2} \left(\mathcal{I}\right)$, where
$\left[X_{\lambda } \left(t\right)\right]_{1} \in
B\left(\mathcal{H}^{r} ,\, \mathcal{H}\right)$ is the first row of
the operator solution $X_{\lambda } \left(t\right)$ of homogeneous
equation \eqref{GrindEQ__54_} with coefficients
\eqref{GrindEQ__6_}, \eqref{GrindEQ__7_}, \eqref{GrindEQ__14_},
\eqref{GrindEQ__15_} that is written in the matrix form and such
that $X_{\lambda } \left(0\right)=I_{r} $,
\begin{equation} \label{GrindEQ__85_}
\varphi \left(\mu ,f\right)=\left\{\begin{array}{l} {\int
_{\mathcal{I}}\left(\left[X_{\mu } \left(t\right)\right]_{1}
\right)^{*} m\left[f\right]dt\text{ if }f\left(t\right)\in
\mathop{H}\limits^{\circ } } \\ {\int
_{\mathcal{I}}\left(\left[X_{\mu } \left(t\right)\right]_{1}
\right)^{*} W\left(t,l_{\mu } ,m\right)F\left(t,l_{\mu }
,m\right)dt,\text{ if }f(t)\in \overset{\circ}H\text{ or
}f\left(t\right)\in H,\, \mathcal{I}\text{ is finite}}
\end{array}\right.,
\end{equation}
$\mu \in \left[\alpha ,\beta \right]$.

Moreover, if vector-function $f\left(t\right)$ satisfy the
following conditions
\begin{gather}\label{star3}
\begin{matrix}
\mathop{{ P} }\limits^{\circ } E_{\infty } f=f,\, \, \mathop{{ P}
}\limits^{\circ } \int _{\mathbb{R}^{1}\setminus
\mathcal{B}}dE_{\mu } f =0\text{ if } f\in
\mathop{H}\limits^{{}^\circ },\ \mathcal{I}\text{ is infinite}
\\E_{\infty }f=f,\
\int_{\mathbb{R}^1\setminus \mathcal{B}}dE_{\mu } f =0\text{ if
}f\in H,\, \, \mathcal{I}\text{ is finite }
\end{matrix}
\end{gather}
then the inverse formulae in $L_{m}^{2} \left(\mathcal{I}\right)$
\begin{equation} \label{GrindEQ__87_}
\begin{matrix} {f\left(t\right)=\mathop{P}\limits^{\circ }
\int_{\mathcal{B}}\left[X_{\mu } \left(t\right)\right]_{1} d\sigma
\left(\mu \right)\varphi \left(\mu
,f\right)\text{ if }f\left(t\right)\in \mathop{H}\limits^{\circ } },\mathcal{I}\text{ is infinite}, \\
{f\left(t\right)=\int _{\mathcal{B}}^{}\left[X_{\mu }
\left(t\right)\right]_{1}  d\sigma \left(\mu \right)\varphi
\left(\mu ,f\right),\, \, if\, f\left(t\right)\in H,\mathcal{I}\,
\, is\, finite} \end{matrix}
\end{equation}
and Parcevel's equality
\begin{equation} \label{GrindEQ__88_}
m\left[f,g\right]=\int _{\mathcal{B}}\left(d\sigma \left(\mu
\right)\varphi \left(\mu ,f\right),\varphi \left(\mu
,g\right)\right) ,
\end{equation}
are valid, where $g\left(t\right)\in \mathop{H}\limits^{\circ } $
if $\mathcal{I}$ is infinite or $g\left(t\right)\in H$, if
$\mathcal{I}$ is finite.

In general case for $f\left(t\right),\, g\left(t\right)\in
\mathop{H}\limits^{\circ }$ if $\mathcal{I}$ is infinite or
$f\left(t\right),g\left(t\right)\in {H}$ if $\mathcal{I}$ is
finite, the inequality of Bessel type
\begin{equation} \label{GrindEQ__89_}
m\left[f\left(t\right),g\left(t\right)\right]\le \int
_{\mathcal{B}}\left(d\sigma \left(\mu \right)\varphi \left(\mu
,f\right),\, \varphi \left(\mu ,g\right)\right)
\end{equation}
is valid.
\end{theorem}

Let us notice that $\mathcal{B}=\cup_k (a_k,b_k)$, $(a_j,b_j)\cap
(a_k,b_k)=\varnothing$, $k\not= j$ since $\mathcal{B}$ is an open
set. In \eqref{GrindEQ__87_}
$\mathop{P}\limits^{\circ}\int_\mathcal{B}=\sum_k\lim\limits_{\alpha_k\downarrow
a_k,\beta_k\uparrow
b_k}\mathop{P}\limits^{\circ}\int_{\alpha_k}^{\beta_k}$.  In
\eqref{GrindEQ__87_}-\eqref{GrindEQ__89_} we understand
$\int_\mathcal{B}$ similarly.

\begin{proof}
Let for definiteness $r=s=2n$, $\mathcal{I}$ is infinite (for
another cases the proof becomes simpler). Let the vector-functions
$f\left(t\right),\, g\left(t\right)\in \mathop{H}\limits^{\circ }
,\, \lambda =\mu +i\varepsilon ,G_{\lambda } \left(t,l_{\lambda }
,m\right)$ be defined by  \eqref{GrindEQ__23_} with
$g\left(t\right)$ instead of $f\left(t\right)$. In view of the
Stieltjes inversion formula, we have
\begin{multline} \label{GrindEQ__90_}
\left(E_{\alpha ,\beta } f,g\right)_{L_{m}^{2}
\left(\mathcal{I}\right)} =\mathop{\lim }\limits_{\varepsilon
\downarrow 0} \frac{1}{2\pi i} \int _{\alpha }^{\beta
}\left(\left[y_{1} \left(t,\lambda ,f\right)-y_{1}
\left(t,\bar{\lambda },f\right)\right],g\right)_{m} d\mu  =\\
=\mathop{\lim }\limits_{\varepsilon \downarrow 0} \frac{1}{2\pi i}
\int_{\alpha }^{\beta }\bigg[\left(\vec{y}_{1} \left(t,l_{\lambda
} ,m,f\right),G\left(t,l_{\lambda }
,m\right)\right)_{L_{W\left(t,l_{\lambda } ,m\right)}^{2}
\left(\mathcal{I}\right)} -\left(\vec{y}_{1}
\left(t,l_{\bar{\lambda }} ,m,f\right),G\left(t,l_{\bar{\lambda }}
,m\right)\right)_{L_{W\left(t,l_{\bar{\lambda }} ,m\right)}^{2}
\left(\mathcal{I}\right)} +\\
+2i\int _{\mathcal{I}}\left(\left(\Im  {p}_{n}^{ -1}
\left(t,\lambda \right)\right)f^{\left[n\right]}
\left(t\left|m\right. \right),g^{\left[n\right]}
\left(t\left|m\right. \right)\right)dt\bigg]d\mu  =\\
=\mathop{\lim }\limits_{\varepsilon \downarrow 0} \frac{1}{2\pi i}
\int _{\alpha }^{\beta }\left[\left( M\left(\lambda \right) \int
_{\mathcal{I}}X_{\bar{\lambda }}^{*}
\left(t\right)W\left(t,l_{\bar{\lambda }}
,m\right)F\left(t,l_{\bar{\lambda }} ,m\right)dt, \, \int
_{\mathcal{I}}X_{\lambda }^{*} \left(t\right)W\left(t,l_{\lambda }
,m\right)G\left(t,l_{\lambda } ,m\right)dt \right)\right.  -
\\
\left. -\left( M^{*} \left(\lambda \right) \int
_{\mathcal{I}}X_{\lambda }^{*} \left(t\right)W\left(t,l_{\lambda }
,m\right)F\left(t,l_{\lambda } ,m\right)dt ,\, \int
_{\mathcal{I}}X_{\bar{\lambda }}^{*}
\left(t\right)W\left(t,l_{\bar{\lambda} }
,m\right)G\left(t,l_{\bar{\lambda }} ,m\right)dt
\right)\right]d\mu=\\
=\int _{\alpha }^{\beta }\left(d\sigma \left(\mu \right)\int
_{\mathcal{I}}X_{\mu }^{*} \left(t\right)W\left(t,l_{\mu }
\right)F\left(t,l_{\mu } ,m\right)dt ,\, \int _{\mathcal{I}}X_{\mu
}^{*} \left(t\right)W\left(t,l_{\mu } \right)G_{\mu }
\left(t,l_{\mu } ,m\right)dt \right),
\end{multline}
where the second equality is a corollary of \eqref{GrindEQ__34_},
the next to last is a corollary of \eqref{GrindEQ__62_} and the
last one follows from the well-known generalization of the
Stieltjes inversion formula \cite[Proposition (Б), p.
803]{Shtraus1}, \cite[Lemma, p. 952]{Bruk1}. (In the case of
finite $\mathcal{I}$ we have to substitute in \eqref{GrindEQ__90_}
$M(\lambda)$ by $P M(\lambda) P$ and then when passing to the next
to the last equality in \eqref{GrindEQ__90_} we have to use the
remark after the proof of Lemma \ref{lm6}.) But for $\lambda \in
\mathcal{B}$
\begin{equation} \label{GrindEQ__91_}
\int_{\mathcal{I}}X_{\bar{\lambda }}^{*}
\left(t\right)W\left(t,l_{\bar{\lambda }}
,m\right)F\left(t,l_{\bar{\lambda }} \right)dt =\int
_{\mathcal{I}}\left(\left[X_{\lambda } \left(t\right)\right]_{1}
\right)^{*} m\left[f\right]dt ,
\end{equation}
because, in view of \eqref{GrindEQ__29_},
\begin{multline*}
\forall h\in \mathcal{H}^{r} :(\int _{\mathcal{I}}X_{\bar{\lambda
}}^{*} \left(t\right)W\left(t,l_{\bar{\lambda }}
,m\right)F\left(t,l_{\bar{\lambda }} \right)dt,h )=\\
=\int _{\mathcal{I}}\left(W\left(t,l_{\bar{\lambda }}
,m\right)F\left(t,l_{\bar{\lambda }} \right),X_{\bar{\lambda }}
\left(t\right)h\right)dt =(\int
_{\mathcal{I}}\left(\left[X_{\bar{\lambda }} \right]_{1}
\right)^{*} m\left[f\right],h)dt.
\end{multline*}
Due to \eqref{GrindEQ__90_}, \eqref{GrindEQ__91_},
\eqref{GrindEQ__85_}

\begin{equation} \label{GrindEQ__92_}
\left(E_{\alpha ,\beta } f,g\right)_{L_{m}^{2}
\left(\mathcal{I}\right)} =\int _{\alpha }^{\beta }\left(d\sigma
\left(\mu \right)\varphi \left(\mu ,f\right),\varphi \left(\mu
,g\right)\right) .
\end{equation}

The equality \eqref{GrindEQ__88_} and inequality
\eqref{GrindEQ__89_} are the corollaries of \eqref{GrindEQ__92_}.

Representing $\varphi \left(\mu ,g\right)$ in \eqref{GrindEQ__92_}
by the second variant of \eqref{GrindEQ__85_}, changing in
\eqref{GrindEQ__92_} the order of integration and replacing
$\int_{\alpha}^\beta$ by integral sum and using
\eqref{GrindEQ__29_} we obtain that
\begin{multline}
\label{GrindEQ__93_}\left(E_{\alpha ,\beta } f,g\right)_{L_{m}^{2}
\left(\mathcal{I}\right)} =(\int _{\alpha }^{\beta }\left[X_{\mu }
\left(t\right)\right]_{1} d\sigma \left(\mu \right)\varphi
\left(\mu ,f\right),g\left(t\right) )_{L_{m}^{2}
\left(\mathcal{I}\right)} =\\
=\int _{\alpha }^{\beta }\left[X_{\mu } \left(t\right)\right]_{1}
d\sigma \left(\mu \right)\varphi \left(\mu
,f\right),g\left(t\right) )_{L_{m}^{2} \left(\mathcal{I}\right)}
\end{multline}
and \eqref{GrindEQ__84_} is proved since
$g(t)\in\overset{\circ}{H}$. Equalities \eqref{GrindEQ__87_} are
the corollary of \eqref{GrindEQ__84_}, \eqref{star3}. Theorem
\ref{th7} is proved.
\end{proof}

Let us notice that if
$L_m^2(\mathcal{I})=\overset{\circ}L_m^2(\mathcal{I})$ then
Theorem \ref{th7} is valid without condition \eqref{GrindEQ__55_}
with $P=I_r$ if $\mathcal{I}$ is infinite.

It is known (see for example \cite{Khrab1} or Example \ref{ex2})
that just in the case $n_{\lambda } \left[y\right]\equiv 0$ in
\eqref{GrindEQ__1_}, \eqref{GrindEQ__2_} there is such $E_{\mu} $
satisfying \eqref{GrindEQ__3_}, \eqref{GrindEQ__62_} --
\eqref{GrindEQ__66_}, \eqref{GrindEQ__83_} that ${E} _{\infty }
\ne \mathbf{I}$.

On the other hand if $n_\lambda[y]\equiv 0$ then $\forall f\in
D\left(\mathcal{L}_{0} \right){E} _{\infty } f=f$ in view of
\cite{Khrab1}, \cite{Khrab3}.

Let expression $n_{\lambda } $ in representation
\eqref{GrindEQ__2_}, \eqref{GrindEQ__561_} have a divergent form
with coefficients $\tilde{\tilde{p}}_{j} =\tilde{\tilde{p}}_{j}
\left(t,\lambda \right),\tilde{\tilde{q}}_{j}
=\tilde{\tilde{q}}_{j} \left(t,\lambda
\right),\tilde{\tilde{s}}_{j} =\tilde{\tilde{s}}_{j}
\left(t,\lambda \right)$.

We denote $m\left(t\right)$ three-diagonal $\left(n+1\right)\times
\left(n+1\right)$ operator matrix, whose elements under main
diagonal are equal to $\left(-\frac{i}{2} \tilde{q}_{1} ,\, \ldots
,\, -\frac{i}{2} \tilde{q}_{n} \right)$, the elements over the
main diagonal are equal to $\left(\frac{i}{2} \tilde{s}_{1} ,\,
\ldots ,\, \frac{i}{2} \tilde{s}_{n} \right)$, the elements on the
main diagonal are equal to $\left(\tilde{p}_{0} ,\, \ldots ,\,
\tilde{p}_{n} \right)$, where $\tilde{p}_{j} ,\tilde{q}_{j}
,\tilde{s}_{j} =\tilde{q}_{j}^{*} $ are the coefficients of
expressions $m$. (Here either $2n$ or $2n+1$ is equal to the order
$r$ of $l_\lambda$). If order of $n_{\lambda } $ is less or equal
to $2n$, we denote $n\left(t,\lambda \right)$ the analogues
$\left(n+1\right)\times \left(n+1\right)$ operator matrix with
$\tilde{\tilde{p}}_{j} ,\tilde{\tilde{q}}_{j}
,\tilde{\tilde{s}}_{j} $ instead of $\tilde{p}_{j} ,\tilde{q}_{j}
,\tilde{s}_{j} $. If order $m$ or order $n_{\lambda } $ is less
than $2n$, we set the correspondent elements of $m\left(t\right)$
or $n\left(t,\lambda \right)$ be equal to zero.

\begin{theorem}\label{th8}
Let in \eqref{GrindEQ__1_}, \eqref{GrindEQ__2_} the order of the
expression $n_{\lambda } $ is less or equal to the order of the
expression $l-\lambda m$ (and therefore in view of
(\ref{GrindEQ__561_}) the order of $l-\lambda m$ is equal to $r$;
so $Q\left(t,l_{\lambda } \right)=Q\left(t,l-\lambda m\right)$).
Let $y=R_{\lambda } f,\, f\in H$ be the generalized resolvent of
the relation $\mathcal{L}_{0} $ and $y$ satisfy equation
\eqref{GrindEQ__1_}. Let $y_{1} =R\left(\lambda \right)f,\, f\in
H$ be the operator \eqref{GrindEQ__62_}, \eqref{GrindEQ__63_} from
Theorem \ref{th4}.

Let the following conditions hold for $\tau >0$ large enough:

1$^\circ$.

{\small
\begin{multline} \label{GrindEQ__94_}
\left.\mathop{\lim }\limits_{\alpha \downarrow a,\, \beta \uparrow
b} \frac{(\Re Q\left(t,l_{\lambda } \right)\left(\vec{y}_{1}
\left(t,l_{\lambda } ,m,f\right)-\vec{y}\left(t,l-\lambda
m,m,f\right)),\left(\vec{y}_{1} \left(t,l_{\lambda }
,m,f\right)\right)-\vec{y}\left(t,l-\lambda
m,m,f\right)\right)}{\Im \lambda }\right|_\alpha^\beta\le\\\le 0,\
\lambda =i\tau
\end{multline}
}

2$^\circ$.
\begin{equation} \label{GrindEQ__95_}
\Im{n}\left({\rm t,}\lambda \right)\le c\left(t,\tau
\right)m\left(t\right),\, t\in \bar{\mathcal{I}},\, \, \lambda
=i\tau ,
\end{equation}
where the scalar function $c\left(t,\tau \right)$ satisfles the
following condition:
\begin{equation} \label{GrindEQ__96_}
\mathop{\sup }\limits_{t\in \bar{\mathcal{I}}} c\left(t,\tau
\right)=o\left(\tau \right),\, \, \, \tau \to +\infty .
\end{equation}

Then for generalized spectral family ${E}_{\mu } $
\eqref{GrindEQ__83_} corresponding by \eqref{GrindEQ__3_} to the
resolvent $R\left(\lambda \right)$
\eqref{GrindEQ__62_}-\eqref{GrindEQ__63_} from Theorem \ref{th4}
and for generalized spectral family $\mathcal{E}_{\mu } $
corresponding to the generalized resolvent $R_{\lambda } $ one has
${E} _{\infty } =\mathcal{E} _{\infty } $.

\end{theorem}

Let us notice that in view of \eqref{GrindEQ__95_} the coefficient
at the highest derivative in the expression $l-\lambda m$ has
inverse from $B\left(\mathcal{H}\right)$ if $t\in
\bar{\mathcal{I}}$, $\Im \lambda \ne 0$.

\begin{proof}
Let $f\left(t\right)\in H$, $y_{1} =R\left(\lambda \right)f$,
$y=R_{\lambda } f$. Then $z=y_{1} -y$ satisfies the following
equation
\begin{equation} \label{GrindEQ__97_}
l\left[z\right]-\lambda m\left[z\right]=n_{\lambda } \left[y_{1} \right].
\end{equation}

Applying to the equation \eqref{GrindEQ__97_} the Green formula
\eqref{GrindEQ__36_}, one has
$$
\left.\int_{\alpha }^{\beta }\Im \left(n_{\lambda } \left\{y_{1}
,z\right\}\right)dt +\int _{\alpha }^{\beta }m\left\{z,z\right\}dt
=\frac{1}{2} \frac{\Re \left(Q\left(t,l_{\lambda }
\right)\vec{z},\vec{z}\right)}{\Im \lambda } \right|_\alpha^\beta
,
$$
where $\vec{z}=\vec{z}\left(t,l-\lambda m,n_{\lambda } ,y_{1}
\right)=\vec{y}_1\left(t,l_{\lambda }
,m,f\right)-\vec{y}\left(t,l-\lambda m,m,f\right)$ in view of
\eqref{GrindEQ__25_} and Lemma \ref{lm2}. Hence for $\tau >0$
large enough
\begin{multline} \label{GrindEQ__98_}
m\left[z,z\right]\le -\int_{\mathcal{I}}\Im\left(n_{\lambda }
\left[y_{1} ,z\right]\right)dt/\tau
 \leq\\
 \leq{\int _{\mathcal{I}}\left|\left({ n}\left({\rm t,}\lambda
\right)col\left\{y_{1} ,y'_{1} ,\ldots ,y_{1}^{\left(n\right)}
\right\},col\left\{z,z',\ldots ,z^{\left(n\right)}
\right\}\right)\right| dt \mathord{\left/{\vphantom{\int
_{\mathcal{I}}\left|\left({\rm n}\left({\rm t,}\lambda
\right)col\left\{y_{1} ,y'_{1} ,\ldots ,y_{1}^{\left(n\right)}
\right\},col\left\{z,z',\ldots ,z^{\left(n\right)}
\right\}\right)\right| dt \tau }}\right.\kern-\nulldelimiterspace}
\tau } ,\quad \lambda =i\tau
\end{multline}
in view of \eqref{GrindEQ__94_}. But due to the inequality of the
Cauchy type for dissipative operators \cite[p. 199]{Shtraus1} and
\eqref{GrindEQ__95_}, \eqref{GrindEQ__96_}: subintegral function
in \eqref{GrindEQ__98_} is less or equal to
$\left(m\left\{z,z\right\}\right)^{1/2}
\left(m\left\{y_1,y_1\right\}\right)^{1/2} o\left(1\right)$ with
$\lambda =i\tau ,\, \tau \to +\infty $. Therefore $\left\|
z\right\|_m \le o\left({1/\tau}\right)\left\| f\right\| _m $ since
$\left\| R_{\lambda } \right\| \le {1 \mathord{\left/{\vphantom{1
\left|\Im \lambda \right|}}\right.\kern-\nulldelimiterspace}
\left|\Im \lambda \right|} $. Hence
$$\left\| R\left(\lambda
\right)-R_{\lambda } \right\|\le o\left({1
\mathord{\left/{\vphantom{1 \tau
}}\right.\kern-\nulldelimiterspace} \tau } \right),\quad \lambda
=i\tau ,\quad \tau \to +\infty$$ To complete the proof of the
theorem it remains to prove the following
\begin{lemma}\label{lm10+}
Let $R_{k} \left(\lambda \right)=\int
_{\mathbb{R}^1}\frac{dE^k_{\mu } }{\mu -\lambda }  ,\, k=1,2$,
where $E_{\mu }^{k} $ are the generalized spectral families the
type \eqref{GrindEQ__83_} in Hilbert space $\mathbf{H}$. If
$\left\| R_{1} \left(\lambda \right)-R_{2} \left(\lambda
\right)\right\| \le o\left({1 \mathord{\left/{\vphantom{1 \tau
}}\right.\kern-\nulldelimiterspace} \tau } \right),\, \, \,
\lambda =i\tau ,\, \, \tau \to +\infty $, then $E_{\infty }^{1}
=E_{\infty }^{2} $.
\end{lemma}

\begin{proof}
Let $f\in $\textbf{H} $\, \sigma \left(\mu
\right)=\left(\left(E_{\mu }^{1} -E_{\mu }^{2} \right)f,f\right)$.
One has
\begin{multline*}
\left|\left(\left[R_{1} \left(\lambda \right)-R_{2} \left(\lambda
\right)\right]f,f\right)\right|=\\=\frac{1}{\tau } \left|-\int
_{\Delta }d\sigma \left(\mu \right) +\int _{\Delta }\frac{\mu
d\sigma \left(\mu \right)}{\mu -\lambda }  +\int _{R^{1}
\backslash \Delta }\frac{\lambda d\sigma \left(\mu \right)}{\mu
-\lambda }  \right|\le o\left({1 \mathord{\left/{\vphantom{1 \tau
}}\right.\kern-\nulldelimiterspace} \tau } \right)\left\|
f\right\| ^{2} ,\, \, \lambda =i\tau ,\, \, \tau \to +\infty
\end{multline*}
Therefore
\begin{equation} \label{GrindEQ__99_}
\left|-\int _{\Delta }d\sigma \left(\mu \right)+\int _{\Delta
}\frac{\mu d\sigma \left(\mu \right)}{\mu -\lambda }  +\int
_{R^{1} \backslash \Delta }\frac{\lambda d\sigma \left(\mu
\right)}{\mu -\lambda }   \right|\le o\left(1\right),\, \, \,
\lambda =i\tau ,\, \, \tau \to +\infty .
\end{equation}
For an arbitrarily small $\varepsilon >0$ we choose such finite
interval $\Delta \left(\varepsilon \right)$ that for any finite
interval $\Delta \supseteq \Delta \left(\varepsilon
\right):\left|\int _{R^{1} \backslash \Delta }\frac{\lambda
d\sigma \left(\mu \right)}{\mu -\lambda }
\right|<\frac{\varepsilon }{2} ,\, \, \lambda =i\tau $. But for
any finite interval $\Delta \supseteq \Delta \left(\varepsilon
\right)\, \exists N=N\left(\Delta \right):\forall \tau
>N:\left|\int _{\Delta }\frac{\mu d\sigma \left(\mu \right)}{\mu
-\lambda }  \right|<\frac{\varepsilon }{2} ,\, \, \lambda =i\tau
$. Therefore $\forall \varepsilon >0,\Delta \supseteq \Delta
\left(\varepsilon \right):\left|\int _{\Delta }d\sigma \left(\mu
\right) \right|<\varepsilon $ in view of \eqref{GrindEQ__99_}.
Hence $\forall f\in $\textbf{H}$:\left(E_{\infty }^{1}
f,f\right)=\left(E_{\infty }^{2} f,f\right)\Rightarrow E_{\infty
}^{1} =E_{\infty }^{2} $.

Lemma \ref{lm10} and Theorem \ref{th8} are proved.
\end{proof}
\end{proof}

\begin{corollary}\label{cor3}
Let the conditions of Theorems \ref{th7}, \ref{th8} hold. Then for
generalized spectral family $E_{\mu } $ from Theorem \ref{th7}
$\forall f\left(t\right)\in D\left(\mathcal{L}_{0}
\right):E_{\infty } f=f$.
\end{corollary}

\begin{remark}
\label{rem3} If $L_{m}^{2}
\left(\mathcal{I}\right)=\mathop{L_{m}^{2} }\limits^{\circ }
\left(\mathcal{I}\right)$, then it is sufficient to verify
condition \eqref{GrindEQ__94_} in Theorem \ref{th8} for $f\in
\mathop{H}\limits^{\circ } $.
\end{remark}

\begin{proposition}
\label{prop2} Let the order of expression $n_{\lambda } $ be less
or equal to the order of expression $l-\lambda m$ and the
coefficient of $l-\lambda m$ at the highest derivative has the
inverse from $B\left(\mathcal{H}\right)$ for $t\in
\bar{\mathcal{I}},\, \lambda \in \mathcal{B} \left(l-\lambda
{m}\right)$, where $\mathcal{B}\left(l-\lambda m\right)$ is an
analogue of the set $\mathcal{B} =\mathcal{B}\left(l_\lambda
\right)$. Let interval $\mathcal{I}$ be finite and for equation
\eqref{GrindEQ__1_}, \eqref{GrindEQ__2_} with $n_{\lambda }
\left[y\right]\equiv 0$ condition \eqref{GrindEQ__55_} holds with
${P} ={I} _{r}$. Then for equation \eqref{GrindEQ__1_},
\eqref{GrindEQ__2_} this condition also holds with $P=I_r$ and
resolvents $y=R_{\lambda } f,\, y_{1} =R\left(\lambda \right)f$
from Theorem \ref{th8} satisfies condition \eqref{GrindEQ__94_}
for $\Im \lambda \ne 0$ if they are the solutions of boundary
value problems for equations \eqref{GrindEQ__1_},
\eqref{GrindEQ__2_}, ($n_{\lambda } \left[y\right]\equiv 0$) and
\eqref{GrindEQ__1_}, \eqref{GrindEQ__2_} with boundary conditions
from Theorem \ref{th5} with the same operators
$\mathcal{M}_{\lambda } ,\, \mathcal{N}_{\lambda } $.
\end{proposition}

\begin{proof}

In view of Theorem \ref{th5} it is sufficient to prove only
proposition about condition \eqref{GrindEQ__55_}.

Let for definiteness order $l=$ order $m=$ order $n_{\lambda }
=2n$.

Let for equation \eqref{GrindEQ__1_}, \eqref{GrindEQ__2_},
($n_{\lambda } \left[y\right]\equiv 0$) condition
\eqref{GrindEQ__55_} with ${P} ={ I}_{r}$ hold, but for equation
\eqref{GrindEQ__1_}, \eqref{GrindEQ__2_} that is not true. Then in
view of \cite{Khrab5} the solutions $y_k\left(t\right)$ of
equation \eqref{GrindEQ__1_}, \eqref{GrindEQ__2_} with $\lambda
=i$ exist for which
\begin{equation} \label{GrindEQ__100_}
\int _{\alpha }^{\beta }\left(m+\Im n_{i} \right)\left\{y_{k}
,y_{k} \right\} dt\to 0,\quad \vec{y}_{k} \left(0,l_{i}
,m,0\right)=f_{k} ,\; \left\| f_{k} \right\| =1,
\end{equation}
where $i\Im n_i =n_{i} $ in view of \eqref{GrindEQ__561_}. Hence
in view of \eqref{GrindEQ__29_}
\begin{equation} \label{GrindEQ__101_}
\int _{\alpha }^{\beta }\left(W_{i} \left(t,l+im,n_{i} \right)
Y_{k} \left(t,l+im,n_i\right),Y_{k}\left(t,l+im,n_i\right)
\right)dt = \int _{\alpha }^{\beta }n_{i} \left\{y_{k} ,y_{k}
\right\}dt \to 0.
\end{equation}

On the other hand
\begin{equation} \label{GrindEQ__102_}
X_{i} \left(t\right)f_{k} =\tilde X_{i} \left(t\right)f_{k}
+\tilde X_{i} \left(t\right)\int _{o}^{t}\tilde X_{i}^{-1}
\left(s\right)J^{-1} W\left(s,l+im,n_{i} \right)Y_{k}
\left(s,l+im,n_i\right) ds .
\end{equation}
in view of Theorem \ref{th1} and the fact that
$\vec{y}_k(t,l-im,n_i,y_k)=\vec{y}_k(t,l_i,m,0)$, where $\tilde
X_{\lambda } \left(t\right)$ is an analogue of $X_{\lambda }
\left(t\right)$ for the case $n_{\lambda } \left[y\right]\equiv
0$.

Comparing \eqref{GrindEQ__101_}, \eqref{GrindEQ__102_} we see that
\begin{equation} \label{GrindEQ__103_}
\left\|X_{i} \left(t\right)f_{k} -\tilde X_{i} \left(t\right)f_{k}
\right\| \to 0
\end{equation}
uniformly in $t\in \left[\alpha ,\beta \right]$.

In view of \eqref{GrindEQ__100_} subsequence $y_{k_{q} } $ exist
such that
\begin{equation} \label{GrindEQ__104_}
m\left\{y_{k_{q} } ,y_{k_{q} } \right\}\mathop{\to }\limits^{a.a.} 0,\quad n_{i} \left\{y_{k_{q} } ,y_{k_{q} } \right\}\mathop{\to }\limits^{a.a.} 0.
\end{equation}
Due to second proposition \eqref{GrindEQ__104_} and the arguments
as in the proof of Proposition \ref{prop1} one has
\begin{equation} \label{GrindEQ__105_}
y_{k_{q} }^{\left[j\right]} \left(t\left|n_{i} \right.
\right)\mathop{\to }\limits^{a.a.} 0\quad j=n,\ldots ,2n.
\end{equation}

Let us denote $\tilde y_{k_{q} } \left(t\right)=\tilde X_{i}
\left(t\right)f_{k_{q} } $. In view of Theorem \ref{th1} and
\eqref{GrindEQ__103_}
\begin{equation} \label{GrindEQ__106_}
\left\| y_{k_{q} }^{\left(j\right)} \left(t\right)-\tilde y_{k_{q}
}^{\left(j\right)} \left(t\right)\right\| \to 0,\quad j=1,\ldots
,n-1,
\end{equation}
\begin{multline} \label{GrindEQ__107_}
\left\| \left(p_{n} \left(t\right)-i\tilde{p}_{n}
\left(t\right)\right)\left[y_{k_{q} }^{\left(n\right)}
\left(t\right)-\tilde y_{k_{q} }^{\left(n\right)}
\left(t\right)\right]-\right.\\\left.-{i\over 2}(q_n(t)-i\tilde
q_n(t))\left[y_{k_q}^{(n-1)}(t)-\tilde y_{k_q}^{(n-1)}(t)\right]
-y_{k_{q} }^{\left[n\right]} \left(t\left|n_{i} \right.
\right)\right\| =\\=\left\| y_{k_{q} }^{\left[n\right]}
\left(t|l_i\right)-\tilde y_{k_{q} }^{\left[n\right]}
\left(t|l-im\right)\right\| \to 0
\end{multline}
uniformly in $t\in \left[\alpha ,\beta \right]$. Comparing
\eqref{GrindEQ__104_}, \eqref{GrindEQ__105_},
\eqref{GrindEQ__107_} and using $\left(p_{n}
\left(t\right)-i\tilde{p}_{n} \left(t\right)\right)^{-1} \in
B\left(\mathcal{H}\right)$ we have
\begin{equation} \label{GrindEQ__108_}
\left(\tilde{p}_{n} \left(t\right)\tilde y_{k_{q}
}^{\left(n\right)} \left(t\right),\tilde y_{k_{q}
}^{\left(n\right)} \left(t\right)\right)\mathop{\to
}\limits^{a.a.} 0.
\end{equation}
In view of \eqref{GrindEQ__106_}, \eqref{GrindEQ__108_},
\eqref{GrindEQ__104_}
$$
m\left\{\tilde y_{k_{q} } ,\tilde y_{k_{q} } \right\}\mathop{\to
}\limits^{a.a.} 0,\quad \vec{y}_{k_q}(0,l-im, m,0)=f_{k_q},
$$
that contradicts to the condition \eqref{GrindEQ__55_} with
$P=I_{r} $ for equation \eqref{GrindEQ__1_}, \eqref{GrindEQ__2_}
with $n_{\lambda } \left[y\right]\equiv 0$. Proposition
\ref{prop2} is proved.
\end{proof}

In the next theorem $\mathcal{I}=\mathbb{R}^1$ and condition
\eqref{GrindEQ__55_} hold with $P=I_{r} $ both on the negative
semi-axis $R_{-} $ (i.e. as $\mathcal{I}=R_{-} $) and on the
positive semi-axis $R_{+} $ (i.e. as $\mathcal{I}=R_{+} $).

\begin{theorem}\label{th9}
Let $\mathcal{I}=\mathbb{R}^1$, the coefficient of the expression
$l_\lambda$ \eqref{GrindEQ__2_} be periodic on each of the
semi-axes $R_{-} $ and $R_{+} $ with periods $T_{-} >0$ and $T_{+}
>0$ correspondingly. Then the spectrums of the monodromy operators
$X_{\lambda } \left(\pm T_{\pm } \right)$ ($X_{\lambda }
\left(t\right)$ is from Theorem \ref{th4}) do not intersect the
unit circle as $\Im  \lambda \ne 0$, the c.o. $M\left(\lambda
\right)$ of the equation \eqref{GrindEQ__5_} is unique and equal
to
\begin{equation} \label{GrindEQ__109_}
M\left(\lambda \right)=\left(\mathcal{P}\left(\lambda
\right)-\frac{1}{2} I_{r} \right)\, \left(iG\right)^{-1} \quad
\left(\Im  \lambda \ne 0\right),
\end{equation}
where the projection $\mathcal{P}\left(\lambda \right)=P_{+}
\left(\lambda \right)\left(P_{+} \left(\lambda \right)+P_{-}
\left(\lambda \right)\right)^{-1} ,\, P_{\pm } \left(\lambda
\right)$ are Riesz projections of the monodromy operators
$X_{\lambda } \left(\pm T_{\pm } \right)$ that correspond to their
spectrums lying inside the unit circle, $\left(P_{+} \left(\lambda
\right)+P_{-} \left(\lambda \right)\right)^{-1} \in
B\left(\mathcal{H}^{r} \right)$ as $\Im  \lambda \ne 0$.

Also let $\dim \mathcal{H}<\infty $, a finite interval $\Delta
\subseteq \mathcal{B} $. Then in Theorem \ref{th7} $d\sigma
\left(\mu \right)=d\sigma _{ac} \left(\mu \right)+d\sigma _{d}
\left(\mu \right),\mu \in \Delta $. Here $\sigma _{ac} \left(\mu
\right)\in AC\left(\Delta \right)$ and, for $\mu \in \Delta $,
$$\sigma '_{ac} \left(\mu \right)=\frac{1}{2\pi } G^{-1}
\left(Q_{-}^{*} \left(\mu \right)GQ_{-} \left(\mu
\right)-Q_{+}^{*} \left(\mu \right)GQ_{+} \left(\mu
\right)\right)G^{-1}, $$ where the projections $Q_{\pm } \left(\mu
\right)=q_{\pm } \left(\mu \right)\left(P_{+} \left(\mu
\right)+P_{-} \left(\mu \right)\right)^{-1} ,\, q_{\pm } \left(\mu
\right)$ are Riesz projections of the monodromy matrixes $X_{\mu }
\left(\pm T_{\pm } \right)$ corresponding to the multiplicators
belonging to the unit circle and such that they are shifted inside
the unit circle as $\mu $ is shifted to the upper half plane,
$P_{\pm } \left(\mu \right)=P_{\pm } \left(\mu +i0\right);\,
\sigma _{d} \left(\mu \right)$ is a step function.
\end{theorem}

Let us notice that the sets on which $q_{\pm } \left(\mu
\right),P_{\pm } \left(\mu \right),\left(P_{+} \left(\mu
\right)+P_{-} \left(\mu \right)\right)^{-1} $ are not infinitely
differentiable do not have finite limit points $\in \mathcal{B} $
as well as the set of points of increase of $\sigma _{d} \left(\mu
\right)$.

\begin{proof}
The proof of Theorem \ref{th9} is similar to that on in the case
$n_{\lambda } \left[y\right]\equiv 0$ \cite{Khrab6}.
\end{proof}

The following examples demonstrate effects that are the results of
appearance in $l_{\lambda } $ \eqref{GrindEQ__2_} of perturbation
$n_{\lambda } $ depending nonlinearly on $\lambda $.

In Examples \ref{ex3}, \ref{ex4} nonlinear in $\lambda $
perturbation does not change the type of the spectrum. In this
examples $\dim \mathcal{H}=1,\, \, m\left[y\right]=-y''+y$.
$L_{m}^{2} \left(\mathcal{I}\right)=\mathop{L_{m}^{2}
}\limits^{\circ } \left(\mathcal{I}\right)=W_{2}^{1,2}
\left(\mathbb{R}^{1} \right)$. In Example \ref{ex5} such
perturbation implies an appearance of spectral gap with
"eigenvalue" in this gap.

\begin{example}\label{ex3}
Let
$$
l_{\lambda } \left[y\right]=iy'-\lambda
\left(-y''+y\right)-\left(-\frac{h}{\lambda } y\right),\, \, \,
h\ge 0.$$ Here $\mathcal{B} =\mathbb{C}\setminus 0,\, \, E_{0}
=E_{+0} $, spectral matrix $\sigma \left(\mu \right)\in AC_{loc}
$,
\\
$\sigma '\left(\mu \right)=\frac{1}{2\pi } \left(\begin{array}{cc}
{\frac{2}{\sqrt{4h+1-4\mu ^{2} } } } & {0} \\ {0} & {\frac{1}{2}
\sqrt{4h+1-4\mu ^{2} } } \end{array}\right)$, as $\left|\mu
\right|<\sqrt{h+{1 \mathord{\left/{\vphantom{1
4}}\right.\kern-\nulldelimiterspace} 4} } $,
\\
$\sigma '\left(\mu \right)=0$, as $\left|\mu \right|>\sqrt{h+{1
\mathord{\left/{\vphantom{1 4}}\right.\kern-\nulldelimiterspace}
4} } $.

In Example \ref{ex3} nonlinear in $\lambda $ perturbation change
edges of spectral band.
\end{example}

\begin{example}\label{ex4}
Let
$$
l_{\lambda } \left[y\right]=y^{\left(IV\right)} -\lambda
\left(-y''+y\right)-\left(-\frac{h}{\lambda } y\right),\, \, \,
h\ge 0.$$ Here $\mathcal{B} =\begin{cases}
\mathbb{C}\setminus\{0\},&h\ne 0 \\ \mathbb{C},&h=0
\end{cases},\, \, \, E_{0} =E_{+0} $, spectral matrix
$\sigma \left(\mu \right)\in AC_{loc} $,
\begin{gather*}
\sigma '\left(\mu \right)=\begin{cases} \frac{1}{2\pi }
\sqrt{\frac{\lambda +\sqrt{D} }{D} } \left(\begin{array}{cccc}
{\frac{2}{\lambda +\sqrt{D} } } & {0} & {0} & {-1} \\ {0} & {1} &
{\frac{-\lambda +\sqrt{D} }{2} } & {0} \\ {0} & {\frac{-\lambda
+\sqrt{D} }{2} } & {\left(\frac{\lambda -\sqrt{D} }{2} \right)^{2}
} & {0} \\ {-1} & {0} & {0} & {\frac{\lambda +\sqrt{D} }{2} }
\end{array}\right),&\text{ as }-\sqrt{h} <\mu <0,\, \mu
>\sqrt{h}  \\ \frac{1}{2\pi } \cdot \frac{1}{\sqrt{\lambda
-2\sqrt{q} } } \left(\begin{array}{cccc} {\frac{1}{\sqrt{q} } } &
{0} & {0} & {-1} \\ {0} & {1} & {-\sqrt{q} } & {0} \\ {0} &
{-\sqrt{q} } & {\sqrt{q} \left(\lambda -\sqrt{q} \right)} & {0} \\
{-1} & {0} & {0} & {\lambda -\sqrt{q} }
\end{array}\right),&\text{ as }\mu ^{*} <\mu <\sqrt{h},
\end{cases},
\end{gather*}
where $D=\lambda ^{2} -4q,\, q={h \mathord{\left/{\vphantom{h
\lambda }}\right.\kern-\nulldelimiterspace} \lambda } -\lambda ,\,
\mu ^{*} =\mu ^{*} \left(h\right)$ - nonnegative root of equation
$D=0$. $\sigma '\left(\mu \right)=0$, as $\mu \notin
\left[-\sqrt{h} ,0\right]\bigcup \left[\mu ^{*} ,\infty \right)$.

In Example \ref{ex4} nonlinear in $\lambda $ perturbation implies
an appearance of additional spectral band $\left[-h,0\right]$,
variation of edge of semi in finite spectral band and appearance
of interval $\left(\mu ^{*} ,\sqrt{h} \right)$ of fourfold
spectrum.
\end{example}

\begin{example}\label{ex5}
Let $\dim \mathcal{H}=2$,
$$l_{\lambda } \left[y\right]=\left(\begin{array}{cc} {0} & {-1} \\
{1} & {0}
\end{array}\right)y'-\lambda y-\left(\begin{array}{cc} {-{h
\mathord{\left/{\vphantom{h \lambda
}}\right.\kern-\nulldelimiterspace} \lambda } } & {0} \\ {0} & {0}
\end{array}\right)y,\quad h\ge 0.$$
Here $\mathcal{B} =\begin{cases} \mathbb{C}\setminus\{0\},& h\ne 0 \\
\mathbb{C},& h=0 \end{cases},$ spectral matrix $\sigma \left(\mu
\right)=\sigma _{ac} \left(\mu \right)+\sigma _{d} \left(\mu
\right)$, $\sigma_{ac}(\mu)\in AC_{loc},  \sigma '_{ac} \left(\mu
\right)\ne 0$, as $\left|\mu \right|>\sqrt{h} ,\, \, \sigma '_{ac}
\left(\mu \right)=0$, as $\left|\mu \right|<\sqrt{h} $,
step-function $\sigma _{d} \left(\mu \right)$ has only one jump
$\left(\begin{array}{cc} {0} & {0} \\ {0} & {{\sqrt{h}
\mathord{\left/{\vphantom{\sqrt{h}
2}}\right.\kern-\nulldelimiterspace} 2} } \end{array}\right)$ in
point $\mu =0$ (inside of spectral gap). In this point
$$\left(E_{+0} -E_{0} \right)f=\left(\begin{array}{cc} {0} & {0} \\
{0} & {{\sqrt{h}  \mathord{\left/{\vphantom{\sqrt{h}
2}}\right.\kern-\nulldelimiterspace} 2} } \end{array}\right)\int
_{-\infty }^{\infty }e^{-\sqrt{h} \left|t-s\right|}
f\left(s\right) ds,\, \, \, f\left(t\right)\in L^{2}
\left(\mathbb{R}^{1} \right).$$
\end{example}

Let us notice that in view of Floquet theorem conditions of
Theorem \ref{th8} ((\ref{GrindEQ__94_}) with account of Remark
\ref{rem3}) hold for all Examples \ref{ex3}-\ref{ex5}.

\section*{Acknowledgments} The author is grateful to professor
F.S. Rofe-Beketov for his great attention to this work.

\end{document}